\documentclass[10pt, reqno, a4paper]{amsart}

\usepackage{amssymb}
\usepackage{amsthm}
\usepackage{indentfirst}
\usepackage{mathtools} 
\usepackage[colorlinks=true]{hyperref}  
\usepackage{enumitem} 
\usepackage[font=small]{caption}
\usepackage{tikz} 
\usepackage{graphicx,overpic}
\usepackage[dvips]{psfrag}
\usepackage{xcolor} 

\usepackage{etoolbox}
\patchcmd{\subsection}{-.5em}{.5em}{}{}
\patchcmd{\subsubsection}{\itshape}{\bfseries}{}{}

\makeatletter

\def\myMRbibitem{\@ifnextchar[\my@lbibitem\my@bibitem}

\def\mybiblabel#1#2{\@biblabel{{\hyperref{http://www.ams.org/mathscinet-getitem?mr=#1}{}{}{#2}}}}

\def\myhyperanchor#1{\Hy@raisedlink{\hyper@anchorstart{cite.#1}\hyper@anchorend}}

\def\my@lbibitem[#1]#2#3#4\par{%
    \item[\mybiblabel{#2}{#1}\myhyperanchor{#3}\hfill]#4%
    \@ifundefined{ifbackrefparscan}{}{\BR@backref{#3}}%
    \if@filesw{\let\protect\noexpand\immediate
       \write\@auxout{\string\bibcite{#3}{#1}}}\fi\ignorespaces%
}

\def\my@bibitem#1#2#3\par{%
    \refstepcounter\@listctr
    \item[\mybiblabel{#1}{\the\value\@listctr}\myhyperanchor{#2}\hfill]#3%
    \@ifundefined{ifbackrefparscan}{}{\BR@backref{#2}}%
    \if@filesw\immediate\write\@auxout
        {\string\bibcite{#2}{\the\value\@listctr}}\fi\ignorespaces%
}

\makeatother

\theoremstyle{plain}
	\newtheorem{theo}{Theorem}
	\newtheorem{corom}{Corollary}
	\newtheorem{theor}{Theorem}[section]
	\newtheorem{lemm}{Lemma}[section]
	\newtheorem{claim}[lemm]{Claim}
	
	\newtheorem{coro}[lemm]{Corollary}
	\newtheorem{prop}[lemm]{Proposition}
	\newtheorem{ques}{Question}
	\newtheorem{conj}{Conjecture}
	
\theoremstyle{definition}
	\newtheorem*{clai}{Claim}

	\newtheorem{rema}[lemm]{Remark}

	\newtheorem{defi}[lemm]{Definition}
	\newtheorem{cave}[lemm]{Caveat}
	\newtheorem{scho}[lemm]{Scholium}

\hypersetup{bookmarksdepth = 2} 
\numberwithin{equation}{section}         

\setlist[enumerate,1]{label={\upshape\alph*)},ref=\alph*}
\setlist[enumerate,2]{label={\upshape\arabic*)},ref=\arabic*}

\renewcommand{\setminus}{\smallsetminus}
\renewcommand{\emptyset}{\varnothing}
\renewcommand{\epsilon}{\varepsilon}
\renewcommand{\rho}{\varrho}
\renewcommand{\phi}{\varphi}

\def\Si{\Sigma}\def\La{\Lambda}\def\la{\lambda}
\def\De{\Delta}\def\de{\delta}\def\Om{\Omega}
\def\Ga{\Gamma}

   \def\DD{{\mathbb D}}

 \def\NN{{\mathbb N}}  
 \def\RR{{\mathbb R}}  
   
 \def\ZZ{{\mathbb Z}}

\def\fD{{\mathfrak{D}}}

\def\fK{{\mathfrak{K}}}
\def\fF{{\mathfrak{F}}}

\newcommand{\cC}{\mathcal{C}}
\newcommand{\cD}{\mathcal{D}}\newcommand{\cF}{\mathcal{F}}
\newcommand{\cG}{\mathcal{G}}\newcommand{\cH}{\mathcal{H}}

\newcommand{\cO}{\mathcal{O}}
\newcommand{\cP}{\mathcal{P}}\newcommand{\cQ}{\mathcal{Q}}\newcommand{\cR}{\mathcal{R}}
\newcommand{\cU}{\mathcal{U}}
\newcommand{\cV}{\mathcal{V}}\newcommand{\cW}{\mathcal{W}}
\newcommand{\cZ}{\mathcal{Z}}

\def\Diff{\mathrm{Diff}}
\def\dim{\operatorname{dim}}

\def\diam{\operatorname{diam}}

\newcommand{\eqdef}{\coloneqq} 
  
\newcommand{\arxiv}[1]{\href{http://arxiv.org/abs/#1}{\tt arXiv:{#1}}}

\makeatletter
\newcommand{\customlabel}[2]{\protected@write\@auxout{}{\string\newlabel{#1}{{#2}{\thepage}{#2}{#1}{}}}\hypertarget{#1}{#2}}
\makeatother

\begin{document}

\title[Robust existence of nonhyperbolic
ergodic measures]{Robust criterion for the existence of nonhyperbolic
ergodic measures}

\date{\today}

\author[J.~Bochi]{Jairo Bochi}
\address{Facultad de Matem\'aticas, Pontificia Universidad Cat\'olica de Chile}
\urladdr{\href{http://www.mat.uc.cl/~jairo.bochi}{http://www.mat.uc.cl/$\sim$jairo.bochi}}
\email{jairo.bochi@mat.puc.cl}

\author[C.~Bonatti]{Christian Bonatti}
\address{Institut de Math\'ematiques de Bourgogne}
\urladdr{\href{http://bonatti.perso.math.cnrs.fr}{http://bonatti.perso.math.cnrs.fr}}
\email{bonatti@u-bourgogne.fr}

\author[L.J.~D\'\i az]{Lorenzo J.~D\'\i az}
\address{Departamento de Matem\'atica, Pontif\'{\i}cia Universidade Cat\'olica do Rio de Janeiro} 
\urladdr{\href{http://www.mat.puc-rio.br/~lodiaz}{http://www.mat.puc-rio.br/$\sim$lodiaz}}
\email{lodiaz@mat.puc-rio.br}

\thanks{The authors received support from CNPq, FAPERJ, PRONEX, and Ci\^encia Sem Fronteiras CAPES (Brazil), 
Balzan--Palis Project, ANR (France), FONDECYT (Chile), and EU Marie-Curie IRSES Brazilian-European partnership in Dynamical Systems (FP7-PEOPLE-2012-IRSES 318999 BREUDS). The authors acknowledge the kind hospitality of
Institut de Math\'ematiques de Bourgogne and Departamento de Matem\'atica PUC-Rio while preparing this paper}



\begin{abstract}
We give explicit $C^1$-open conditions that ensure that a diffeomorphism
possesses a nonhyperbolic ergodic measure with positive entropy.
Actually, our criterion provides the existence of a partially hyperbolic 
compact set with one-dimensional center and positive topological entropy
on which the center Lyapunov exponent vanishes uniformly.

The conditions of the criterion are met 
on a $C^1$-dense and open subset of the set of a diffeomorphisms having a robust 
cycle.
As a corollary,
there exists a $C^1$-open and dense subset of
the set of non-Anosov robustly transitive diffeomorphisms consisting  of systems with
nonhyperbolic ergodic  measures with positive entropy.

The criterion is based on a notion of a blender defined dynamically in terms of strict invariance of a family of discs.
\end{abstract}

\maketitle


\section{Introduction}

\subsection{General context}
Since the end of the sixties it is well known that there exist 
dynamical systems (diffeomorphisms and flows) 
that are $C^r$-robustly nonhyperbolic:
every perturbation of the system fails to be hyperbolic. 
Nevertheless, hyperbolic-like techniques and concepts 
which are essentially weakened forms of  hyperbolicity 
are still among the main tools for  studying dynamical
systems, even beyond uniform hyperbolicity. 

\emph{Nonuniform hyperbolicity} is an ergodic-theoretical version of 
hyperbolicity and has its origin in Oseledets' theorem about the existence of Lyapunov exponents. 
An invariant ergodic  probability measure is called {\emph{hyperbolic}}  
if all its Lyapunov exponents are different from zero, and \emph{nonhyperbolic} otherwise.
For  $C^{1+\alpha}$-systems, Pesin's theory recovers essential parts of 
hyperbolicity  for generic points of hyperbolic measures. 
One speaks of nonuniform hyperbolicity  
when studying 
systems endowed with some reference measure. For instance, in the conservative setting, 
a natural question is to know in what generality the volume is a hyperbolic measure. 
(For instance, see \cite{SW,BB} for perturbative methods for removing zero Lyapunov exponents
in the conservative setting.)

Another point of view, which is the one adopted in this paper, 
consists in looking at all ergodic invariant measures carried by the system.
A diffeomorphism is called \emph{completely nonuniformly hyperbolic} if all its ergodic 
invariant measures are hyperbolic. 
Let us observe that there are examples of nonhyperbolic systems
that are completely nonuniformly hyperbolic and even have Lyapunov exponents far away from zero: see \cite{BBS,CLR,LOR}. 
However, all of these examples are very fragile. 
Therefore one can naively ask whether 
every diffeomorphism can be approximated by completely nonuniformly hyperbolic ones,
thus reclaiming in a weaker sense Smale's dream of denseness of hyperbolicity.
Unfortunately, the answer to this question is negative,
as was shown by Kleptsyn and Nalsky \cite{KN}.

Another question, in the opposite direction, is the following:
Can every diffeomorphism be approximated either by  uniformly hyperbolic ones
or by diffeomorphisms with nonhyperbolic ergodic measures?
The answer is positive: 
it was shown by Ma\~n\'e \cite{Ma} 
that every $C^1$-robustly nonhyperbolic diffeomorphism can be $C^1$-approximated
by  diffeomorphisms that have nonhyperbolic periodic points\footnote{Note that 
such a property of diffeomorphisms does not extend to vector fields, 
since there exist robustly nonhyperbolic Lorenz-like attractors whose periodic orbits are all robustly hyperbolic.} 
and in particular nonhyperbolic ergodic invariant measures.
In particular, the existence of zero 
Lyapunov exponents is  $C^1$-dense in the complement of
the closure of the set of hyperbolic diffeomorphisms. 
However, this dense subset is too ``thin'' (actually meager),
since  periodic orbits of  $C^r$-generic diffeomorphisms (for any $r\geq 1$) are hyperbolic,
by Kupka--Smale's theorem.
Now if we consider not only measures supported on periodic orbits
but arbitrary ergodic probability measures, we expect to have a much ``thicker'' dense set.

In order to be more precise, let us introduce some notation.
Given a compact manifold $M$ without boundary, 
define the following two disjoint $C^r$-open sets:
\begin{itemize}
 \item 
 $\cH^r(M)$ is the set of 
 hyperbolic $C^r$-diffeomorphisms, that is, those satisfying Axiom~A with no cycles
 or, equivalently,  with a hyperbolic chain recurrent set. 
 \item $\cZ^r(M)$ is the $C^r$-interior of the set of $C^r$-diffeomorphisms with a nonhyperbolic ergodic measure.
\end{itemize}
We pose the following question:

\begin{ques}[Hyperbolicity vs.\ robust zero exponents]\label{q.1}
Given $r\ge1$, is the (open) set $\cH^r(M) \cup \cZ^r(M)$ dense in $\Diff^r(M)$?
\end{ques}
A positive answer to this question would mean that the
existence of nonhyperbolic ergodic measures basically characterizes nonhyperbolic $C^r$-dynamics. 

Let us remark that the first set $\cH^r(M)$ is always nonempty,
and the second set $\cZ^r(M)$ is also nonempty provided $\dim M\ge 3$,
as it is proved in the aforementioned paper \cite{KN} 
(for $r=1$ and therefore for any $r\geq 1$) as a continuation of the results from \cite{GIKN}.

In this paper, still assuming $\dim M\ge 3$, we prove that the 
open set $\cZ^1(M)$ is not only nonempty but actually very large,
thus providing evidence that Question~\ref{q.1} should have a positive answer.
An answer in complete generality, even for $r=1$, 
seems to be a hard problem.
Before stating our results let us observe that Question~\ref{q.1} is  motivated and closely related to 
the following conjecture of Palis~\cite{Pa} about topological characterizations of nonhyperbolicity: 

\begin{conj}[Hyperbolicity vs.\ cycles] \label{conj.Palis}
Every diffeomorphism $f\in \Diff^r(M)$, ${r\geq 1}$, can be $C^r$-approximated 
by diffeomorphisms that are either hyperbolic or display homoclinic bifurcations 
(homoclinic tangencies or heterodimensional cycles).
\end{conj}
This conjecture is true for $C^1$-diffeomorphisms of surfaces, see
\cite{PuSa}. See also \cite{CP} for important progress in higher dimensions.

Note that the  homoclinic bifurcations in this conjecture are 
associated to periodic points. 
By Kupka--Smale's genericity theorem, most systems (i.e.,  generic ones)
do not display homoclinic bifurcations.
Bearing this fact in mind, 
a stronger version of the conjecture above was proposed in \cite{BD-cycles} and \cite[Conjecture~7]{B}.
It order to state it, let us recall that two transitive hyperbolic basic sets 
$\Lambda$ and $\Theta$ of a diffeomorphism $f$ have 
a {\emph{robust (heterodimensional) cycle}} if: 
\begin{itemize}
 \item $\Lambda$ and $\Theta$ have different \emph{$\mathrm{u}$-indices} 
 (i.e., their unstable bundles have different dimensions) and
 \item 
 there is a $C^1$-neighborhood $\cU$ of $f$ so that for every $g\in\cU$ one has 
 $$
 W^\mathrm{s}(\Lambda_g)\cap W^\mathrm{u}(\Theta_g)\neq \emptyset 
 \quad \mbox{and} \quad W^\mathrm{u}(\Lambda_g)\cap W^\mathrm{s}(\Theta_g)\neq \emptyset,
 $$
 where $\Lambda_g$ and $\Theta_g$ 
 are the hyperbolic continuations of $\Lambda$ and $\Theta$ for $g$. 
\end{itemize}
Observe that robust cycles can only occur in dimension $3$ or larger.
Let us denote by 
$\cR\cC^1(M)$ the $C^1$-open subset of  $\Diff^1(M)$
of diffeomorphisms with robust cycles.   
We can now state the following:

\begin{conj}[Hyperbolicity vs.\ robust cycles]\label{conj.robust} 
The union of the disjoint open sets $\cH^1(M)$ and $\cR\cC^1(M)$ is dense in $\Diff^1(M)$.
In other words, every diffeomorphism in $\Diff^1(M)$ can be $C^1$-approximated by diffeomorphisms that are either
hyperbolic or have robust cycles.\footnote{This conjecture involves robust cycles but does not involve homoclinic tangencies. 
The rationale behind this is the fact that 
most  heterodimensional cycles can be made robust by small perturbations: see \cite{BDK}. In contrast,
the only known examples of $C^1$-robust homoclinic tangencies occur in dimension at least $3$ 
and are associated to $C^1$-robust cycles: see \cite{BD-tang}.
Furthermore, $C^1$-surface diffeomorphisms do not have robust 
homoclinic tangencies: see \cite{gugu}.}
\end{conj}

A consequence of the results in \cite{DG} is  that $C^1$-generic
 diffeomorphisms with robust cycles have ergodic nonhyperbolic measures.
 Note that by  Kupka--Smale's theorem 
 the nonhyperbolic measures in 
 \cite{DG} cannot  be supported on
 periodic orbits. In the partially hyperbolic setting with one-dimensional central direction, 
\cite{BDG}  improves the results in \cite{DG} by showing that
these nonhyperbolic measures can be chosen
 with full support in the appropriate homoclinic class.

In this paper we strengthen the $C^1$-generic conclusion of the result of \cite{DG} 
and prove that $C^1$-robust cycles yield $C^1$-robust existence of 
nonhyperbolic ergodic measures. 
We also will see that these measures can be chosen with positive entropy.

\begin{theo}\label{t.dichotomy} 
Let $M$ be a compact manifold without boundary and of dimension \mbox{$d \ge 3$}. 
Then the subset of $\Diff^1(M)$ consisting of diffeomorphisms  
with  a nonhyperbolic ergodic  measure with positive entropy
contains an open and dense subset of the $C^1$-open set $\cR\cC^1(M)$
of diffeomorphisms with robust cycles.   
\end{theo}

In the notations above, Theorem~\ref{t.dichotomy} implies that 
$\cZ^1(M)\cap\cR\cC^1(M)$ is dense in  $\cR\cC^1(M)$.
Therefore if Conjecture~\ref{conj.robust} is true then 
Question~\ref{q.1} has a positive answer when $r=1$.

Theorem~\ref{t.main} below is a more general (and technical) version of Theorem~\ref{t.dichotomy}. 
The proof of that theorem is based on Theorem~\ref{t.practical},
which gives an explicit robust criterion for zero center Lyapunov exponents,
which in turn relies on Theorem~\ref{t.FlipFlop_weak}, which is an
abstract ergodic-theoretical criterion for the existence of zero Birkhoff averages.
The strategy used in this paper for the construction
of nonhyperbolic measures is very different from the previous methods,
which are reviewed in Subsection~\ref{ss.discussion}.

\medskip

Let us present some consequences of Theorem~\ref{t.dichotomy}.
A diffeomorphism is called {\emph{$C^1$-robustly transitive}}
if every $C^1$-diffeomorphism nearby  is
 {\emph{transitive}}, that is, has 
a dense orbit.  In dimension two, $C^1$-robustly transitive diffeomorphisms are  
Anosov diffeomorphisms of the torus $\mathbb{T}^2$ (see \cite{Ma}), but there are non-Anosov robustly
transitive diffeomorphisms on manifolds of dimension three or more 
(see~\cite{Sh,Mda,BD-robtran}).
Combining the results in \cite{Ha,BC,BD-cycles}, one obtains that there is a $C^1$-open and dense subset of the set of
nonhyperbolic robustly transitive diffeomorphisms consisting of  diffeomorphisms with robust cycles. In other words, 
 $\cH^1(M)\cup\cR\cC^1(M)$ contains an open and dense subset of  
the set of $C^1$-robustly transitive diffeomorphisms.
Thus the
following is a corollary of Theorem~\ref{t.dichotomy}: 

\begin{corom}
\label{c.informal}
The union of the set of Anosov diffeomorphisms
and the set
$\cZ^1(M)$  is  open and dense in the set of $C^1$-robustly transitive diffeomorphisms.
\end{corom}

There are several \emph{mutatis mutandis} versions of this corollary for robustly 
transitive sets and homoclinic or chain recurrence classes 
robustly  containing saddles of different indices.  
One of these is the following. 
A diffeomorphism  $f$ is called {\emph{$C^1$-tame}}
if each of its chain recurrence classes is $C^1$-robustly isolated. In this case, 
the number of chain recurrence classes is finite and constant in a $C^1$-neighborhood
of $f$. The following corollary asserts that 
Question~\ref{q.1} with $r=1$ has a positive answer if restricted to tame diffeomorphisms. 

\begin{corom} \label{c.tame}
There is a $C^1$-open and dense subset of 
the set of $C^1$-tame diffeomorphisms 
consisting of diffeomorphisms whose 
chain recurrence classes are either hyperbolic 
or support a nonhyperbolic ergodic measure with positive entropy.
\end{corom}

\subsection{Sharper results}


A finite sequence of points $(x_i)_{i=0}^n$ is an
{\emph{$\epsilon$-pseudo-orbit}} of 
a diffeomorphism $f\colon M\to M$
 if
$\mbox{dist}(f(x_i),x_{i+1})<\epsilon$ for all $i=0,\dots,n-1$. A
point $x$ is {\emph{chain recurrent}} for $f$ if
 for every $\epsilon>0$ there is an $\epsilon$-pseudo-orbit
 $(x_i)_{i=0}^n$ 
starting and ending at $x$ (i.e., with $x=x_0=x_n$). 
The {\emph{chain recurrent set}} is composed by all chain recurrent points of $f$
and is denoted by $\cR(f)$. 
This set splits into pairwise disjoint {\emph{chain recurrence classes}}: the
class $C(x,f)$ of $x\in \cR(f)$ is the set of points $y$ such that
for every $\epsilon>0$ there are $\epsilon$-pseudo-orbits joining
$x$ to $y$ and $y$ to $x$.

Given an $f$-invariant set $\Lambda$ and a continuous $Df$-invariant line field 
$E=(E_x)_{x\in \Lambda}$ over $\Lambda$, 
the {\emph{Lyapunov exponent}} of a point $x$ along the direction $E$ is defined as
$$
\chi_E(x)\eqdef 
\lim_{|n| \to \infty} \frac{\log \|Df^n_x (v)\|}{n}, \quad v\in E_x\setminus\{ 0\},
$$
where $\|\mathord{\cdot}\|$ stands for the Riemannian norm,
whenever this limit exists.

\begin{theo}\label{t.main}
Let $M$ be a compact manifold without boundary and of dimension $d \ge 3$. 
Let $\cU \subset \Diff^1(M)$ be an open set of diffeomorphisms  
such that for every $f \in \cU$ there are hyperbolic periodic points
$p_f$ and $q_f$, depending continuously on $f$, 
in the same chain recurrence class $C(p_f,f)$
and having respective $\mathrm{u}$-indices $i_p > i_q$.
Then there exists a $C^1$-open and dense subset $\cV$ of $\cU$ 
with the following properties.
For any $f \in \cV$ and any integer $i$ with $i_q < i \le i_p$
there exists a compact $f$-invariant set $K_{f,i} \subset C(p_f,f)$
with a partially hyperbolic splitting
$$
T_{K_{f,i}} M = E^\mathrm{uu} \oplus E^\mathrm{c} \oplus E^\mathrm{ss}
$$
such that:
\begin{itemize}
	\item 
	$E^\mathrm{uu}$ is uniformly expanding and has dimension $i-1 > 0$,
	$E^{\mathrm{c}}$ has dimension~$1$, and
	$E^\mathrm{ss}$ is uniformly contracting and has dimension $d-i > 0$;
	\item 
      the Lyapunov exponent along the $E^\mathrm{c}$ direction of
      any point in $K_{f,i}$ is zero; 
	\item the topological entropy of the restriction of $f$ to $K_{f,i}$ is positive.
\end{itemize}
\end{theo}

Additionally, the Lyapunov exponent along the central direction $E^\mathrm{c}$ is \emph{uniformly} zero, 
in the sense that the limits that by definition are the Lyapunov exponent
exist everywhere on $K_{f,i}$ and are uniform.
Although this uniformity follows abstractly from the facts that $K_{f,i}$ is compact and $E^\mathrm{c}$ is one-dimensional,
we will obtain it directly from the construction.

\medskip

Let us now see how Theorem~\ref{t.main} implies Theorem~\ref{t.dichotomy}.
Given a pair of hyperbolic sets forming a robust cycle, 
their union is contained in the same chain recurrence class
and contains hyperbolic periodic points of different indices,
permitting us to apply Theorem~\ref{t.main}.
This allows to conclude that diffeomorphisms $f$ in an open and dense subset of $\cR\cC^1(M)$
possess compact invariant partially hyperbolic sets $K_f$ with positive topological entropy
and with uniformly zero Lyapunov exponents along the center direction.
By the variational principle for entropy (see e.g.\ \cite{Walters}), 
each such $K_f$ supports an ergodic measure
of positive metric entropy, and so we obtain the statement of Theorem~\ref{t.dichotomy}.

%
%
%
%
%

In the converse direction, 
if $\cU$ is an open set of diffeomorphisms satisfying the hypotheses of Theorem~\ref{t.main},
that is, such that every $f \in \cU$ has a chain recurrence class
with periodic points $p_f$, $q_f$ of respective $\mathrm{u}$-indices $i_p > i_q$
then by results of \cite{BC,ABCDW,BD-cycles,BDK} 
there exists an open and dense subset $\cW$ of $\cU$ 
such that if $f \in \cW$ then the chain recurrence class $C(p_f,f)$
contains periodic points of every intermediate $\mathrm{u}$-index.
Moreover, any two of these periodic points 
having consecutive $\mathrm{u}$-indices
belong to a pair of hyperbolic basic sets forming a robust cycle\footnote{According to \cite{BC}, 
for $C^1$-generic diffeomorphisms $f$ in $\cU$ the  homoclinic class and the chain recurrent class of $p_f$ coincide, and the same occurs for $q_f$. 
Therefore the homoclinic classes $H(p_f)$ and $H(q_f)$ of $p_f$ and $q_f$ coincide. 
Now \cite[Corollary 2.4]{BCDG} claims that for every generic $f$ for which $H(p_f)$ and $H(q_f)$ coincide, the class $H(p_f)$ contains
hyperbolic sets $K_i$ of $u$-index $i$ for every $i_q\leq i\leq i_q$  so that $K_i$ and $K_{i+1}$, $i_q\leq i< i_p$, have a robust cycle.
}.

So if in Theorem~\ref{t.main} we replace the hypothesis  
``\emph{$p_f$ and $q_f$ are in the same chain recurrence class for every $f\in \cU$}'' 
by the stronger hypothesis ``\emph{$p_f$ and $q_f$ have consecutive indices and 
are related by a robust cycle for every $f\in \cU$}'',
we obtain a result which is not much weaker than its ancestor: 
the two theorems are equivalent modulo a nowhere dense closed subset of $\Diff^1(M)$.
To prove Theorem~\ref{t.main},
we will actually work with these robust cycles
associated to periodic points of consecutive indices.

\bigskip


Given an ergodic measure $\mu$, 
by Oseledets' Theorem we can define its {\emph{Lyapunov exponents}}
$\chi_1(\mu) \ge \chi_2(\mu) \ge \cdots \ge \chi_d(\mu)$, where $d =\dim M$.
While we will not recall here the full statement of that theorem,
let us remark that for $\mu$-a.e.\ point $x$ it is possible 
to choose linearly independent vectors $v_1$, \dots, $v_d$ in $T_x M$ such that 
$$
\lim_{n \to \pm \infty} \frac{\log \|Df^n_x (v_i)\|}{n} = \chi_i(\mu) \quad 
\mbox{for each $i$}.
$$
So Theorem~\ref{t.main} also has the following corollary:

\begin{corom}\label{c.i} 
Under the hypotheses of Theorem~\ref{t.main},
there exists a $C^1$-open and dense subset $\cV \subset \cU$ 
such that  every diffeomorphism $f\in \cV$ has ergodic measures 
$\mu_{i_p}, \dots ,\mu_{i_q-1}$ 
such that each $\mu_i$ has
positive entropy and its $i$-th Lyapunov exponent $\chi_i(\mu_i)$ is zero.
\end{corom}

As a matter of fact, 
our methods permit us to obtain not only a central Lyapunov exponent equal to zero, 
but a whole interval of central Lyapunov exponents.
If $p$ is a periodic point then we denote by $\chi_i(p)$ the $i$-th Lyapunov exponent
of the ergodic measure supported on the orbit of $p$.

\begin{theo}\label{t.interval}
Let $M$ be a compact manifold without boundary and of dimension $d \ge 3$. 
Let $\cU \subset \Diff^1(M)$ be an open set of diffeomorphisms  
such that for every $f \in \cU$ there are hyperbolic periodic points
$p_f$ and $q_f$, depending continuously on $f$, 
in the same chain recurrence class $C(p_f,f)$
and having respective $\mathrm{u}$-indices $i$ and $i-1$.

Then there exists a $C^1$-open and dense subset $\cV$ of $\cU$ such that for
every $f \in \cV$  and every $\chi\in \big(\chi_i (q_f),\chi_i(p_f)\big)$, 
there exists a compact $f$-invariant set $K_{f,\chi} \subset C(p_f,f)$
with a partially hyperbolic splitting 
$$
T_{K_{f,\chi}} M  = E^\mathrm{uu} \oplus E^\mathrm{c} \oplus E^\mathrm{ss}
$$
such that:
\begin{itemize}
	\item 
	$E^\mathrm{uu}$ is uniformly expanding and has dimension $i-1>0$,
	$E^\mathrm{c}$ has dimension~$1$, and
	$E^\mathrm{ss}$ is uniformly contracting and has dimension $d-i>0$;
	\item 
      the Lyapunov exponent along the $E^\mathrm{c}$ direction of any point in $K_{f,\chi}$ equals~$\chi$;
	\item the topological entropy of the restriction of $f$ to $K_{f,\chi}$ is positive.
\end{itemize}
In particular, the $i$-th Lyapunov exponent of any measure supported on $K_{f,\chi}$ equals $\chi$
and there is an ergodic measure supported on  $K_{f,\chi}$ with positive entropy.
\end{theo}

The proof of this theorem has two parts. The construction of the sets $K_{f,\chi}$ for $\chi$ bounded 
away from zero follows from the arguments in \cite{ABCDW} using Markov partitions and has a hyperbolic flavor. 
The new and more difficult part here is the construction of the sets  $K_{f,\chi}$ for $\chi$ close to zero.

\subsection{An abstract criterion for the existence of zero averages}

Along the proof of Theorem~\ref{t.main},
we develop some non-perturbative criteria for the existence of
zero center Lyapunov exponents, or more generally, zero limit Birkhoff averages.
This first criterion holds on a purely topological setting,
and relies in the following concept (see Figure~\ref{fig.ff}):

\begin{defi}[Flip-flop family]\label{d.flipflopfamily}
Let $(X,d)$ be a metric space, 
$f\colon X\to X$ be a continuous map, 
$K$ be a compact subset of $X$, 
and $\phi\colon K\to \RR$ be a continuous function. 

A \emph{flip-flop family} is a family $\fF$ 
of compact subsets of $K$ 
with uniformly bounded diameters that splits as  $\fF = \fF^+\cup \fF^-$ 
into two disjoint families satisfying the following properties:
\begin{enumerate}[label=FF\arabic*),ref=FF\arabic*,leftmargin=*,widest=FF3]
\item\label{i.FF1} 
There is a constant $\alpha>0$ such that for all members $D^+\in\fF^+$, $D^-\in\fF^-$ 
and all points $x^+\in D^+$,  $x^-\in D^-$ we have
$$
\phi(x^-)<-\alpha<0<\alpha<\phi(x^+).
$$
\item\label{i.FF2}
For every $D\in\fF$ there are compact subsets  $D^+$, $D^-$ of $D$ such that 
$f(D^+)\in\fF^+$ and $f(D^-)\in\fF^-$.
\item\label{i.FF3} 
There is a constant $\lambda>1$ such that 
if $E$ is contained in a member of $\cF$
and $f(E)$ is a member of $\cF$ then 
$$
d(f(x) , f(y)) \ge \lambda \, d(x,y) \quad \text{for every $x$, $y \in E$.}
$$
\end{enumerate}
\end{defi}

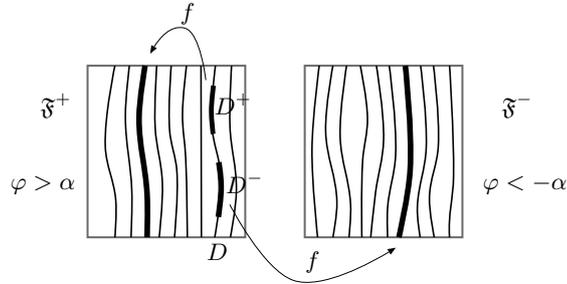
\begin{figure}[hb]
\begin{minipage}[c]{\linewidth}
\centering
\vspace{0.5cm}
\begin{overpic}[scale=.30, 
  ]{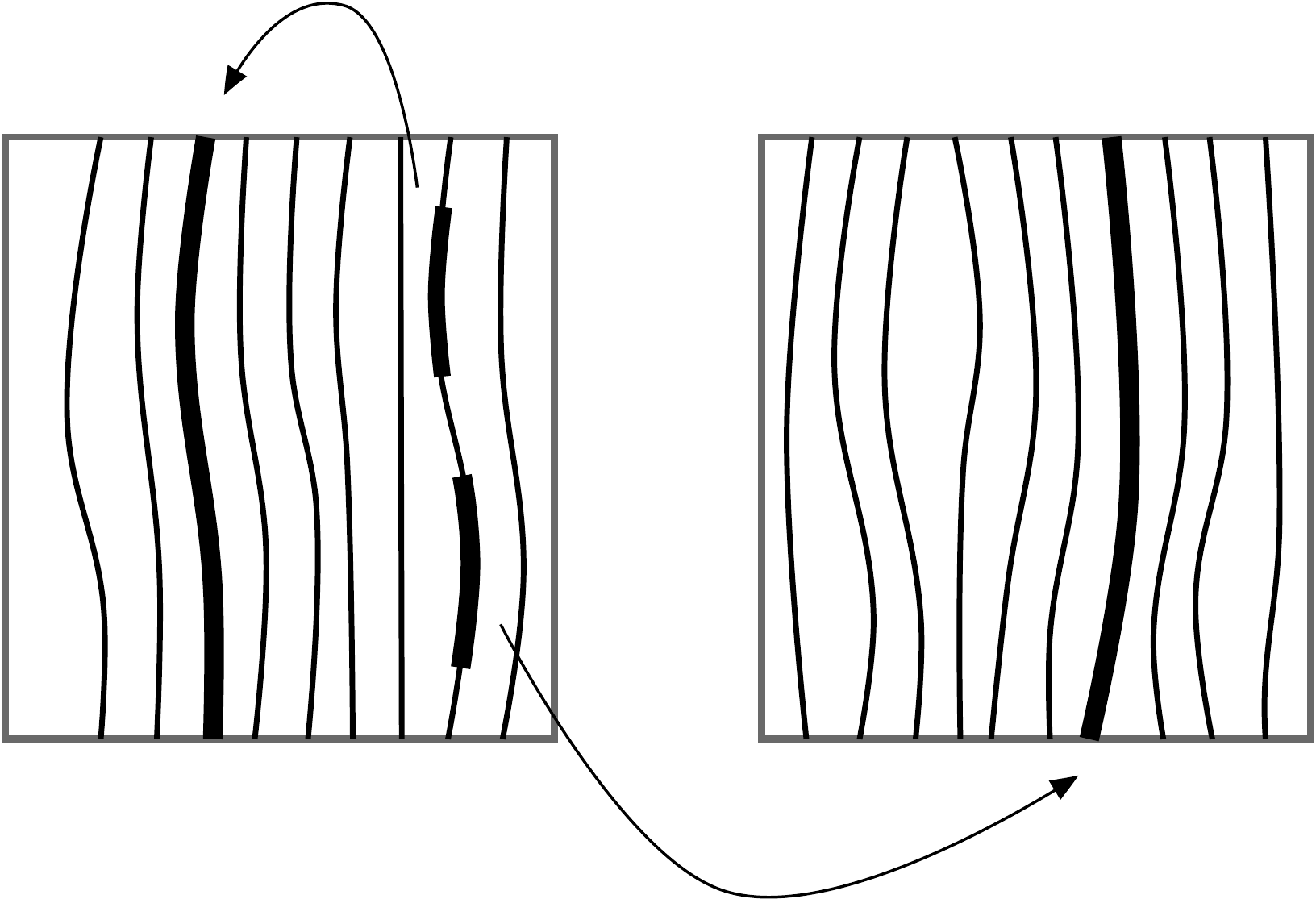}
       \put(-12,45){\small $\fF^+$}
       \put(-20,25){\small $\varphi>\alpha$}
        \put(105,25){\small $\varphi<-\alpha$}
        \put(110,45){\small  $\fF^-$}
        \put(25,70){\small  $f$}
         \put(58,4){\small  $f$}
           \put(34,45){\small  $D^+$}
           \put(37,24){\small  $D^-$}
           \put(32,6){\small  $D$}

  \end{overpic}
\caption{Flip-flop family}
\label{fig.ff}
\end{minipage}
\end{figure}

The motivation for the definition above is the following result:

\begin{theo}[Abstract criterion for zero limit Birkhoff averages] \label{t.FlipFlop_weak} 
Consider a continuous map $f \colon X \to X$ on a 
metric space $X$ having a flip-flop family 
$\fF$ associated to a continuous function $\phi\colon K\to \RR$
defined on a compact subset $K$ of $X$.

Then there exists a compact forward invariant subset $\Omega \subset K$ 
such that the Birkhoff averages of $\phi$ converge to zero uniformly on $\Omega$.
Moreover, the restriction of $f$ to $\Omega$ has positive topological entropy.

In particular, there exist ergodic $f$-invariant measures $\mu$
of positive entropy such that $\int \phi \, d\mu = 0$.
\end{theo}

Let us emphasize that the function $\phi$ is only continuous.
If $\phi$ were assumed to be more regular (say, H\"older)
then the corresponding theorem would be considerably simpler to prove (see Remark~\ref{r.Holder}
in this regard), but it would not be sufficient to obtain $C^1$-openness in 
Theorems~\ref{t.dichotomy} and \ref{t.main} above or Theorem~\ref{t.practical} below.

\subsection{A criterion for the existence of nonhyperbolic measures}

Coming back to diffeomorphisms, let us explain our next result,
which gives explicit $C^1$-open sufficient conditions for the existence of 
partially hyperbolic sets with zero center exponents.


We assume the existence of a \emph{dynamical blender}. This basically means a partially hyperbolic set together with a family of embedded discs tangent to a strong unstable cone such that the image of any perturbation of a disc in the family contains a disc in the family. See Section~\ref{s.blenders} for the precise definition. Dynamical blenders are a more flexible version of previous notions of blenders, see Section~\ref{ss.intro_blenders}.

In addition, we assume that there exists a saddle that forms together with the dynamical blender a certain
 \emph{split flip-flop configuration}. Roughly this means that the dynamics associated to the blender
 and the saddle is partially hyperbolic with one-dimensional center; see Definition~\ref{d.splitflipflop} for the precise definition.

We prove the following:
\begin{theo}[Criterion for zero center Lyapunov exponents]\label{t.practical}
	Let $M$ be a compact manifold without boundary and of dimension $d \ge 3$. 
	Assume that $f\in \Diff^1(M)$ has a periodic saddle $q$ and a dynamical blender $\Gamma$
	in a split flip-flop configuration.
	Let $i-1$ be the $\mathrm{u}$-index of $q$. 

	Then there exists a compact $f$-invariant set $K$ contained in the chain recurrence class $C(q,f)$
	and admitting a partially hyperbolic splitting 
	$$
	T_K M  = E^\mathrm{uu} \oplus E^\mathrm{c} \oplus E^\mathrm{ss}
	$$
	such that:
	\begin{itemize}
		\item 
		$E^\mathrm{uu}$ is uniformly expanding and has dimension $i-1 > 0$, 
		$E^\mathrm{c}$ has dimension~$1$, and
		$E^\mathrm{ss}$ is uniformly contracting and has dimension $d-i>0$;
		\item the Lyapunov exponent along the $E^\mathrm{c}$ direction of any point in $K$ is zero;
		\item the topological entropy of the restriction of $f$ to $K$ is positive.
	\end{itemize}
\end{theo}

We will deduce Theorem~\ref{t.practical} from Theorem~\ref{t.FlipFlop_weak}:
starting from split flip-flop configurations,
we will construct flip-flop families $\fF$ composed of discs 
contained in the strong unstable manifolds of a partially hyperbolic set 
with one dimensional center bundle;
the function $\phi$ will be essentially
the logarithm of the center Jacobian.

To prove Theorem~\ref{t.main}, we basically show that in the presence of a robust cycle
the hypothesis of Theorem~\ref{t.practical} are satisfied after a suitable $C^1$-perturbation.

\subsection{Discussion on the methods}\label{ss.discussion}


The literature contains a number of results on the existence of nonhyperbolic measures.
Let us briefly compare these results with those of the present paper.

The paper \cite{GIKN} deals with certain partially hyperbolic skew-products.
It constructs nonhyperbolic ergodic measures with nondiscrete support
as limits of sequences of measures supported on periodic orbits. 
Each periodic orbit in the sequence shadows the previous one for a 
large proportion of time -- this is the key property to obtain ergodicity.  
Each periodic orbit has a proportionally small tail far from the previous orbits,
which is chosen in order to make the center Lyapunov exponent smaller.
That \emph{method of periodic approximations}
was also used in subsequent papers \cite{KN,DG,BDG,BoBoDi} 
to find nonhyperbolic measures for partially hyperbolic dynamics.
In \cite{BoBoDi}, the method was extended to higher center dimension (in the skew-product setting) 
so to yield multiple zero exponents.
On the other hand, the nonhyperbolic ergodic measures constructed by the method of periodic approximations are highly ``repetitive'' and are likely to have zero entropy.

\medskip

The strategy developed in this paper is completely different from the periodic approximations one.
Using a recursive construction we find a point $x$
whose central expansion is controlled at all time scales.
Then its omega-limit set $\omega(x)$ is an invariant compact set which is 
\emph{completely nonhyperbolic}, meaning that every invariant measure supported on this set is nonhyperbolic. 
Moreover, the orbit of $x$ can be chosen ``noisy'' enough so that
the restriction of $f$ to $\omega(x)$ has positive topological entropy.

The strategy of recursive control at all time scales does not need much regularity:
the diffeomorphism $f$ in Theorem~\ref{t.practical} is only $C^1$,
and the function $\phi$ in Theorem~\ref{t.FlipFlop_weak} is only continuous.
Under stronger regularity assumptions it is possible to use a simpler strategy and obtain sharper conclusions:
assuming that $\phi$ is H\"older, for example, it is possible to obtain a set $\Omega$
in Theorem~\ref{t.FlipFlop_weak} over which the Birkhoff \emph{sums} are uniformly bounded
(see Remark~\ref{r.Holder}).
Actually that simpler strategy already appears in the proof of another result from the paper \cite{BoBoDi},
which constructs completely nonhyperbolic compact sets with positive topological entropy
in the context of skew-products.

\medskip

With the method of periodic approximations, 
it is relatively easy to obtain nonhyperbolic measures with ``large'' or sometimes full support.
On the other hand, the supports of the nonhyperbolic measures constructed in this paper
are completely nonhyperbolic and therefore should be relatively ``small''.
Nevertheless, 
it is seems reasonable to conjecture that our methods could be sharpened (by dropping uniformity)
so to yield nonhyperbolic measures with bigger support. 

Another natural question concerns the abundance of ergodic measures
that have a zero Lyapunov exponent with multiplicity as high as possible.

\subsection{Dynamical blenders}\label{ss.intro_blenders}

The paper \cite{BD-robtran} introduced a dynamical mechanism called \emph{blender},
which provides the existence of a hyperbolic set whose stable set behaves as its dimension were greater than the dimension
 of its stable bundle: the stable set intersects all elements of an open family of embedded discs of low dimension. 
 Important applications of blenders are the construction of $C^1$robust cycles, see \cite{BD-cycles}, 
 and $C^1$-robust transitive sets, see \cite{BD-robtran}.
Afterwards, the paper \cite{BD-tang} introduced a variation of this concept called \emph{blender-horseshoe} to obtain
robust tangencies. 

These definitions of blender
require quite specific mechanisms and are not flexible enough for the purposes of this paper.
 So we introduce here a new concept, what we call \emph{dynamical blender},
that only requires an invariance property of a family of disks, see Definition~\ref{d.dynamicalblender}. This property  
can be checked more easily and also
has the  clear advantage of being intrinsically robust.
Also, it easily implies the aforementioned property of the stable set (though we do not use this property directly).

\subsection{Organization of the paper}\label{ss.organization}

In Section~\ref{s.flipflop} we 
prove the abstract zero averages criterion (Theorem~\ref{t.FlipFlop_weak}).
This is done by finding special points whose orbits 
have controlled (i.e., small) Birkhoff averages at every time scale,
and are noisy enough so to produce positive entropy.

In Section~\ref{s.blenders}, we introduce \emph{dynamical blenders} 
and establish some properties for later use. 
Then in Section~\ref{s.flipflopyield} we introduce \emph{split flip-flop configurations}, show how they permit us to find 
flip-flop families associated to a certain function $\phi$ related to the central Jacobian, and complete the proof of Theorem~\ref{t.practical}.

In Section~\ref{s.spawners} we prove Theorem~\ref{t.main}:
Starting from a robust cycle associated to saddles of consecutive indices, we show that with an additional perturbation a flip-flop configuration appears,
and therefore the theorem follows from the previous results.

Finally, in Section~\ref{s.interval} we explain how to obtain Theorem~\ref{t.interval}.


\section{Flip-flop families and Birkhoff averages: proof of Theorem~\ref{t.FlipFlop_weak}}\label{s.flipflop}

The aim of this section is to prove Theorem~\ref{t.FlipFlop_weak},
the zero averages result for flip-flop families.
In fact, we will prove a slightly finer version of this result (Theorem~\ref{t.FlipFlop}).

In all of this section,
let $(X,d)$ be a metric space, 
$f\colon X\to X$ be a continuous map, 
$K$ be a compact subset of $X$, 
$\phi\colon K\to \RR$ be a continuous function,
and $\fF  = \fF^+ \cup \fF^-$ be a flip-flop family
as in  Definition~\ref{d.flipflopfamily}. 

Let us fix some notation:
Denote $\bigcup \fF \eqdef \bigcup_{D\in \fF} D$
and analogously for $\bigcup \fF^+$, $\bigcup \fF^-$. 
The Birkhoff sums of $\phi$ will be denoted as
$$
\phi_n(x)\eqdef 
\sum_{i=0}^{n-1} \phi(f^i(x)) 
\quad \text{if } n \in \NN , \ x \in \bigcap_{i=0}^{n-1} f^{-i}(K) \, .
$$

\subsection{Control of Birkhoff averages at all scales}\label{ss.control}

We begin with some definitions that will be central to our constructions.

\begin{defi}[Control] \label{d.anyscale} 
Given $\beta>0$, $t \in \NN^*$, and $T\in \NN^* \cup \{\infty\}$,
we say that a point  $x \in K$ is 
\emph{$(\beta,t,T)$-controlled}  
if $f^i(x)\in K$ for $0 \le i < T$ and
there exists a subset $\cP \subset \NN$ of \emph{control times}
such that
\begin{itemize}
	\item $0 \in \cP$,
	\item $T \in \cP$ if $T < \infty$, and $\cP$ is infinite if $T = \infty$, and
	\item if $k<\ell$ are two consecutive control times in $\cP$ then
	$$\ell - k \le t \quad \text{and} \quad \frac{1}{\ell-k} \left| \phi_{\ell-k} (f^k(x)) \right| \le \beta. $$
\end{itemize}
The point $x$ is \emph{controlled at all scales}
(with respect to $\phi$)
if there exist monotone sequences $(t_i)_i$ of natural numbers and $(\beta_i)_i$
of positive numbers,
 $t_i \to \infty$ and $\beta_i \to 0^+$, such that  $x$ is
 $(\beta_i,t_i,T)$-controlled 
 for every $i$.
\end{defi}

One can easy prove the following:


%
%
\begin{lemm} \label{l.omega} 
If $x \in K$ is controlled at all scales
then every point $y\in\omega(x)$ satisfies  
$$
\lim_{n\to \infty}\frac1n \phi_n(y)=0\,.
$$
Moreover, the limit is uniform over the $\omega$-limit set $\omega(x)$.
\end{lemm}

Although we will not need it, let us remark that the converse holds:
if the Birkhoff averages converge uniformly to zero
over an $f$-invariant set $\Omega$, then every point in $\Omega$
is controlled at all scales.

In this section, we will prove the following result:

\begin{theor} \label{t.FlipFlop} 
If $\fF$ is a flip-flop family associated to the map~$f$ and the function~$\phi$ 
then every member $D\in\fF$ contains a point $x$ that is controlled at all scales 
with respect to $\phi$
and such that the restriction of $f$ to the $\omega$-limit set $\omega(x)$ has positive topological entropy. 
\end{theor}

Note that Theorem~\ref{t.FlipFlop_weak} is a direct consequence
of  Theorem~\ref{t.FlipFlop} together with Lemma~\ref{l.omega}
and the variational principle for entropy.

\medskip

The idea of the proof of Theorem~\ref{t.FlipFlop} is roughly as follows:
For any member $D$ of $\fF$, it is possible to find points $x \in D$ 
that have a large number of iterates $n_1$ in the positive region $\bigcup \fF^+$,
and then a large number of iterates $n_2$ in the negative region $\bigcup \fF^-$ 
in such a way that the Birkhoff average $\phi_{n_1+n_2}(x)/(n_1+n_2)$ is positive and small,
but not exceedingly small.
So we obtain some control at the first time scale: see Figure~\ref{f.scale1}.
Analogously we can obtain small (but not exceedingly small) negative Birkhoff averages 
at this same scale.
We then pass to a second time scale where much smaller (say, positive) Birkhoff averages
are obtained by concatenating several controlled segments of the first scale,
the initial ones being of positive type, and the later ones being of negative type:
see Figure~\ref{f.scale2}.
The construction then proceeds recursively in order to control longer and longer time scales.
Moreover, we can incorporate some periodic noise in the construction and obtain positive entropy.

\begin{figure}[htp]
	\begin{center}
	\begin{tikzpicture}[scale=.8]
		\newcommand{\gap}{.06}
		\fill[color=gray!30](0,0)
			--++(1,1-\gap)--++(1,0.5-\gap)--++(1,0.7-\gap)--++(1,0.8-\gap)--++(1,0.5-\gap)--++(1,-0.9-\gap)--++(1,-0.6-\gap)--++(1,-0.9-\gap)		
			--++(0,16*\gap)
			--++(-1,0.9-\gap)--++(-1,0.6-\gap)--++(-1,0.9-\gap)--++(-1,-0.5-\gap)--++(-1,-0.8-\gap)--++(-1,-0.7-\gap)--++(-1,-0.5-\gap)--++(-1,-1-\gap);		
		\draw[thick,-latex] (0,0)--(8.5,0) node[below] {\footnotesize $n$};
		\draw[thick,-latex] (0,-.3)--(0,3.1) node[left] {\footnotesize $\phi_n(x)$}; 
		\draw (0,0)
			--++(1,1.0)--++(1,0.5)--++(1,0.7)--++(1,0.8)--++(1,0.5)--++(1,-0.9)--++(1,-0.6)--++(1,-0.9); 
		\draw[dashed] (8,1.1)--(0,0);
	\end{tikzpicture}
	\end{center}
	\caption{Control of Birkhoff averages at the first time scale.}\label{f.scale1}
\end{figure}
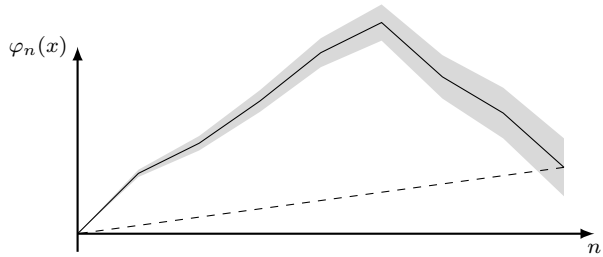	

\begin{figure}[htp]
	\begin{center}
	\begin{tikzpicture}[scale=.18]
		\fill[color=gray!30](0,0)
			--++(1,0.98)--++(1,0.48)--++(1,0.68)--++(1,0.78)--++(1,0.48)--++(1,-0.92)--++(1,-0.62)--++(1,-0.92)
			--++(1,0.98)--++(1,0.68)--++(1,0.78)--++(1,0.98)--++(1,0.48)--++(1,-0.72)--++(1,-0.82)--++(1,-0.82)
			--++(1,0.78)--++(1,0.78)--++(1,0.58)--++(1,1.08)--++(1,1.08)--++(1,-0.82)--++(1,-0.82)--++(1,-0.62)
			--++(1,0.68)--++(1,0.28)--++(1,0.38)--++(1,0.68)--++(1,0.58)--++(1,-0.32)--++(1,-1.02)--++(1,-1.02)
			--++(1,0.28)--++(1,0.48)--++(1,0.68)--++(1,0.78)--++(1,0.78)--++(1,-0.82)--++(1,-0.32)--++(1,-0.32)
			--++(1,-0.52)--++(1,-1.02)--++(1,-0.32)--++(1,-0.92)--++(1,-1.12)--++(1,0.98)--++(1,0.28)--++(1,0.68)
			--++(1,-0.62)--++(1,-0.52)--++(1,-0.32)--++(1,-0.32)--++(1,-0.92)--++(1,0.38)--++(1,0.58)--++(1,0.78)
			--++(1,-1.12)--++(1,-0.72)--++(1,-1.02)--++(1,-0.32)--++(1,-0.32)--++(1,0.78)--++(1,0.48)--++(1,0.38)
			--++(0,2.56)
			--++(-1,-0.42)--++(-1,-0.52)--++(-1,-0.82)--++(-1,0.28)--++(-1,0.28)--++(-1,0.98)--++(-1,0.68)--++(-1,1.08)
			--++(-1,-0.82)--++(-1,-0.62)--++(-1,-0.42)--++(-1,0.88)--++(-1,0.28)--++(-1,0.28)--++(-1,0.48)--++(-1,0.58)
			--++(-1,-0.72)--++(-1,-0.32)--++(-1,-1.02)--++(-1,1.08)--++(-1,0.88)--++(-1,0.28)--++(-1,0.98)--++(-1,0.48)
			--++(-1,0.28)--++(-1,0.28)--++(-1,0.78)--++(-1,-0.82)--++(-1,-0.82)--++(-1,-0.72)--++(-1,-0.52)--++(-1,-0.32)
			--++(-1,0.98)--++(-1,0.98)--++(-1,0.28)--++(-1,-0.62)--++(-1,-0.72)--++(-1,-0.42)--++(-1,-0.32)--++(-1,-0.72)
			--++(-1,0.58)--++(-1,0.78)--++(-1,0.78)--++(-1,-1.12)--++(-1,-1.12)--++(-1,-0.62)--++(-1,-0.82)--++(-1,-0.82)
			--++(-1,0.78)--++(-1,0.78)--++(-1,0.68)--++(-1,-0.52)--++(-1,-1.02)--++(-1,-0.82)--++(-1,-0.72)--++(-1,-1.02)
			--++(-1,0.88)--++(-1,0.58)--++(-1,0.88)--++(-1,-0.52)--++(-1,-0.82)--++(-1,-0.72)--++(-1,-0.52)--++(-1,-1.02);
		\draw[thick,-latex] (0,0)--(64+2,0) node[below] {\footnotesize $n$};
		\draw[thick,-latex] (0,-3)--(0,10) node[left] {\footnotesize $\phi_n(x)$};		
		\draw (0,0)
			--++(1,1.0)--++(1,0.5)--++(1,0.7)--++(1,0.8)--++(1,0.5)--++(1,-0.9)--++(1,-0.6)--++(1,-0.9) 
			--++(1,1)--++(1,0.7)--++(1,0.8)--++(1,1)--++(1,0.5)--++(1,-0.7)--++(1,-0.8)--++(1,-0.8) 
			--++(1,0.8)--++(1,0.8)--++(1,0.6)--++(1,1.1)--++(1,1.1)--++(1,-0.8)--++(1,-0.8)--++(1,-0.6) 
			--++(1,0.7)--++(1,0.3)--++(1,0.4)--++(1,0.7)--++(1,0.6)--++(1,-0.3)--++(1,-1)--++(1,-1) 
			--++(1,0.3)--++(1,0.5)--++(1,0.7)--++(1,0.8)--++(1,0.8)--++(1,-0.8)--++(1,-0.3)--++(1,-0.3) 
			--++(1,-0.5)--++(1,-1)--++(1,-0.3)--++(1,-0.9)--++(1,-1.1)--++(1,1)--++(1,0.3)--++(1,0.7) 
			--++(1,-0.6)--++(1,-0.5)--++(1,-0.3)--++(1,-0.3)--++(1,-0.9)--++(1,0.4)--++(1,0.6)--++(1,0.8) 
			--++(1,-1.1)--++(1,-0.7)--++(1,-1)--++(1,-0.3)--++(1,-0.3)--++(1,0.8)--++(1,0.5)--++(1,0.4); 
		\draw[dashed](0,0)
			--++(8,1.1)--++(8,1.7)--++(8,2.2)--++(8,.4)--++(8,1.7)
			--++(8,-1.8)--++(8,-.8)--++(8,-1.7);
		\draw[loosely dashed](64,2.8)--(0,0);
	\end{tikzpicture}
	\end{center}
	\caption{Control  of Birkhoff averages  at the second time scale.}\label{f.scale2}
\end{figure}
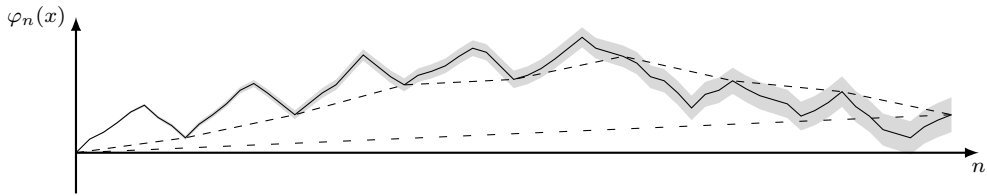

The sketch above is imprecise in many ways.
We need to make sense of what concatenation and noise mean.
Since the desired point $x$ is not known a priori,
at each time scale we need to control simultaneously
not only a single point but many of them.
Thus some ``uncertainty clouds'' appear; they are shown in gray in Figures~\ref{f.scale1} and \ref{f.scale2}.
Uncertainty increases with time due to the expanding character of the dynamics imposed by condition~(\ref{i.FF3}). 
Nevertheless it does not have a big effect on Birkhoff averages over long time scales.

We now proceed with precise proofs.
We need a few preparatory lemmas and definitions before 
proving Theorem~\ref{t.FlipFlop}.

%
%

\subsection{Segments and concatenations}

Recall that $\fF$ is a fixed flip-flop family with respect to a dynamics $f \colon X \to X$
and a function $\phi \colon K \to \RR$.

\begin{defi}[Segments]
Let $T \in \NN^*$. 
A \emph{$\fF$-segment of length $T$} 
is a sequence
$\mathbf{D} = \{D_i\}_{0\leq i\leq T}$ such that 
$f(D_i) = D_{i+1}$,
each $D_i$ is contained in a member of $\fF$,
and $D_T$ is a member of $\fF$.
The sets $D_0$ and $D_T$ are respectively called the \emph{entrance} and the \emph{exit} of $\fF$.
\end{defi}

\begin{lemm}[Birkhoff averages over long segments have small distortion]\label{l.distortion}
Let $\fF$ 
be a flip-flop family.
For every $\eta>0$ there exists $N = N(\eta) \in \NN$ 
such that if $\mathbf{D}=\{D_i\}_{0\leq i\leq T}$ is a segment of length $T \ge N$
then for every pair of points $x$, $y$ in the entrance $D_0$ we have
$$
\frac{1}{T} \left| \phi_T(x) - \phi_T(y) \right| < \eta \, .
$$
\end{lemm}

\begin{proof}
Let $\beta_1 \eqdef \sup_K |\phi|$.
Given $\eta>0$, let $\rho>0$ be such that 
$$
|\phi(z) - \phi(w)| < \eta/2, \quad
\mbox{for every $z$, $w \in K$
with $d(z,w) < \rho$.}
$$
Let $d_0$ be an upper bound for the diameters of the members of $\fF$,
and let $\lambda>1$ be the expansivity constant given by condition \eqref{i.FF3}.
Take $n_0 \in \NN$ such that $\lambda^{-n_0} d_0 < \rho$.
Fix an integer $N > n_0$ such that
$$
\frac{\beta_1\, n_0}{N} < \frac{\eta}{4} \, .
$$
Now suppose that $\mathbf{D} = \{D_i\}_{0\leq i\leq T}$ is a segment of length $T \ge N$,
and that $x$ and $y$ are points in the entrance $D_0$.
By condition~(\ref{i.FF3}) in definition of flip-flop family, 
for each $i  = 0, \dots, T-1$ we have
$$
d_0 \ge d(f^T x, f^T y) \ge \lambda^{T-i} d(f^i x, f^i y).
$$
In particular, if $i < T - n_0$ then $d(f^i x, f^i y) < \rho$.
Therefore we estimate:
$$
|\phi_T(x) - \phi_T(y)| 
\le
\sum_{i=0}^{T-n_0-1} \underbrace{|\phi(f^i x) - \phi(f^i y)}_{\le \eta/2}  +
\sum_{i=T-{n_0}}^{T-1} \underbrace{|\phi(f^i x) - \phi(f^i y)|}_{\le 2 \beta_1}
< T \eta ,
$$
which proves the lemma.
\end{proof}

\begin{rema} \label{r.Holder} 
Notice that in Lemma~\ref{l.distortion} we are comparing Birkhoff \emph{averages},
and not Birkhoff \emph{sums}.
If the function $\phi$ were not only continuous but, say, H\"older,
then by the classical (and easy) bounded distortion argument,
a much stronger result would hold:
there would exist an uniform (not depending on the length $T$)
upper bound for $\left| \phi_T(x) - \phi_T(y) \right|$.
With such estimates the proof of  Theorem~\ref{t.FlipFlop} would be considerably easier.
Actually, as mentioned in Subsection~\ref{ss.discussion},
one could strengthen the conclusion 
and obtain a set $\Omega$ which the Birkhoff sums are uniformly bounded.
\end{rema}

\begin{defi} [$(\beta,t)$-controlled segments] \label{d.atcontrolledsegment}
Given $\beta>0$ and $t \le T \in \NN^*$, 
a $\fF$-segment $\mathbf{D}=\{D_i\}_{0\le i \le T}$ 
 is said to be 
\emph{$(\beta,t)$-controlled} if 
there exists a \emph{set of control times}
$\cP \subset \{0,\dots,T\}$
such that
\begin{itemize}
	\item $0$, $T \in \cP$ and
	\item if $k<\ell$ are two consecutive control times in $\cP$ then
	$$\ell - k \le t \quad \text{and} \quad \frac{1}{\ell-k} \left| \phi_{\ell-k} (f^k(x)) \right| \le \beta \text{ for all } x \in D_0 \, . $$
\end{itemize}
That is, every point in the entrance $D_0$ is $(\beta,t,T)$-controlled and moreover we can take 
an uniform set of control times.
\end{defi}

\begin{defi}[Concatenations]
Consider $\fF$-segments 
$\mathbf{D}=\{D_i\}_{0\leq i\leq T}$ and $\mathbf{E}=\{E_j\}_{0\leq j\leq L}$ of lengths $T$ and $L$ such that
the exit of the first contains the entrance of the second, that is, $D_T \supset E_0$.  
Then the \emph{concatenation of  $\mathbf{D}$ and $\mathbf{E}$} is 
$\fF$-segment $\mathbf{D}*\mathbf{E} = \{F_k\}_{0\leq k \leq T+L}$ of length $T+L$
defined by
$$
F_k \eqdef
\begin{cases}
(f^{T-k}|D_k)^{-1}(E_0) &\quad\text{if } 0 \le k < T , \\ 
E_{k-T} &\quad\text{if } T \le k \le T+L . 
\end{cases}
$$ 
\end{defi}

The concatenation of $\fF$-segments is an associative operation on $\fF$-segments. This  allows us to define  multiple concatenations. 

The proof of the next lemma is straightforward and thus omitted.

\begin{lemm}[Concatenation preserves control]
\label{l.concatenation2} 
Consider $\beta>0$  and integers $t>0$, $T>t$, and $L>t$. 
Assume that 
$\mathbf{D}=\{D_i\}_{0\le i \le T}$  and $\mathbf{E}=\{E_j\}_{0\le  j\le L}$ are $(\beta,t)$-controlled $\fF$-segments with
$\cP$ and $\cQ$ as respective sets of control times.
Suppose that $D_T \supset E_0$.
Then  the concatenation 
$\mathbf{D}*\mathbf{E}$  is $(\beta,t)$-controlled with $\cP \cup (T+\cQ)$ as a set of control times.
\end{lemm}

\subsection{Patterns}\label{ss.pattern}

\begin{defi}[Pattern of an orbit]
Given a point $x$ in $\bigcup\fF$, 
a sequence of signs 
$\mathbf{s} = (s_n) \in \{+,-\}^\NN$, and $T\in \NN^* \cup \{\infty\}$,
we say that \emph{$x$ follows the $\tau$-pattern $\mathbf{s}$ up to time $T$} if
$$
\left. 
\begin{array}{l}
0 \le n < T \\ 
n \equiv 0 \pmod{\tau} 
\end{array}
\right\} \ \Rightarrow \ f^{n+1}(x) \in {\textstyle \bigcup} \fF^{s_n} \, .
$$
\end{defi}

Consider a $\fF$-segment
$\mathbf{D}=\{D_i\}_{0\leq i\leq T}$ of length $T$. 
Note that if a point $ D_0$  follows the $\tau$-pattern $\mathbf{s} \in\{+,-\}^\NN$ up 
to time $T$ then every point in $D_0$ 
follows the same $\tau$-pattern. 
If in addition $\tau$ divides $T$ then we say that the $\fF$-segment 
$\mathbf{D}$ of length  $T$ 
\emph{follows the $\tau$-pattern $\mathbf{s}$}.

Let $\sigma\colon\{+,-\}^\NN\to\{+,-\}^\NN$ denote the one-sided shift to the left.
The proof of the next lemma is straightforward.

\begin{lemm}[Concatenation and patterns]
\label{l.concatenation}
Let $\tau\in \NN^*$ and $\mathbf{s}\in \{+,-\}^\NN$.
Suppose $\mathbf{D}$ and $\bf E$ are $\fF$-segments of respective lengths $T$ and $L$ such that:
\begin{itemize}
 \item 
 $\mathbf{D}$ follows the $\tau$-pattern $\mathbf{s}$,
 \item 
 $\bf E$ follows the $\tau$-pattern $\sigma^{T/\tau}(\mathbf{s})$, and 
 \item 
 $E_0\subset D_T$.
\end{itemize}
Then the concatenation $\mathbf{D}*\mathbf{E}$ follows the $\tau$-pattern $\mathbf{s}$.
\end{lemm}

The purpose of the next lemma is to allow us to introduce periodic ``noise''
in the patterns followed by our orbits, while still allowing us to control Birkhoff averages:


\begin{lemm}\label{l.tau}
Fix $\alpha_1$ with $0 < \alpha_1 < \alpha$,
where $\alpha$ is as in condition (\ref{i.FF1}).
Then there exists an integer $\tau = \tau(\alpha_1) > 1$ with the following properties. 
Given any member $D$ of $\fF$, there exist four $\fF$-segments
$\mathbf{D}^{+,+}$, $\mathbf{D}^{-,+}$, $\mathbf{D}^{+,-}$, $\mathbf{D}^{-,-}$ such that:
\begin{enumerate}
\item the entrances $D^{+,+}$, $D^{-,+}$, $D^{+,-}$, $D^{-,-}$ of the segments are contained in~$D$;
\item the segments have length $\tau$;
\item the segments follow the respective $1$-patterns:
$$
(+,+,\ldots,+), \quad (-,+,+,\dots,+), \quad (+,-,-,\dots,-), \quad (-,-,\dots,-) \, ;
$$ 
\item for all $x \in D^{+,+} \cup D^{-,+}$ and $y \in D^{+,-} \cup D^{-,-}$, we have
$$
\frac{1}{\tau} \phi_{\tau}(x) \ge \alpha_1 >0 
\quad \text{and} \quad
\frac{1}{\tau} \phi_{\tau}(y) \le -\alpha_1 < 0 \, . 
$$
\end{enumerate}
\end{lemm}

%

\begin{proof}
Note that  $\alpha \le \beta_1 \eqdef \sup_K |\phi|$.
Choose  an integer $\tau > 1$ such that
$$
\alpha_1 < \frac{-\beta_1 + \alpha(\tau-1)}{\tau} \, . 
$$
Take $D \in \fF$.
Fix $s_1$, $s_2 \in \{+,-\}$.
By condition~(\ref{i.FF2}) in the definition of flip-flop family, 
the image $f(D)$ contains a member $D_1$ of $\fF^{s_1}$.
By induction, assume that a member $D_i$ of $\fF$ (where $0<i<\tau$) has already been defined.
Again by definition of flip-flop family, the image $f(D_i)$ contains an element $D_{i+1}$ of $\fF^{s_2}$.
Continuing in this way we eventually define a member $D_\tau$ of $\fF^{s_2}$.
Let $\mathbf{D}^{s_1,s_2}$ be the segment of length $\tau$ that has $D_\tau$ as exit.
Then this segment has the required properties.
\end{proof}

\subsection{Constructing controlled sets} \label{ss.controlled}

Fix sequences $(\beta_k)$ and $(\alpha_k)$ of positive numbers converging to zero of the form
$$
\beta_1  > \alpha_1 > \beta_2 > \alpha_2 > \cdots 
$$
and such that $\beta_1 \eqdef \sup_K |\phi|$ and $\alpha_1$ is less that the constant $\alpha$
given by condition \eqref{i.FF1}.
Let $\tau = \tau(\alpha_1)$ be given by Lemma~\ref{l.tau}.

The core of the proof of Theorem~\ref{t.FlipFlop} is the following lemma.

\begin{lemm}\label{l.induction}
There exists a sequence of integers $t_0 < t_1 < t_2 < \cdots$, 
where $t_0 = 1$, $t_1 = \tau$ and  
each element of the sequence is a multiple of its predecessor,
such that the following properties hold:

For every integer $k \ge 1$, every member $D$ of $\fF$,
and every  pattern  $\mathbf{s} \in \{+,-\}^\NN$,
there exist numbers  $T_+$, $T_- \in \NN^*$ and 
$\fF$-segments 
$\mathbf{D}^+$ 
and 
$\mathbf{D}^-$ 
of respective lengths 
$T_+$ and $T_-$ such that:
\begin{enumerate}
\item \label{i.La}
the entrances of $\mathbf{D}^+$ and $\mathbf{D}^-$ are contained in $D$;
\item \label{i.Lb}
the lengths $T_\pm$ are multiples of $\tau$ and satisfy
$t_{k-1} < T_\pm \le t_k$;

\item \label{i.Lc}
the segments $\mathbf{D}^+$  and  $\mathbf{D}^-$ are $(\beta_i,t_i)$-controlled for $i=1,2,\dots,k-1$;
\item \label{i.Ld}
for all $x$ in the entrance of $\mathbf{D}^+$ and all $y$ in the entrance of $\mathbf{D}^-$, we have
\begin{alignat}{2}
\alpha_k   &\le \frac{1}{T_+} \, \phi_{T_+}(x) &&\le  \beta_k,  \label{e.+}  \\
-\beta_k &\le \frac{1}{T_-} \, \phi_{T_-}(y) &&\le -\alpha_k;   \label{e.-}
\end{alignat}
in particular, the segment $\mathbf{D}^\pm$ is also $(\beta_k,t_k)$-controlled with 
$\cP^{\pm}=\{0, T_\pm\}$ as a set of control times;

\item \label{i.Le} 
the segments $\mathbf{D}^+$  and  $\mathbf{D}^-$ follow the $\tau$-pattern $\mathbf{s}$.
\end{enumerate}
\end{lemm}

\begin{proof} 
The sequence $(t_k)$ is constructed by induction.
Since $t_1 = \tau$, the conclusion of the lemma for $k=1$ follows from Lemma~\ref{l.tau},
taking $T_{\pm} = \tau$.

Let $k \ge 2$ and assume that $t_1 < t_2 < \cdots < t_{k-1}$
are already defined and that the conclusions of the lemma are met up to this point.
Fix an element $D$ of $\fF$.
We will explain how to construct the announced 
segment $\mathbf{D}^+$. The construction of $\mathbf{D}^-$ is analogous and hence omitted.

The segment will be obtained as a concatenation of $m+\ell$ $\fF$-segments $\mathbf{D}_i^+$, $i=1,\dots, m$, and 
$\mathbf{D}_{m+j}^-$, $j=1,\dots,\ell$, which are taken with the following properties:
\begin{enumerate}[label=I\arabic*),ref=I\arabic*,leftmargin=*,widest=I3]
 \item\label{i.I1}
 For each $i\in\{1,\dots,m\}$,   
 $\mathbf{D}_i^+$ 
 is a $\fF$-segment given by the induction hypothesis associated to $k-1$, 
 so its length $T_i^+ \le t_{k-1}$ is a multiple of $\tau$ and 
 all points in the entrance of the segment
 satisfies the inequality~\eqref{e.+} with $k-1$ in the place of $k$. 
 \item\label{i.I2}
 For each $j\in\{1,\dots,\ell\}$,   
 $\mathbf{D}_{m+j}^-$ 
 is a $\fF$-segment given by the induction hypothesis associated to $k-1$,
 so its length $T_{m+j}^- \le t_{k-1}$ is a multiple of $\tau$ and 
 all points in the entrance of the segment
 satisfies the inequality~\eqref{e.-} with $k-1$ in the place of $k$. 
 \item\label{i.I3}
 The following concatenation conditions hold:
 \begin{enumerate}
 \item 
 $D$ contains the entrance of $\mathbf{D}_1^+$;
 \item 
 for each $i\in \{1,\dots, m-1\}$, the exit of $\mathbf{D}_i^+$ contains the entrance of $\mathbf{D}_{i+1}^+$;
 \item
 the exit of $\mathbf{D}_m^+$ contains the entrance of $\mathbf{D}_{m+1}^-$;
 \item 
 for each $j\in \{1, \dots, \ell-1\}$, the exit of $\mathbf{D}_{m+j}^-$ contains the entrance of $\mathbf{D}_{m+j+1}^-$; 
 \end{enumerate}
 \item\label{i.I4}
 Each segment $\mathbf{D}_i^\pm$ follows the $\tau$-pattern $\sigma^{S_{i-1}/\tau}(\mathbf{s})$, where $S_0\eqdef0$ and
 $$
 S_i \eqdef\sum_{j=1}^i T_j^\pm \, ,
 \text{ the $j$-th sign $\pm$ being $+$ if $j \le m$ and $-$ otherwise.}
 $$
\end{enumerate}
It follows from the induction hypothesis on $k$
that for every $m>0$ and $\ell>0$
there are families of $\fF$-segments satisfying properties (\ref{i.I1})--(\ref{i.I4}) above;
just reason by induction on $m$ and $\ell$.


Given $m$ and $\ell$ and  families  
$\mathbf{D}_i^+ $, $i=1,\dots m$, and  $\mathbf{D}_{m+j}^-$, $j=1,\dots \ell$
as above, consider their concatenation 
$$
\mathbf{D}=\mathbf{D}_1^+*\cdots *\mathbf{D}_m^+*\mathbf{D}_{m+1}^-*\cdots*\mathbf{D}_{m+\ell}^-.
$$
Note that for every choice of $m$ and $\ell$ 
and of the families $\mathbf{D}_i^+ $ and  $\mathbf{D}_{m+j}^-$,
the concatenated   $\fF$-segment   $\mathbf{D}$ satisfies conditions (\ref{i.La}), 
(\ref{i.Lb}) (taking $t_k=(m+\ell)t_{k-1}$, say), 
(\ref{i.Lc}), and (\ref{i.Le}) 
in Lemma~\ref{l.induction}. 
The difficult part is to show that we can also obtain item (\ref{i.Ld}) 
with an uniform bound on $m+\ell$.  
This involves careful choices of these numbers and of the families. 

Before going into the details of this construction let us give a heuristic explanation.
The average of $\phi$ along the first $m$ segments is between $\alpha_{k-1}$ and $\beta_{k-1}$, 
in particular larger than the desired average (between $\alpha_k$ and $\beta_k$).  
The average along the next $\ell$ segments is negative, between $-\beta_{k-1}$ and $-\alpha_{k-1}$,
and so the total average decreases. 
The idea is to stop when this average is between $\alpha_k$ and $\beta_k$. 
We implement this idea in precise terms in the next paragraphs.

\subsubsection{Control of the distortion of the Birkhoff averages}
The first issue to be dealt with is that the Birkhoff averages depend on the point 
in the entrance 
of the concatenated segment $\mathbf{D}$.
Lemma~\ref{l.distortion} provides the solution. 

Choose a number $\eta$ with
$$
0< \eta < \frac{\beta_k-\alpha_k}3
$$
and
let $N=N(\eta)$ be given by Lemma~\ref{l.distortion}.
Then, for all integers $m  > N$ and  $j>0$, 
the variation of the Birkhoff averages of $\phi$ 
over the points in the entrance of a segment obtained by concatenation of $m+j$ segments is less than $\eta$. 

Suppose we find $\ell$ such that the concatenation of the $m+\ell$ segments has a point in its entrance
whose Birkhoff average 
$\phi_{S_{m+\ell}}/ S_{m+\ell}$ hits the interval $[\alpha_k+\eta,\beta_k-\eta]$,
which we call the \emph{target}.
In that case, it follows from the definition of $\eta$ that
the Birkhoff averages of $\phi$ at \emph{all} points in the entrance of the segment 
belong to the interval $(\alpha_k,\beta_k)$, as desired.

\subsubsection{Hitting the target}
Note that as the Birkhoff sums corresponding to the segments $\mathbf{D}_{m+j}^-$
are less than $-T_{m+j}^- \,\alpha_{k-1} < -\alpha_{k-1}$, 
if $\ell$ is large enough then the Birkhoff average $\phi_{S_{m+\ell}}/ S_{m+\ell}$
will be less than $\alpha_k$.  
We do not want that these averages go from above $\beta_k$ to below
$\alpha_k$ without hitting the target $[\alpha_k+\eta,\beta_k-\eta]$.  

To deal with this issue, consider a segment $\mathbf{D}_{m+j}^-$
and recall that by property~(\ref{i.I2})
the length of the segment is $T_{m+j}^- \le t_{k-1}$
and the Birkhoff sums $\phi_{T_{m+i}^-}$ of points in the entrance of this segment
belong to the interval
$[-T_{m+j}^-\,\beta_{k-1}, -T_{m+j}^-\,\alpha_{k-1}]$. 

An average $\phi_{S_{m+j}} / S_{m+j}$ belongs to the target interval $[\alpha_k+\eta,\beta_k-\eta]$ 
if, and only if, the Birkhoff sum  $\phi_{S_{m+j}}$ belongs to 
interval $[S_{m+j}\, (\alpha_k+\eta),S_{m+j}\, (\beta_k-\eta)]$,
whose length is 
$S_{m+j}\,  (\beta_k-\alpha_k-2\eta)$.
This is the size of gap that we want to be 
large enough so that it cannot be jumped over in just one step.

By the definition of $\eta$ we have
$$
\beta_k-\alpha_k-2\eta \geq \frac{\beta_k-\alpha_k}{3} \, ,
$$
while a crude lower bound for $S_{m+j}$ is $m$,
so the gap is larger than 
$$
\frac{m\,(\beta_k-\alpha_k)}3 \, .
$$ 
On the other hand, any step $|\phi_{S_{m+j}} - \phi_{S_{m+j+1}}|$
is at most $T_{m+j}^- \beta_{k-1} \le t_{k-1}\beta_{k-1}$.
Thus choosing  
$$
m > \frac{3 \, t_{k-1}\, \beta_{k-1}}{\beta_k-\alpha_k},
$$
the averages 
$\phi_{S_{m+j}/S_{m+j}}$  cannot go from a value 
bigger than
$\beta_k$ to a value less than $\alpha_k$ without hitting the target.
We now choose and fix such a number $m$ that additionally satisfies
$m > \max(N(\eta), t_{k-1})$, so that the previous reasoning applies.
It follows that there exists an integer $\ell>0$ such that 
the target is hit and so the desired \eqref{e.+} estimate holds
for all points in the the entrance of the segment 
$\mathbf{D}=\mathbf{D}_1^+*\cdots *\mathbf{D}_m^+*\mathbf{D}_{m+1}^-*\cdots*\mathbf{D}_{m+\ell}^-$.
We take the least such number $\ell$.

\subsubsection{Upper bound for the hitting time}
The final step is to bound $\ell$ (independently of $D$ etc) and thus be able to define $t_k$.
Fix an integer 
$$
\ell_0 > \frac{m \, t_{k-1}\, \beta_{k-1}}{\alpha_{k-1}}.
$$
We claim that for any point in the entrance of the 
concatenated segment $\mathbf{D}_1^+*\cdots *\mathbf{D}_m^+*\mathbf{D}_{m+1}^-*\cdots*\mathbf{D}_{m+\ell_0}^-$,
its Birkhoff sum $\phi_{S_{m+\ell_0}}$ is negative.
Indeed, this follows by breaking the sum into two parts, 
the first being at most $S_m \beta_k \le m t_{k-1} \beta_{k-1}$,
and the second being at most $(S_{m+\ell_0}-S_m)(-\alpha_{k-1}) \le - \ell_0 \alpha_{k-1}$.

It follows that $\ell < \ell_0$
and therefore $t_k \eqdef (m + \ell_0)t_{k-1}$ 
is an upper bound for the length $S_{m+\ell}$ of the segment $\mathbf{D}$.
Notice that $t_k$ does not depend on the member $D$ nor the pattern $\mathbf{s}$.

This completes the inductive construction, and so Lemma~\ref{l.induction} is proved.
\end{proof}


\subsection{End of the proof}

In this subsection we use the previous lemmas to prove Theorem~\ref{t.FlipFlop}.
For simplicity we break the proof into two parts:
in the first part we show that every $D \in \fF$ contains a point $x$ that is controlled at all scales,
and in the second part we explain how to find such a point with the additional 
property that $f|\omega(x)$ has positive entropy.

\subsubsection{Finding a point that is controlled at all scales}
Fix any member $D$ of the flip-flop family $\fF$.
By Lemma~\ref{l.induction}, we can find 
a monotone sequence $(\beta_i)$ of positive numbers converging to zero,
a monotone sequence $(t_i)$ of positive integers 
converging to infinity,
and a sequence of $\fF$-segments $(\mathbf{D}_k^+)$ such that:
\begin{itemize}
\item 
the entrance of $\mathbf{D}_k^+$ is contained in $D$;
\item 
the length $T_k^+$ of the segment $\mathbf{D}_k^+$ goes to infinity with $k$; 
\item 
every point in the entrance of $\mathbf{D}_k^+$ is $(\beta_i,t_i,T_k^+)$-controlled, for every $i$ with
$1\leq i\leq k$.
\end{itemize}

Choose a point $x_k$ in the entrance of  $\mathbf{D}_k^+$.
Let $x \in D$ be any accumulation point of the sequence $(x_k)$.
Since $x_k \in \bigcap_{i=0}^{T_k^+} f^{-i}(K)$ 
and $T_k^+ \to \infty$,
we conclude that the orbit of $x$ does not leave $K$.
We will show that $x$ is $(\beta_i, t_i, \infty)$-controlled for every $i$.

Suppose $x$ is the limit of a subsequence $(x_{k_j})$.
Let $i$ be fixed.
By construction, for every sufficiently large $j$,
the point $x_{k_j}$ is $(\beta_i, t_i, T_{k_j}^+)$-controlled.
Let  $\cP_j \subset \{0,\dots,T_{k_j}^+\}$ be the corresponding set of control times.
The $n$-th element of $\cP_j$ is bounded by $nt_i$, 
so that there are finitely many possibilities for this value. 
By a standard diagonal argument,
there exists a (strictly increasing) subsequence $\{j_\ell\}_{\ell\in \NN}$  
and an infinite set $\cP \subset \NN$
such that the first $\ell$ elements of $\cP_{j_\ell}$ and $\cP$ coincide.
 
By continuity, the average of $\phi$ along the orbit of $x$ 
between two successive times in $\cP$  
is the limit as $\ell \to \infty$ of the averages of $\phi$ along the orbit of
$x_{k_\ell}$ between the same times. 
This implies that all these averages belong to $[-\beta_i,\beta_i]$, 
proving that $x$ is $(\beta_i,t_i,\infty)$-controlled with set of control times $\cP$. 

We have proved the existence of a point $x \in D$ that is controlled at all scales,
thus proving Theorem~\ref{t.FlipFlop} up to the part of positive entropy,

\subsubsection{Positive entropy}
Let $\mathbf{s} = (s_0,s_1,\dots)$ be a point in the symbolic space $\{+,-\}^{\NN}$
whose orbit under the shift $\sigma$ is dense.

By construction, the lengths $T_k^+$ of the $\fF$-segments $\mathbf{D}_k^+$ considered above
are all multiples of a fixed integer $\tau$.
Moreover, by property~(\ref{i.Le}) in Lemma~\ref{l.induction},
we can choose each segment $\mathbf{D}_k^+$ following the $\tau$-pattern $\left(s_0,s_1,\dots,s_{T_k^+}\right)$.

Let $\alpha>0$ be the constant given by condition (\ref{i.FF1}),
and define two disjoint compact subsets of $K$:
$$
K^+\eqdef \{x\in K ; \; \phi(x)\geq \alpha\} 
\quad \mbox{and} \quad
K^-\eqdef \{x\in K ; \; \phi(x)\leq -\alpha\}.
$$
By definition, every member of $\fF$ is contained either in $K^+$ or in $K^-$.
As a consequence, every point in the orbit of 
$x$ belongs either to $K^+$ or to $K^-$. Thus  
$$
\omega(x)\subset \bigcap_{n\in\NN}f^{-n}( K^+\cup K^-) \,.
$$ 
Define a map $\Pi_\tau \colon \omega(x) \to \{+,-\}^\NN$ by 
$$
\Pi_\tau(y) \eqdef (\pi_i(y))_{i\in\NN}
\quad
\mbox{where}
\quad
\pi_i(y)= 
\begin{cases}
+ \text{ if } f^{\tau i}(y)\in K^+ \, ,\\
- \text{ if } f^{\tau i}(y)\in K^- \, .
\end{cases}
$$
In other words, $\Pi_\tau(y)$ is the itinerary of $y$ under iterations of $f^\tau$
with respect to the partition $\{\omega(x)\cap K^+, \omega(x)\cap K^-\}$
of the set $\omega(x)$.
The map $\Pi_\tau$ is continuous and satisfies $\Pi_\tau \circ f^\tau = \sigma^\tau \circ \Pi_\tau$.

Let us show that $\Pi_\tau$ is also onto.
It is sufficient to show that its image intersects every cylinder in $ \{-,+\}^{\NN}$. 
Consider a finite word
$\epsilon = (\epsilon_0,\dots,\epsilon_n)$ over the alphabet $\{+,-\}$. 
As the pattern $\mathbf{s} = (s_i)_{i\in\NN}$ has a dense orbit under the shift,
there is a sequence $k_i\to +\infty$ so that $(s_{k_i}, \dots, s_{k_i+n})=\epsilon$.
Let $y$ be any accumulation point of the sequence $(f^{\tau k_i}(x))_i$, 
so that in particular $y \in \omega(x)$.
Then $(\pi_0(y),\dots,\pi_n(y))=\epsilon$.
This shows that the image of $\Pi_\tau$ intersects the cylinder corresponding to $\epsilon$,
and we conclude that $\Pi_\tau$ is onto.

Hence we have proved that the restriction of $f^\tau$ to $\omega(x)$ is topologically semiconjugate
to the one-sided full-shift on $2$ symbols, and in particular has topological entropy 
$h_\text{top} (f^\tau|\omega(x)) \ge \log 2$.
In particular, $h_\text{top}(f|\omega(x)) \ge \tau^{-1}\log 2$ is positive,
as we wanted to show.

The proof of Theorem~\ref{t.FlipFlop} is complete.
As explained before, Theorem~\ref{t.FlipFlop_weak} follows.


\section{Blenders}\label{s.blenders}

In this section we introduce the notion of a {\emph{dynamical blender.}} This definition
involves three main ingredients: a space of discs (Section~\ref{ss.spaceofdiscs}),
invariant families of  discs (Section~\ref{ss.invariantfamilies}), and
invariant cone fields (Section~\ref{ss.invariantconefields}).
Thereafter with these ingredients
on hand we will define dynamical blenders in 
Section~\ref{ss.manyblenders}.

\subsection{The space of discs}
\label{ss.spaceofdiscs}
Let $M$ be a compact Riemannian manifold of dimension $d$.
For each $i\in\{1,\dots,d-1\}$ we denote by $\cD^i(M)$  the set of 
$i$-dimensional (closed) discs $C^1$-embedded in $M$.
We endow the space $\cD^i(M)$ with the following $C^1$-topology:
given a disc $D\in \cD^i(M)$ that is the image of an embedding 
$\phi \colon \DD^i \to M$ (here $\DD^i$ is the closed unit disc in $\RR^i$)
a basis of neighborhoods of it consists of the
sets
$\{\psi(\DD^i)\colon \psi \in \cV\}$, where $\cV$ runs over the neighborhoods of $\phi$
in the space of embeddings.
Alternatively, we can 
use
the sets of the form $\{f(D) \colon  f \in \cW\}$, where $\cW$ runs over the neighborhoods of the identity in $\Diff^1(M)$.



We now show that the topological space $\cD^i(M)$ can be metrized with a
 distance that behaves nicely with respect to  the composition of diffeomorphisms.

\begin{prop}\label{p.distance} 
There is a distance $\delta(\mathord{\cdot},\mathord{\cdot})$ inducing the 
$C^1$-topology in $\cD^i(M)$ that satisfies the following property.
For every $\varepsilon>0$ and $K>0$ there exists $\eta>0$ such that
for every pair of diffeomorphisms  $f, g \in \Diff^1(M)$ 
whose derivatives  and their inverses  are bounded by $K$ it holds
$$
d_{C^1}(f,g)<\eta \Longrightarrow \delta(f(D),g(D))<\varepsilon
$$
 for every disc 
$D\in \cD^i(M)$. 
\end{prop}

The rest of this subsection is dedicated to the proof of this proposition.

\subsubsection{The distance $\delta$}
\label{sss.distancedelta}
For each $j\in \{1,\dots, d-1\}$ let $\cG_j(M) \to M$ be the fiber bundle over $M$ whose fiber over a point $x\in M$  is the $j^{th}$-Grassmannian 
manifold of the tangent space $T_xM$.  In other words, a point $P\in \cG_j(M)$ whose projection in $M$ is the point $x$ is a  
subspace of dimension $j$ of the vector space $T_xM$. Recall that $\cG_j(M)$ is a compact manifold naturally endowed with a 
metric associated to the metric on $M$.

Given a disc $D\in \cD^i(M)$, we use the following notations:
\begin{itemize}
 \item $TD\subset \cG_i(M)$ denotes the compact subset $TD
 \eqdef \{(x,T_xD)\}_{x\in D}${\,};
 \item 
 $T\partial D\subset \cG_{i-1}(M)$ denotes the compact subset $T\partial D
 \eqdef \{(x,T_x\partial D)\}_{x\in\partial D}$.  
\end{itemize}

The distance $\delta$ is defined as follows, 
given a pair of discs $D_1$, $D_2\in \cD^i(M)$ we let
$$ 
\delta(D_1,D_2)\eqdef  d_\mathrm{Haus}(TD_1,TD_2)+d_\mathrm{Haus}(T\partial D_1,T\partial D_2),
$$
where $d_\mathrm{Haus}$ denotes the Hausdorff distance between compact subsets of a  metric space. 

\begin{rema}
\label{r.distancedelta}
$\,$
 \begin{enumerate}
  \item 
  $\delta$ defines a distance  in $\cD^i(M)$.
  \item 
  The distance $\delta$ is continuous with respect to the $C^1$-topology, 
  that is, the map $(D_1,D_2)\mapsto \delta(D_1,D_2)$ 
  is continuous in 
  $\cD^i(M)^2$.  
  This follows noting that  the maps $D\mapsto TD$ and $D\mapsto T\partial D$ are both continuous from $\cD^i(M)$ 
  (endowed with the $C^1$-topology) to the spaces $\cG_i(M)$ and $\cG_{i-1}(M)$, respectively, 
  (endowed with the Hausdorff topologies).
\end{enumerate}
\end{rema}

We now see that the distance 
$\delta$ satisfies the ``continuity"  property in Proposition~\ref{p.distance}.

\begin{lemm}\label{l.deltacompatibility} 
For every $\varepsilon>0$ and $K>0$ there exists $\eta>0$ such that
for every pair of diffeomorphisms $f, g \in \Diff^1(M)$ 
whose derivatives and their inverses 
are bounded by $K$ it holds
$$
d_{C^1}(f,g)<\eta \Longrightarrow \delta(f(D),g(D))<\varepsilon
$$
for every disc $D\in \cD^i(M)$.
\end{lemm}

\begin{proof}
It is enough to note that given any $K$ and $\varepsilon>0$ 
there is $\eta>0$ such that given any 
$f,g\in  \Diff^1(M)$ whose derivatives and their inverses of 
are bounded by $K$ and are $\eta$-close one has that 
for every $j\in \{1,\dots, d-1\}$ and every point 
$(x,P)\in \cG_j(M)$ 
it holds 
\[
d(Df(x)(P),Dg(x)(P))<\varepsilon. \qedhere
\]
\end{proof}

\subsubsection{
The distance $\delta$ defines the $C^1$-topology.}
To conclude the  proof of Proposition~\ref{p.distance} we are left to prove the following: 

\begin{lemm}
\label{l.esquecido}
Given any disk $D_0\in \cD^i(M)$ and any $C^1$-neighborhood $\cU$ of $D_0$ there is $\varepsilon >0$ such that if
$\delta(D_1,D_0)<\varepsilon$ then  $D_1\in \cU$.
\end{lemm}

\begin{proof} 
Choose a smooth embedding $\DD^i\times \DD^{d-i}\to M$ such that $D_0$ 
is contained in the (interior of)  image of the graph $\Ga$ of a map from $\DD^i$ to $\DD^{d-i}$.  
For $\varepsilon>0$ small enough every disc $D_1$ with $\delta(D_1,D_0)<\varepsilon$ is contained in $\DD^i\times \DD^{d-i}$ and is transverse
to the fibers $\{x\}\times \DD^{d-i}$.  
We will prove the following:

\begin{lemm}\label{l.tubular} 
With the notation above, 
for every $\varepsilon>0$ small enough, the projection of $D_1$ in the graph $\Ga$ is a diffeomorphism of $D_1$ into
a disc contained in $\Ga$ whose boundary is $C^1$-close to $\partial D_0$. 
\end{lemm}

Let us observe that this lemma implies Lemma~\ref{l.esquecido}.
For that just note that 
the disc $D_1$ is the image by a diffeomorphisms $C^1$-close to the identity 
of a disc in $\Ga$
which is $C^1$-close to $D_0$, this implies the proposition. 
We are left to prove Lemma~\ref{l.tubular}.
\end{proof}

\begin{proof}[Proof of Lemma~\ref{l.tubular}] 
Consider a small tubular neighborhood $\De\subset \DD^i\times \DD^{d-i}$ 
of $\partial D_0$ whose projection $\De\to\partial D_0$ 
commutes with the projection on $\Gamma$.

As, by hypothesis,  $T\partial D_1$ is $\delta$-close to $T\partial D_0$ the boundary $\partial D_1$ is 
contained in the 
tubular neighborhood $\Delta$  and is transverse to the fibers of $\Delta$. As a consequence, the projection $\pi$ from $\partial D_1$ to $\partial D_0$ along 
the fibers of $\Delta$ is a $C^1$-covering map.  

\begin{claim}
The projection  $\pi\colon
\partial D_1\to \partial D_0$ is a $C^1$-diffeomorphism. 
\end{claim}
\begin{proof} 
If $i>2$ then $\partial D_0$ is simply connected and  then the projection is a diffeomorphisms. Hence 
the unique case where the claim is not trivial is when $i=2$. 
Thus assume $i=2$ and that the projection is a covering with $k$ sheets. 
We need to see that $k=1$.

Consider the projection of  $\partial D_1$ in $\DD^2$ along the 
vertical fibers $\{x\}\times \DD^{d-2}$.  On the one hand, this projection is a closed 
immersed curved of 
$\DD^2$ whose index is precisely $k$, where the index is the number of turns made by the tangent direction of the curve when one goes around the curve. 
On the other hand, the index of the boundary of an immersed disc of $\DD^2$ is 
$1$.  Thus $k=1$,  
ending the proof of the claim.
\end{proof}

Using the claim we get that the projection of $\partial D_1$ in $\Ga$ is an embedded $i-1$ sphere that bounds 
an $i$-disc $\widetilde D$ in $\Ga$. The disc $D_1$ is a graph of a map defined on
$\Gamma$ close to the identity. This implies that $D_1$ is in a small neighborhood
of $D_0$, ending the proof the lemma.
\end{proof}

The proof of Proposition~\ref{p.distance} is now complete.

\subsection{Strictly invariant families of  discs and robustness}
\label{ss.invariantfamilies}

In this section we introduce strictly invariant families of discs and see that this property persists after small  perturbations of the diffeomorphisms. 

Given a family of discs $\fD\subset \cD^i(M)$ and $\eta>0$ we denote by $\cV^\delta_\eta(\fD)$ the open
$\eta$-neighborhood of $\fD$ with respect to the distance $\delta$, that is,
$$
\cV^\delta_\eta(\fD)\eqdef
\{D\in \cD^i(M)\colon \; \delta(D,\fD)<\eta\}.
$$

\begin{defi}[Strictly $f$-invariant families of discs]
Let $f$ be a diffeomorphism. A family of discs $\fD\subset \cD^i(M)$  is \emph{strictly $f$-invariant} if there is $\varepsilon>0$ such that for every disc
$D_0\in\cV^\delta_\varepsilon(\fD)$ there is a disc $D_1\in \fD$ with 
$D_1\subset f(D)$.  The number $\varepsilon$ is \emph{the strength} of the strict invariance.
\end{defi}

To be strictly invariant is a 
$C^1$-robust property  of a family of discs:

\begin{lemm} 
\label{l.strictdiscs}
Let $f$ be a diffeomorphism and  $\fD\subset \cD^i(M)$ 
a strictly $f$-invariant family of discs of strength $\varepsilon>0$. 
Then for every $0<\mu<\varepsilon$ there exists $\eta>0$ such that 
the family $\fD_\mu= \cV^\delta_\mu(\fD)$ 
is strictly $g$-invariant with strength $\varepsilon-\mu$
for every  $g\in \Diff^1(M)$ which is $\eta$-$C^1$-close to $f$. 
\end{lemm}

\begin{proof} Let $\eta$ (associated to $\mu$) be given by Lemma~\ref{l.deltacompatibility},
 then for every diffeomorphism $g$ which is $\eta$-$C^1$-close to
 $f$  and every disc in $D\in\cD^i(M)$ it holds
$\delta(f(D),g(D))<\mu$.

Take any disc
$D\in \cV^\delta_{\varepsilon-\mu}(\fD_\mu)$ and note that
$D\in \cV^\delta_{\varepsilon}(\fD)$.
As the family $\fD$ is strictly $f$-invariant with strength $\varepsilon$,
the set $f(D)$ contains a disc in $D_1 \in \fD$. 
Consider the disc $D_0=f^{-1}(D_1)\subset D$.  Now, by the choice of $\eta$,
$\delta(g(D_0),f(D_0))
= \delta (g(D_0),D_1)
<\mu$. Thus $g(D_0)\in \fD_\mu$ and hence $g(D)$ contains a disc
of $\fD_\mu$,
 concluding the proof of the lemma.
\end{proof}

\subsection{Strictly invariant cone fields}
\label{ss.invariantconefields}


A subset $C$ of a vector space $E$ is a \emph{cone of index $i$} 
if there are a splitting   $E = E_1 \oplus E_2$ with $\dim E_1 = i$
and a norm $\| \mathord{\cdot} \|$ on $E$ such that
$$
C = \{v_1 + v_2  \colon v_i \in E_i, \ \|v_2\| \le \|v_1\|\}.
$$
A cone $C'$ is \emph{strictly contained} in the cone $C$  above if there exists $\alpha>1$
such that 
$$
C' \subset C_\alpha=\{v_1 + v_2 \colon 
v_i \in E_i, \ \|v_2\| \le \alpha^{-1} \|v_1\|\} \subset C.
$$

A \emph{cone field of index $i$} defined on a subset $V$ of a compact 
manifold $M$ 
is a continuous assignment $x\mapsto \cC(x)\subset T_x M$  of
a cone of index $i$ for each $x \in V$.
Given $f\in \Diff^1(M)$, we say that this cone field is \emph{strictly $Df$-invariant} 
if $Df(x)(\cC(x))$ is strictly contained in $\cC(f(x))$
for every $x \in V \cap f^{-1}(V)$.

\begin{lemm}\label{l.cones} 
Let  $f$ be a diffeomorphism and $\cC$ a strictly $Df$-invariant cone field 
defined on a compact set $V\subset M$. Then there is a $C^1$-neighborhood
$\cU$ of $f$ so that $\cC$ is strictly $Dg$-invariant for every $g\in \cU$. 
\end{lemm}

\begin{proof} 
It suffices to note that, since $V$ is compact, the set $g^{-1}(V)\cap V$ depends upper semi-continuously on $g$, that is, for $g$ $C^0$-close to $f$, that set is contained in a small neighborhood of $V\cap f^{-1}(V)$. 
\end{proof}

\subsection{Geometric and dynamically defined blenders}
\label{ss.manyblenders}

We next introduce the notion of a {\emph{dynamical blender}} and state some of its properties
which are pertinent in our context. 
As a motivation let us first recall the definition of a {\emph{geometric blender}} as defined in 
\cite[Definitions 6.9 and 6.11]{BDV_book}.

\begin{defi}[Geometric blender] Let $f$ be a diffeomorphism and 
$\La$ a compact $f$-invariant set.  
We say that $\La$ is a \emph{geometric $\mathrm{cu}$-blender of $\mathrm{uu}$-index $i$} if
it is (uniformly) hyperbolic with $\mathrm{u}$-index strictly larger than $i$
and there exist
\begin{itemize}
 \item an open family $\fD\subset \cD^i(M)$ of discs and
 \item a $C^1$-neighborhood $\cU$ of $f$ 
\end{itemize}
such that
$D\cap W^s(\La_g)\neq \emptyset$
for every $g \in \cU$ and every $D\in\fD$, 
here $\La_g$ is the continuation of $\La$ for $g$.

The family $\fD$ is called the \emph{superposition region} and the $C^1$-neighborhood $\cU$ is the \emph{validity domain} of the blender. 
\end{defi}

The definition of a geometric blender is well suited for the study of the generation of robust cycles \cite{BD-cycles}.
A geometric blender is also 
an important mechanism to obtain
robust transitivity \cite{BD-robtran}. Indeed, these constructions were the main
motivations for its definition.
However, 
the fact that robustness forms part of its definition makes difficult 
in general to check if a given hyperbolic set is a geometric blender, restraining further applications. 
On the other hand, all known mechanisms producing 
geometric blenders involve the dynamical property of the existence of a strictly invariant family of 
discs\footnote{Besides the references above, see for instance the \emph{criterion of the recurrent compact set} given by \cite[Proposition 7.3]{ACW}, which is based on the previous works \cite{GZ,GY}.}.  
This motivates the following definition:

\begin{defi}[Dynamical blender] 
\label{d.dynamicalblender}
Let $f$ be a diffeomorphism. A compact $f$-invariant set $\Lambda$  
is a \emph{dynamically defined $\mathrm{cu}$-blender} (or simply \emph{dynamical blender}) 
\emph{of $\mathrm{uu}$-index $i$} if the following holds:
\begin{enumerate}
 \item  
 there is an open neighborhood $U$ of $\La$ such that 
 $\La$ is the maximal invariant set of $f$ in the closure of $U$,
 $$
 \La=\bigcap_{n\in\ZZ} f^n(\overline U);
 $$
 \item 
 the set $\La$ is (uniformly) hyperbolic with $\mathrm{u}$-index strictly larger than $i$;
 \item 
 the set $\La$ is transitive (thus it is a hyperbolic basic set);
 \item 
 there is a strictly $Df$-invariant cone field $\cC^{\mathrm{uu}}$ of index $i$ defined on $\overline U$;
 and
 \item 
 there is a strictly $f$-invariant family of discs  $\fD\subset \cD^i(M)$ with strength $\varepsilon>0$ such  
 that every disc in
 $\cV^\delta_\varepsilon(\fD)$ is contained in $U$ and tangent to $\cC^{\mathrm{uu}}$.
\end{enumerate}
We say that $U$ is the \emph{domain of the blender}, $\cC^{\mathrm{uu}}$ is its \emph{strong unstable cone field}, 
and $\fD$ is its 
\emph{strictly invariant family of discs} with strength $\varepsilon$. 
To emphasize the role of these objects in the definition of a geometrical blender we write
$(\La, U,\cC^{\mathrm{uu}}, \fD)$. To emphasize the strength $\varepsilon$ of the family of discs we write
$(\La, U,\cC^{\mathrm{uu}}, \fD, \varepsilon)$.
\end{defi}

\begin{rema}\label{r.partially}
Note that the hyperbolic splitting of a dynamical blender $\La$ as above 
can be refined in order to  get a partially hyperbolic splitting of the form 
$T_\La M= E^\mathrm{uu} \oplus E^\mathrm{c} \oplus E^\mathrm{ss}$,
where $E^\mathrm{uu}$ has dimension $i$
and $E^\mathrm{c}$ has positive dimension and is expanding.
Actually, all dynamical blenders we  consider in this paper have one-dimensional
central (unstable) bundle $E^\mathrm{c}$.
\end{rema}

We have that dynamical blenders are geometric ones and are robust:

\begin{lemm}[Blenders are robust]\label{l.robust} 
Assume that $(\La,U,\cC^{\mathrm{uu}},\fD, \varepsilon)$ is a dynamical 
blender of a diffeomorphism $f$. 
Let
$\fD_{\varepsilon/2}=\cV^\delta_{\varepsilon/2}(\fD)$  and for $g$ $C^1$-close to $f$ let
$\La_g$ be the hyperbolic continuation of $\La$ for $g$.
Then 
there is a $C^1$-neighborhood $\cU$ of $f$ such  that 
 for every diffeomorphism $g\in\cU$ the following holds:
\begin{itemize}
\item
the $4$-tuple $(\La_g,U,\cC^{\mathrm{uu}},\fD_{\varepsilon/2} )$ is a dynamical blender,
\item
the hyperbolic set $\La_g$ is
a geometric blender with superposition region 
$\fD_{\varepsilon/2}=\cV^\delta_{\varepsilon/2}(\fD)$ and 
validity domain $\cU$.
\end{itemize}
\end{lemm}

By analogy with the terminology for geometric blenders, the neighborhood $\cU$ is called a {\emph{validity domain}} of the dynamical
blender.

\begin{proof}[Proof of Lemma~\ref{l.robust}]
We choose a small neighborhood $\cU$ of $f$ such that
for every $g$ in $\cU$ the hyperbolic continuation $\La_g$ of $\La$ is well defined and equal
to the maximal $g$-invariant set in $\overline U$. By shrinking $\cU$ if necessary, we can also assume that 
the cone field $\cC^{\mathrm{uu}}$ is still strictly 
$Dg$-invariant (Lemma~\ref{l.cones}) and that the family $\fD_{\varepsilon/2}=\cV^\delta_{\varepsilon/2}(\fD)$ is 
still strictly $g$-invariant with strength $\varepsilon/2$ (Lemma~\ref{l.strictdiscs}). 
Summarizing, 
the $4$-tuple $(\La_g,U,\cC^{\mathrm{uu}},\fD_{\varepsilon/2}, \varepsilon/2 )$ is a dynamical blender.

To prove the second part of the lemma, first note that
by definition $\fD_{\varepsilon/2}$ is an open family of discs. Take any $g\in \cU$
and consider any disc $D_0\in \fD_{\varepsilon/2}$. By the strict $g$-invariance of this family (first part of the lemma) 
we have that
$g(D_0)$ contains a disc $D_1\in \fD_{\varepsilon/2}$. 
Arguing inductively we construct 
a sequence of discs $(D_n)_{n\in\NN}$ in $\fD_{\varepsilon/2}$ so that $D_{n+1}\subset g(D_n)$. 

By construction, the set   $\bigcap g^{-n}(D_n)$ is nonempty as it is the intersection of a decreasing sequence 
of nonempty compact sets. Moreover,
 the forward orbit of any point  $x$ in such an intersection 
is contained in $U$. Hence, as $\La_g$ is the maximal invariant set in $U$,  $x\in W^s(\La_g)$. This shows that every disc in the family  
$\fD_{\varepsilon/2}$ intersects $W^s(\La_g)$. 
This proves that $\La_g$ is a geometrical blender with validity domain $\cU$ and 
superposition region $\fD_{\varepsilon/2}$.
\end{proof}


\begin{scho}\label{s.fj} 
If $(\La,U,\cC^{\mathrm{uu}},\fD)$ is a dynamical blender then there exists a disc $D_\infty\in\cD^i(M)$ contained in the 
local strong unstable manifold of a point of $\La$
that is $C^1$-accumulated by discs in $\fD$.
\end{scho}

\begin{cave}
In the definition of a dynamical blender, there are two properties that we require
only for convenience in this paper and that  may be removed if necessary for further 
uses:
\begin{enumerate}
 \item 
 {\emph{Transitivity of the hyperbolic set $\La$.}} 
 We will use this condition in Proposition~\ref{p.l.flipflopcycle}  to show that a 
 flip-flop configuration contains a robust cycle. 
 \item 
 {\emph{The set
  $\La$ is maximal invariant in $\overline U$ and contained in $U$.}} 
  If this hypothesis is removed then 
 Lemma~\ref{l.robust} must  be slightly modified as follows: For every $g$ 
 $C^1$-close to $f$ the blender for $g$ is the maximal invariant set of $g$  in $\overline U$
 instead of the continuation of $\La$ for $g$. 
\end{enumerate}
\end{cave}

\subsection{Dynamical blenders and towers}
\label{ss.blendersandtowers}
In many situations, as in this paper, it is more natural and convenient to construct dynamical
blenders for some induced maps instead of blenders for the given diffeomorphism.
In this subsection, we show that a dynamical blender 
for an induced map of a diffeomorphism $f$
leads to a dynamical blender of the initial $f$ and see how these blenders are naturally 
and appropriately related.

\subsubsection{Induced maps}

Let $f$ be a diffeomorphism of a compact manifold $M$, 
$U, U_1,\dots,U_k \subset M$ be nonempty open sets, and $n_1,\dots, n_k$ be positive integers such
 that: 
\begin{itemize}
 \item the sets $\overline U_i$ are pairwise disjoint and are contained in $U$;
 \item $f^j(\overline U_i)\cap \overline U=\emptyset$ for all $0<j<n_i$; and
 \item $f^{n_i}(\overline U_i)\subset U$. 
\end{itemize}
The map $F\colon \bigcup_1^k \overline U_i\to U$ 
defined by $F(x) \eqdef f^{n_i}(x)$ if
$x\in \overline U_i$ is called an \emph{induced map} of $f$. 
We say that  $U_1, \dots, U_k$ are the {\emph{domains}} of $F$ and that
$n_1,\dots, n_k$ are the associated {\emph{return times.}} 
For such  an induced map we use the notation $F=[f, U, (U_i)_{i=1}^k, (n_i)_{i=1}^k]$.

A \emph{safety domain}  of the induced map $F=[f, U, (U_i)_{i=1}^k, (n_i)_{i=1}^k]$ is a set 
of the form 
$$
V=\bigcup_{i=1}^k  \Big( \bigcup_{j=0}^{n_i-1} V_{i,j}\Big),
$$
where the 
 sets $V_{i,j}$ are open and satisfy the following conditions:
 \begin{itemize}
  \item 
  the sets $\overline V_{i,j}$, $i\in\{1,\dots k\}$ and $j\in\{0,\dots, n_i-1\}$, are pairwise disjoint;
  \item 
  $\overline U_i\subset V_{i,0}\subset U$ for every $i\in\{1,\dots k\}$;
  \item 
  $f(\overline V_{i,j})\subset V_{i,j+1}$
for every $i\in\{1,\dots k\}$  and $j\in\{0,\dots, n_i-2\}$; and 
  \item
  $f(\overline V_{i,n_i-1})\subset U$  for every $i\in\{1,\dots k\}$.
 \end{itemize}

 
 \begin{lemm}[Existence of safety domains]\label{l.safety}
 Let
 $F=[f, U, (U_i)_{i=1}^k, (n_i)_{i=1}^k]$ be an induced map.
  Then for every family of neighborhoods $\widetilde U_i$ of $\overline U_i$, $i=1,\dots, k$, there is a safety domain 
  $$
  V=\bigcup_{i=1}^k  \Big( \bigcup_{j=0}^{n_i-1} V_{i,j} \Big)
  $$ 
  such that 
 $\overline V_{i,j}\subset f^j(\widetilde U_i)$ for every $i\in\{1,\dots, k\}$ and $j\in\{0,\dots, n_i-1\}$.
 \end{lemm}

 \begin{proof} 
 By definition, 
the compact set $\overline U_i$ is contained in the open set $U$ and the compact sets
 $f^j(\overline U_i)$, $j\in \{1,\dots ,n_i-1\}$,  are  pairwise disjoint and disjoint from $\overline U$. 
  Hence there is a decreasing  basis of open 
 neighborhoods $V^n_i$ of 
 $\overline U_i$ having the same properties 
 and such  that $\overline V^{n+1}_i$ is contained in $V^n_i$. 
 Now, for any $k$ the union of the family of opens $ V_{i,j}=f^j(V^{k+j}_i)$ is a safety domain. Moreover,
 if $k$ is large enough then $\overline V_{i,j}\subset f^j(\widetilde U_i)$ as announced.
\end{proof}

\subsubsection{Induced map and strictly invariant cone fields}
\label{sss.inducedcones}
 Consider an induced map 
 $F=[f, U, (U_i)_{i=1}^k, (n_i)_{i=1}^k]$.
 We say that 
\emph{a cone field $\cC$ defined on $U$ is strictly invariant for $F$} if for every 
$i\in\{1,\dots, k\}$ and  every $x\in\overline U_i$ the cone
$Df^{n_i}({\cC}(x))=
DF ({\cC}(x))$ is strictly contained in the cone $\cC(f^{n_i}(x))=\cC(F(x))$.

\begin{lemm}[Induced extended cone field]
\label{l.coneinduced}
Let
$F=[f, U, (U_i)_{i=1}^k, (n_i)_{i=1}^k]$ be an induced map
and $\cC$  a cone field defined on 
$U$ that is strictly invariant under $F$. Then
there is a safety domain
$V = \bigcup_{i=1}^k \Big( \bigcup_{j=0}^{n_i-1} V_{i,j}\Big)$  
of $F$
and a cone field $\widetilde{\cC}$ defined on 
$\overline{V}$  such that
\begin{itemize}
 \item 
 $\widetilde{\cC}$ coincides with $\cC$ in $\overline U$ and
 \item 
 the cone field $\widetilde{\cC}$ is strictly $Df$-invariant. 
\end{itemize}
\end{lemm}

\begin{proof}
Fix $N>\max_{i\in\{1,\dots k\}} n_i$.
By 
compactness and uniform continuity, one may find a sequence of cone fields 
$\cC_j$, $j\in\{0,\dots, N\}$, defined on
$\overline U$ such  that 
the cones of $\cC$ are  strictly contained in cones of $\cC_0$, 
the cones of  $\cC_j$ are strictly contained in cones of $\cC_{j+1}$ for every $j\in\{0,\dots, N-1\}$, 
and the cones of  $Df^{n_i}(\cC_N(x))$ are  strictly contained in cones
of $\cC(F(x))$ for every $i\in\{1,\dots,k\}$ and every $x\in \overline U_i$.
Again by compactness, the cone $Df^{n_i}(\cC_N(x))$ is strictly contained in 
$\cC (F(x))$ for $x$ in a small  neighborhood on $\overline U_i$. 

Using Lemma~\ref{l.safety} we can choose a sufficiently small safety neighborhood $V$ of $F$ with elements
$V_{i,j}$ and  define the cone field $\widetilde{\cC}$ as $Df^j(\cC_j)$ on the set $\overline V_{i,j}$. 
By construction  the cone field $\widetilde C$ is strictly $Df$-invariant.
\end{proof}

\subsubsection{Induced maps and dynamical blenders}
We now  state a  proposition that relates the dynamical blender of an induced
map  of a diffeomorphism (see the definition below) and the dynamical blenders of the initial dynamics.
We begin with a remark about families of invariant discs.
For the following statements, 
recall that $\cV^\delta_\varepsilon(\mathord{\cdot})$ denotes the open $\varepsilon$-neighborhood for the distance~$\delta$
in the space of discs in Section~\ref{sss.distancedelta}.

\begin{rema}\label{r.discs} 
Let $\fD, \fD_1,\dots, \fD_k \subset \cD^i(M)$ be families of $i$-dimensional discs 
and $n_1,\dots, n_k$  positive integers
such that 
$$
\fD=\fD_1\cup\cdots\cup\fD_k
$$ 
and
 there is $\varepsilon >0$ such that for every
 $i\in \{1,\dots k\}$ the image  $f^{n_i}(D)$ of any disc $D\in \cV^\delta_{\varepsilon}(\fD_i)$ 
 contains a disc in $\fD$.
Then  the family of discs
$$
\bigcup_{i=1}^k \fD_i \, \cup \, 
\bigcup_{i=1}^{k} 
\Big( \bigcup_{j=1}^{n_i-1} f^j \big( \cV^{\delta}_{\frac{j\varepsilon}{n_i}}(\fD_i) \big) \Big)
$$
is strictly $f$-invariant. 
\end{rema}

\begin{prop}[Induced blenders induce blenders]
\label{p.l.blenderinduced}  
Consider an induced map
$F=[f, U, (U_i)_{i=1}^k, (n_i)_{i=1}^k]$ such that:
\begin{itemize}
\item 
The maximal invariant set  
$$
\La_F=\bigcap_{n\in\ZZ} F^n\left(\bigcup\nolimits_{i=1}^k \overline U_i\right) \subset \bigcup_{i=1}^k  U_i
$$
is hyperbolic (for $F$) with $\mathrm{u}$-index strictly larger than $i$.
 \item 
 There is a cone field
 $\cC^{\mathrm{uu}}$ of index $i$ defined on $\overline U$ that is strictly invariant under $F$.
 \item  
 There are families $\fD,\fD_1,\dots,\fD_k \subset \cD^i(M)$ of $i$-dimensional discs
 and $\varepsilon>0$ such that: 
 \begin{enumerate}[label={\rm\alph*)},ref=\alph*]
  \item\label{i.bi_a} the families $\cV^\delta_\varepsilon(\fD),\cV^\delta_\varepsilon(\fD_1),\dots,\cV^\delta_\varepsilon(\fD_k)$ are contained in 
  $U, U_1,\dots, U_k$, respectively,  and are tangent to the cone field $\cC^{\mathrm{uu}}$;
  \item\label{i.bi_b} 
  given any disc $D\in \fD$ there are $i\in\{1,\dots, k\}$
  and a disc $D_0\in \fD_i$ such that $D_0\subset D$;
  \item\label{i.bi_c} 
  for every $i\in\{1,\dots, k\}$ and every $D\in \cV^\delta_\varepsilon(\fD_i)$, 
  the set $f^{n_i}(D)$ contains a disc $D_0\in\fD$.
 \end{enumerate}
\end{itemize}
Then there is a safety domain $V$ of $F$  such that the
maximal invariant set $\La_f$ of $f$ in $\overline V$  hyperbolic and the $4$-tuple
$(\La_f,V, \widetilde{\cC}^\mathrm{uu},\widetilde{\fD})$ is a dynamical blender, where
\begin{itemize}
\item
$\widetilde{\cC}^\mathrm{uu}$ is
a
 strictly $Df$-invariant 
 cone field  
 defined over $\overline V$ that extends 
 $\cC^{\mathrm{uu}}$ and
 \item
 $\widetilde{\fD}$ is 
 a strictly $f$-invariant family of discs 
contained in $V$ such that $\{D\in\widetilde{\fD};\; D\subset \overline U\}=\bigcup_{i=1}^k\fD_i$. 
\end{itemize}
\end{prop}

 With the notation above, we say that 
$(\La_F,U_1\cup\cdots\cup U_k, \cC^\mathrm{uu}, \fD_1\cup \cdots \cup \fD_k)$
is \emph{dynamical blender of the induced map $F$}.

\begin{proof}
Using Lemma~\ref{l.coneinduced} we find a safety domain 
$V = \bigcup_{i=1}^k \Big( \bigcup_{j=0}^{n_i-1} V_{i,j}\Big)$  of $F$
and extend the cone field  ${\cC}^{\mathrm{uu}}$
to a strictly $Df$-invariant.
cone field $\widetilde{\cC}^{\mathrm{uu}}$ defined on $\bar{V}$.
 By assumption, for every disc 
$D\in \cV^\delta_\varepsilon(\fD_i)$ the set $f^{n_i}(D)$ contains a disc $D_0\in\fD$. 
As
$V_{i,0}$ is a neighborhood of $\overline U_i$, the discs in 
$\cV^\delta_\varepsilon(\fD_i)$ are contained $V_{i,0}$. Thus, by definition of the safety neighborhood, 
the discs in $f^j(\fD_i)$
are contained in $V_{i,j}$ and 
are tangent to the cone field $\widetilde{\cC}^{\mathrm{uu}}$ 
(note that the cone field is strictly invariant). 
According to Remark~\ref{r.discs}, 
$$
\bigcup_{i=1}^k \fD_i \cup  \bigcup_{i=1}^{k} \Big(
\bigcup_{j=1}^{n_i-1} f^j \big(  \cV^{\delta}_{\frac{j\varepsilon}{n_i}}(\fD_i) \big) \Big)
$$
is a strictly $f$-invariant family of discs contained in $V$ and tangent to $\cC^{\mathrm{uu}}$ 
that coincides with $\bigcup_{i=1}^k\fD_i$ in $U$.
This concludes the proof of the proposition.
\end{proof}


\section{From flip-flop configurations to flip-flop families:
proof of Theorem~\ref{t.practical}}
\label{s.flipflopyield}

We next introduce  {\emph{flip-flop configurations}} and prove that they are robust and generate robust cycles.

\subsection{Flip-flop configurations}\label{ss.flipflopconfiguration}

\begin{defi}[Flip-flop configuration]
Let 
$(\Ga, V, \cC^\mathrm{uu}, \fD)$ 
 be a dynamical blender of $\mathrm{uu}$-index $i$ of a diffeomorphism $f$.
%
%
Suppose $q$ is a periodic point of $\mathrm{u}$-index $i$.
We say that $(\Ga, V, \cC^\mathrm{uu}, \fD)$   and $q$ form a \emph{flip-flop configuration} if
there exist:
\begin{itemize}
\item a disc $\Delta^\mathrm{u}$  contained in the unstable manifold $W^\mathrm{u}(q)$;
\item a compact submanifold with boundary $\De^\mathrm{s} \subset V \cap W^\mathrm{s}(q)$
\end{itemize}
such that:
\begin{enumerate}[label=FFC\arabic*),ref=FFC\arabic*,leftmargin=*,widest=FFC2,start=0]
\item\label{i.FFC0} 
the disc $\Delta^\mathrm{u}$ belongs to the interior of the family $\fD$;
\item\label{i.FFC1}
$f^{-n}(\Delta^\mathrm{u}) \cap \overline V = \emptyset$ for all $n>0$;
\item\label{i.FFC2}
There is $N>0$ such that $f^{n}(\Delta^\mathrm{s}) \cap \overline V = \emptyset$ for all $n>N$. Moreover, if 
$x\in \De^\mathrm{s}$ and $j>0$ are such  that 
$f^j(x)\notin V$ then 
$f^i(x)\notin \overline V$ for every $i\ge j$;
\item\label{i.FFC3} 
$T_y W^\mathrm{s}(q) \cap \cC^\mathrm{uu}(y) = \{0\}$ for every $y \in \Delta^\mathrm{s}$; 
\item\label{i.FFC4} 
There exists a compact set $K$ contained in the relative interior of $\Delta^\mathrm{s}$ and  $\epsilon>0$ such that 
every element $D$ of $\fD$ intersects the set $K$ at a point $x$ whose distance to $\partial D$ is larger than $\epsilon$.
\end{enumerate}

The sets $\De^{\mathrm{u}}$ and $\De^{\mathrm{s}}$ are the \emph{ unstable and stable connecting sets} 
of the flip-flop configuration, respectively. The compact set $K$ is an 
{\emph{$\varepsilon$-safe stable connecting set}}. 
\end{defi}

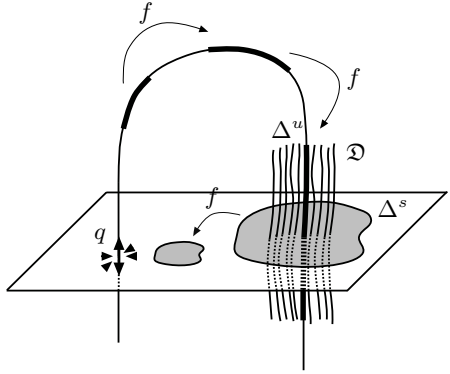
\begin{figure}[htb]
\begin{minipage}[c]{\linewidth}
\centering
\vspace{0.5cm}
\begin{overpic}[scale=.34, 
  ]{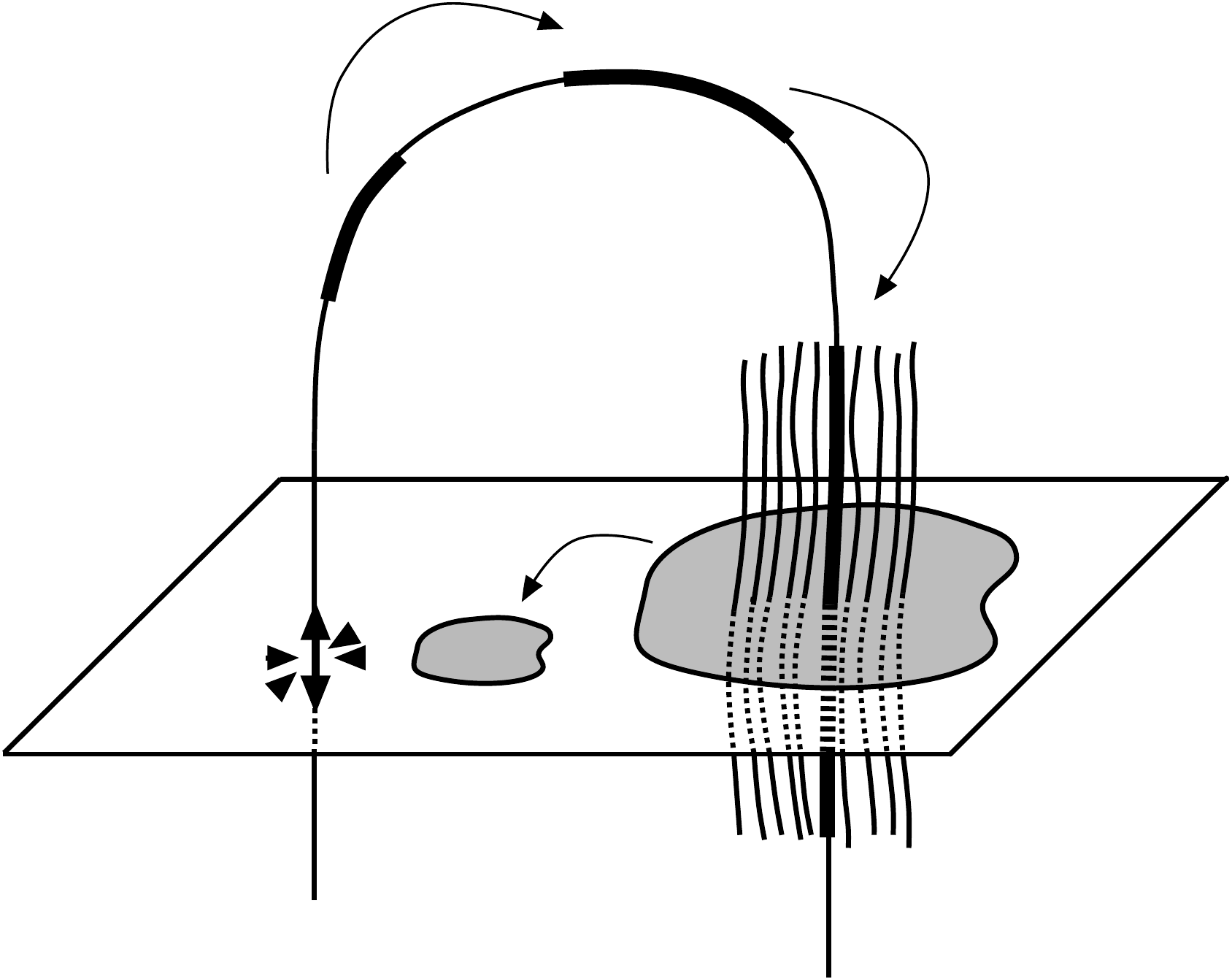}
      
        \put(30,80){\small  $f$}
           \put(77,65){\small  $f$}
         \put(45,38){\small  $f$}

           \put(20,30){\small  $q$}
            \put(60,53){\small  $\De^u$}
              \put(84,35){\small  $\De^s$}
  \put(77,48){\small  $\fD$}

  \end{overpic}
\caption{Flip-flop configuration}
\label{fig.cff}
\end{minipage}
\end{figure}

\subsubsection{Flip-flop configurations are robust}
The aim of this section is the following proposition.

\begin{prop}\label{p.ffrobust}
Consider  a dynamical blender  $(\Ga, V, \cC^\mathrm{uu}, \fD)$  
 and a saddle $q$ of a diffeomorphism $f$ in  a flip-flop configuration. 
 Consider  a validity domain $\cU$ of the blender  such that
the hyperbolic continuations $\Gamma_g$ of $\Ga$ and $q_g$ of $q$ are defined for every $g\in \cU$.
Then
there exists a neighborhood $\cV$ of $f$ contained in $\cU$ such that the dynamical blender 
$(\Ga_g, V, \cC^\mathrm{uu}, \fD)$ and the saddle $q_g$ are
in a flip-flop configuration for every 
$g\in \cV$.
\end{prop}

\begin{proof} 
Let $\De^{\mathrm{s}}$ and $\De^{\mathrm{u}}$ be the stable and unstable connecting set of the configuration and consider an
$\varepsilon$-safe stable connecting set $K$ ($\varepsilon>0$ will be fixed below).

Note that compact parts of the unstable manifold 
$W^\mathrm{u}(q_g,g)$ depend continuously on the diffeomorphism $g$.
Thus condition \eqref{i.FFC0} for $f$ implies that 
a local unstable manifold $W_\mathrm{loc}^\mathrm{u}(q_g,g)$  
of any 
 $g$ $C^1$-close enough to $f$
 contains a disc
 $\De^{\mathrm{u}}_g$ arbitrarily $C^1$-close to $\De^{\mathrm{u}}_f=\De^{\mathrm{u}}$ 
 (these discs depend continuously on $f$).
 Therefore the discs $\De^{\mathrm{u}}_g$ are in the interior of $\fD$, obtaining the robustness of condition \eqref{i.FFC0}.
 Furthermore, the negative iterates of $\De^{\mathrm{u}}_g$ by $g$ remain close to 
the ones of $\De^{\mathrm{u}}$ by $f$, hence they are  disjoint from $\overline V$, proving \eqref{i.FFC1}.

Let $N>0$ as in \eqref{i.FFC2} for $f$.  In the same way as above, for every $g$ close enough to $f$,
 we can choose a continuation $\Delta^{\mathrm{s}}_g$
 of $\De^{\mathrm{s}}$ that is simultaneously  
  contained in a local stable manifold of $q_g$ and in $V$ and whose positive iterates 
$g^j(\Delta^{\mathrm{s}}_g)$ are disjoint from $\partial V$ and
from $\overline V$ for every $j>N$. 
Since the set  $\bigcup_0^{N}g^j(\De^\mathrm{s}_g)$ is compact and disjoint of the compact set 
$\partial V$ we have that  for every $g$ close enough to $f$ the points of $g^j(\Delta^{\mathrm{s}}_g)$ 
cannot return to $V$ after a first exit. In this way we get \eqref{i.FFC2}.  

Note that condition \eqref{i.FFC3} is robust and therefore $\De^{\mathrm{s}}_g$ transverse to the 
cone field $\cC^\mathrm{uu}$. 

For $g$ close to $f$ consider a continuation
$K_g$ of the safety set $K$  contained in $\Delta^\mathrm{s}_g$.
We prove   property \eqref{i.FFC4} with $\varepsilon$ replaced by $\varepsilon/2$:
 the set $K_g$ intersects every 
disc $D\in\fD$ at a point whose distance to $\partial D$ is larger than $\epsilon/2$. 
That is precisely the content of the next lemma. 
\begin{lemm}\label{l.enlarge} 
Let $\cC$ be a cone field of index $i$ and 
$K_0$  a sub-manifold of dimension $n-i$ with boundary 
that is transverse to $\cC$. 
Consider $\varepsilon>0$ and $K$ a compact subset contained the interior of  $K_0$.  
Then there is a $C^1$-neighborhood $\cW$ of $K_0$ 
such that  
any disc $D$ of radius $\varepsilon$
tangent to $\cC$
centered at some point $x\in K$
intersects every sub-manifold in $\cW$.
\end{lemm}

%
%
%

\begin{proof}[Sketch of proof]
It is enough to prove the lemma for some small $\varepsilon>0$.
Note that after shrinking $\varepsilon$, if necessary,
we can assume that
the distance from $K$ to $\partial K_0$ is larger than $\varepsilon>0$.
We can also assume
that the distance between $K_0$ and 
the boundary of any 
disc of radius $\varepsilon$ centered at some point of $K$ and tangent to $\cC$ is larger than 
some $\mu>0$.

Fix any disc $D$ of radius $\varepsilon$  as in the statement of the lemma.
Note that every sub-manifold that 
is sufficiently $C^1$-close to $K_0$ and is transverse to $\cC$ admits an isotopy to $K_0$ by sub-manifolds $K_t\in \cW$.
Thus, for small $t$, the intersection $K_t\cap D$ is just  a point $x_t$ that depends continuously on $t$. 
The unique way of this point ``vanishing'' is to cross either the 
boundary of $D$ or the boundary of $K_t$. We now see that these two possibilities are forbidden.

To discard the first possibility 
 note that the  choices of $\varepsilon$ and $\mu$ above imply that if
the neighborhood $\cW$ of $K_0$ is small enough then 
 $x_t\not\in\partial D$.

To eliminate the second possibility note that
 the choice of $\varepsilon>0$ implies that 
 if $\cW$ is small enough then the distance between
 $K$ to $\partial K_t$ is larger than $\mu/2$ for every $K_t$ and thus
 $x_t\not\in \partial K_t$. This concludes the proof of the lemma. 
\end{proof}

The proof of the robustness of flip-flop configurations is now complete.
\end{proof}

The following presentation of flip-flop configurations will be very useful for us.

\begin{prop}\label{p.fff}
Consider a dynamical blender $(\Ga, V, \cC^{\mathrm{uu}}, \fD)$ 
 of $\mathrm{uu}$-index $i$ of a diffeomorphism $f$.
 Suppose that there are  a periodic point $q$  of $f$ of $\mathrm{u}$-index $i$, 
 a  disc $\Delta^\mathrm{u}$  contained in the unstable manifold 
$W^\mathrm{u}(q)$, 
and a compact sub-manifold with boundary $\De^\mathrm{s} \subset V \cap W^\mathrm{s}(q)$ 
satisfying conditions \eqref{i.FFC1}, \eqref{i.FFC2},
\eqref{i.FFC3}, \eqref{i.FFC4}, and 
\begin{enumerate}[label=FFC\arabic*'),ref=FFC\arabic*',leftmargin=*,widest=FFC0',start=0,format=\textnormal]
\item\label{i.FFC0p}
 $\Delta^\mathrm{u}$ belongs to  the family $\fD$. 
\end{enumerate}
Then there is $\eta>0$ such that $(V,\Ga,\cC^\mathrm{uu}, \cV^\delta_\eta(\fD))$ is a dynamical blender 
in a flip-flop configuration with
$q$ such that   $\Delta^\mathrm{u}$ and $\Delta^\mathrm{s}$ are its unstable and stable connecting sets.
\end{prop}
\begin{proof} 
First, by Lemma~\ref{l.robust}, 
 $(\Ga,V,\cC^\mathrm{uu}, \cV^\delta_\eta(\fD))$  is a dynamical blender of $f$
 for every $\eta>0$ sufficiently small. We now see that this blender and $q$ are in a flip-flop configuration.
  
To check that property \eqref{i.FFC0} holds just note that
condition \eqref{i.FFC0p} implies that $\De^\mathrm{u}$ is a disc in the interior of $\cV^\delta_\eta(\fD)$.

Note that conditions \eqref{i.FFC1}, \eqref{i.FFC2} and \eqref{i.FFC3} remain unchanged and thus there is  nothing to prove.  

Finally, the fact that property \eqref{i.FFC4} still holds for the new blender is a direct consequence of 
Lemma~\ref{l.enlarge}.
\end{proof}

\subsubsection{Flip-flop configurations generate robust cycles}

\begin{prop}\label{p.l.flipflopcycle} 
Consider a dynamical blender 
 $(\Ga,V,\cC^\mathrm{uu},\fD)$
 and a saddle $q$ of a diffeomorphism $f$  in 
 a flip-flop configuration. 
Then $\Ga$ and $q$ form a robust cycle and thus $\Ga$ is contained in the chain recurrence class of $q$.
\end{prop}

\begin{proof}
We first see that $f$ has a cycle associated to $\Ga$ and $q$.
We will see that the invariant manifolds
of $q$ and $\La$ meet cyclically. As $\Ga$ is  transitive then $q$ and $\Ga$ form a cycle.

 Consider the  the connecting sets $\De^\mathrm{s}$ and $\De^\mathrm{u}$ 
 and the $\varepsilon$-safety stable set 
 of this configuration. 
  Note that   $\De^\mathrm{u}\subset W^{\mathrm{u}}(q)$ belongs to the interior of $\fD$ and thus 
 it intersects $W^{\mathrm{s}}(\Ga)$. Hence $W^{\mathrm{u}}(q)\cap W^{\mathrm{s}}(\Ga)\ne\emptyset$.
 
To see that $W^{\mathrm{s}}(q)\cap W^{\mathrm{u}}(\Ga)\ne\emptyset$ note that
 by Scholium~\ref{s.fj} there is a sequence of discs $D_n\in\fD$ converging to a disc $D$ 
 contained in the unstable manifold of some point in $\Ga$. 
Every disc $D_n$ contains a point in the compact set $K\subset \De^s$. 
Thus $\De^s\cap D\neq \emptyset$ and 
hence $W^{\mathrm{s}}(q)\cap W^{\mathrm{u}}(\Ga)\ne\emptyset$. 

 The robustness of the cycle now follows from the robustness of the 
flip-flop configuration (see Proposition~\ref{p.ffrobust}). 
\end{proof}

\subsection{Partially hyperbolic neighborhoods for flip-flop configurations}
\label{ss.partially}

The next lemma is a standard consequence of transversality properties.
Recall first Remark~\ref{r.partially} that claims that a dynamical blender $(\Ga,V,\cC^\mathrm{uu},\fD)$
has a partially hyperbolic splitting
$T_\Ga M= E^\mathrm{uu} \oplus E^\mathrm{c} \oplus E^\mathrm{ss}$,
where $E^\mathrm{uu}$ and $E^\mathrm{c}$ are expanding bundles and
$E^\mathrm{uu}$ is contained in $\cC^\mathrm{uu}$.

\begin{lemm}[Flip-flop configurations and partial hyperbolicity]\label{l.Laranjeiras}  
Consider a diffeomorphism $f$ having a  periodic point $q$  
and a dynamical blender  
$(\Ga,V,\cC^\mathrm{uu},\fD)$ in a flip-flop configuration with connecting sets 
$\De^\mathrm{u}\subset W^\mathrm{u}(q)$  and 
$\De^\mathrm{s}\subset W^\mathrm{s}(q)$.
 
Consider the closed set 
$$
\mathcal{O}(q) \, \cup \,
\overline{V}
\, \cup \,  
\bigcup_{k\ge 0} f^k(\Delta^\mathrm{s}) 
\, \cup \,
\bigcup_{k\le 0} f^k(\Delta^\mathrm{u})
$$
and for every small compact neighborhood $U$ of it
the maximal invariant set $\La(U)$ of $f$ in $U$.
If the neighborhood $U$ is sufficiently small then the set $\Lambda (U)$
 contains the blender $\Gamma$
and has a partially hyperbolic 
splitting of the form
$$
T_{\La(U)} M =
\widetilde E^\mathrm{uu} \oplus \widetilde E^\mathrm{cs},
$$
where  $\widetilde E^\mathrm{uu}$  is 
uniformly expanding and
 extends the sub-bundle $E^\mathrm{uu}$ defined over $\Gamma$ and
$\widetilde E^\mathrm{cs}$ is a dominated bundle that extends the bundle
$E^\mathrm{c} \oplus E^\mathrm{ss}$ defined over~$\Gamma$.

Moreover, there is a strictly $Df$-invariant cone field 
over $U$ that extends the cone field $\cC^\mathrm{uu}$ of the bundle defined $\overline V$
and whose vectors are uniformly expanded by $Df$. 
\end{lemm}

For notational simplicity, we also denote the  extension of the cone field in the lemma by $\cC^\mathrm{uu}$. 
The set $U$ in the lemma 
is called a \emph{partially hyperbolic  neighborhood of the 
flip-flop configuration}.

\begin{proof}[Proof of Lemma~\ref{l.Laranjeiras}]
Since the blender $\Gamma$ is, by definition, the maximal invariant set of $f$ in the set 
$\overline V$, 
it is contained in $\La(U)$.
The existence of the partially hyperbolic splitting over $\La (U)$
is a standard consequence of the transversality  between  $\De^\mathrm{s}$  
and the cone field $\cC^\mathrm{uu}$ defined on $\overline V$, 
and the fact that the disc 
$\Delta^\mathrm{u}$ is tangent to the strong unstable cone field $\cC^{\mathrm{uu}}$ of $\Ga$. 
\end{proof}

If in Lemma~\ref{l.Laranjeiras} we can write 
$\widetilde E^\mathrm{cs}=\widetilde E^\mathrm{c} \oplus \widetilde E^\mathrm{ss}$, where
 $\widetilde E^\mathrm{c}$ 
is one-dimensional and expanding and 
$\widetilde  E^\mathrm{ss}$  is contracting, then
the flip-flop configuration  is called {\emph{split.}} More precisely:

\begin{defi}[Split flip-flop configuration] \label{d.splitflipflop}
Consider a flip-flop configuration of a dynamical blender $(\Ga,V,\cC^\mathrm{uu},\fD)$ and a saddle $q$.
This configuration is {\emph{split}} if it has a partially hyperbolic neighborhood $U$ such that
$T_{\La( U)} M = \widetilde E^\mathrm{uu} \oplus \widetilde E^\mathrm{c} \oplus \widetilde E^\mathrm{ss}$
is a dominated splitting 
where $\widetilde E^\mathrm{c}$ is one-dimensional and uniformly expanding,
 $\widetilde E^\mathrm{uu}$ is uniformly expanding,
 $\widetilde E^\mathrm{ss}$ is uniformly contracting,
and
$\widetilde E^\mathrm{uu} \oplus \widetilde E^\mathrm{c}$
and $\widetilde E^\mathrm{ss}$ 
extend the unstable and stable hyperbolic bundles of $\Gamma$, respectively.

The neighborhood $U$ is called a \emph{strict partially hyperbolic 
neighborhood of the  split flip-flop configuration}. 
\end{defi}

\begin{rema}[Robustness of split flip-flop configurations]\label{r.splitflipflop}
It follows from Proposition~\ref{p.ffrobust} and the persistence of partially hyperbolic splittings that to have
a split flip-flop configuration is a robust property.
\end{rema}

\subsection{Proof of Theorem~\ref{t.practical}}
The aim of this section is to prove 
the criterion for zero center Lyapunov exponents in Theorem~\ref{t.practical}.
Basically, we show that split flip-flop configurations
``embody'' flip-flop families associated to the logarithm of the central Jacobian.
Thus Theorem~\ref{t.practical} follows from the abstract
criterion in Theorem~\ref{t.FlipFlop_weak}.
Actually, the split condition is only used to define the central Jacobian
and the bulk of the work consists in proving the following:

\begin{prop}[Flip-flop families associated to flip-flop configurations]\label{p.flipflopconf}  
Consider a diffeomorphism $f$ having a periodic point
$q$  and a dynamical blender  $\Ga$ in a flip-flop configuration.
Let $U$ be a partially hyperbolic
neighborhood of this flip-flop configuration,
$\La (U)$  the maximal invariant set of $f$ in $U$, and
$\cC^\mathrm{uu}$  the associated strong unstable cone field in $U$.

Suppose $\phi \colon U \to \RR$ 
is a continuous function 
that is positive on  the blender $\Gamma$ and negative on 
the periodic orbit $\cO(q)$ of $q$.

Then there exist an integer $N\ge 1$ and a flip-flop family $\fF$ 
with respect to $f^N$ and the function
$$
\phi_N := \sum_{j=0}^{N-1}\phi \circ f^j 
\quad \text{defined on the set} \quad 
\bigcap_{j=0}^{N-1} f^{-j}(U).
$$

Moreover, given any $\de>0$, we can choose the flip-flop family 
$\fF = \fF^+ \cup \fF^-$ such that 
$\bigcup \fF^+$ (resp.\ $\bigcup \fF^-$) is contained in the $\de$-neighborhood of $\Gamma$
(resp.\ $\cO(q)$).
\end{prop}

We postpone the proof of Proposition~\ref{p.flipflopconf} to Section~\ref{ss.p.flipflopconf}. We now
derive Theorem~\ref{t.practical} from this proposition.

\begin{proof}[Proof of Theorem~\ref{t.practical}]
Suppose $q$ is a saddle of $\mathrm{u}$-index $i-1$
and $\Gamma$ is a dynamical blender of $\mathrm{u}$-index $i$
in a split flip-flop configuration.
Let $U$ be a strict partially hyperbolic neighborhood of this split  configuration
and 
$\widetilde E^\mathrm{uu} \oplus \widetilde E^\mathrm{c} \oplus \widetilde E^\mathrm{ss}$
the corresponding partially hyperbolic splitting defined over
the maximal invariant subset $\Lambda (U)$ of $f$ in $U$. This implies that
$\dim \widetilde E^\mathrm{uu} = i-1$ and $\dim \widetilde E^\mathrm{c} = 1$.

Since $\Gamma \subset \La (U)$ is a hyperbolic set 
whose unstable bundle $E^\mathrm{u}$
is the restriction of $\widetilde E^\mathrm{uu} \oplus \widetilde E^\mathrm{c}$, 
we can take an adapted metric  such that the vectors in $E^\mathrm{u}$
are uniformly expanded by $Df$ (with respect to such a metric).
In particular, the center Jacobian map
$J^\mathrm{c} := \| Df |_{E^\mathrm{c}} \|$, 
which is well defined on $\Lambda (U)$,
is uniformly bigger than $1$ on $\Gamma$.
Since the saddle $q$ has $\mathrm{u}$-index $i-1$,
after a new  change of metric that does not affect the previous one, we can assume that $J^\mathrm{c} < 1$
on the orbit $\cO(q)$.
Applying Tietze extension theorem, 
we continuously extend the function $\log J^\mathrm{c}$ on $\La (U)$
to a function $\phi$ defined on the set $U$.

Fix $\delta>0$ small enough such that
$\varphi$ is defined and positive (resp., negative) on
the $\delta$-neighborhood of $\Gamma$ (resp., of $\cO(q)$).
Applying Proposition~\ref{p.flipflopconf} to the map $\phi$,
we obtain a flip-flop family $\fF$
associated to a power $f^N$ of $f$ and the corresponding Birkhoff sum $\phi_N$.
Applying Theorem~\ref{t.FlipFlop_weak} 
we obtain an $f^N$-invariant compact set 
$\Upsilon \subset U \cap f^{-1}(U) \cap \cdots \cap f^{-N+1}(U)$
such that $h_\mathrm{top} \left(f^N|_\Upsilon\right) > 0$ and 
the Birkhoff averages of $\phi_N$ with respect to $f^N$
converge to zero uniformly on $\Upsilon$.
Then the Birkhoff averages of $\phi$ with respect to $f$
also converge to zero uniformly on the set $\Om \eqdef \bigcup_{j=0}^{N-1} f^j(\Upsilon)$,
which is contained in the maximal invariant set $\La (U)$. 
This means that the Lyapunov exponent  of $f$ along the $E^\mathrm{c}$ direction vanishes on $\Om$. 
Moreover, $h_\mathrm{top} \left(f|_\Om\right)$ is also positive.

To conclude the proof of Theorem~\ref{t.practical} it remains to see that we can choose the set $\Om$ contained in the 
chain recurrent class $C(q,f)$ of $q$. 
First note that we can assume that $\Om$ consists of chain recurrent points.
To see that $\Omega \subset C(q,f)$ fix any point $x\in \Om$.
If the number $\de >0$ in 
Proposition~\ref{p.flipflopconf} is small enough then the orbit of $x$ has infinitely many iterates close to $q$
(otherwise the $\phi$-average would be positive). Hence,  
as $\de$ is small, the strong unstable manifold of $x$ 
intersects the stable manifold of $q$. Similarly, the orbit of $x$ also has infinitely many iterates close to $\Ga$
(otherwise the $\phi$-average would be negative).
Hence, if $\de>0$ is small enough, the strong stable manifold of $x$ intersects the unstable manifold 
of a point in $\Ga$.  

By Proposition~\ref{p.l.flipflopcycle}, the saddle $q$ and the blender $\Ga$ form a robust cycle, 
therefore they are in the same chain 
recurrence class.  Properties
$W^{\mathrm{uu}}(x)\cap W^\mathrm{s}(q)\ne\emptyset$ and 
$W^{\mathrm{ss}}(x)\cap W^\mathrm{u}(\Ga)\ne\emptyset$
 immediately imply that
$x$ belongs to such a  class. 
We have shown $\Omega\subset C(q,f)$, ending the proof of 
Theorem~\ref{t.practical}.
\end{proof}

\subsection{Proof of Proposition~\ref{p.flipflopconf}}
\label{ss.p.flipflopconf}

We divide the proof of Proposition~\ref{p.flipflopconf} into several lemmas.
Consider a diffeomorphism $f$ with a dynamical blender $(\Ga,V,\cC^\mathrm{uu},\fD_\Ga)$ and a saddle $q$ 
in a flip-flop configuration with unstable and stable connecting sets
 $\Delta^\mathrm{u}$ and $\Delta^\mathrm{s}$.
Consider also  a partially hyperbolic neighborhood $U$ of 
this configuration endowed with the extended cone field, that for notational simplicity we denote by $\cC^\mathrm{uu}$.
Let $u$ be the $\mathrm{u}$-index of $q$, by definition the number $u$ is also the
$\mathrm{uu}$-index of $\Ga$ and the index of the cone field $\cC^\mathrm{uu}$.  

Let us fix some ingredients of our construction.
We take a Riemannian metric and a constant $\mu>1$ such  that  for every
$x\in U\cap f^{-1}(U)$ 
the vectors in $\cC^\mathrm{uu}(x)\setminus \{\bar 0\}$ are  
expanded by a ratio at least $\mu$
by $Df$. 

\begin{rema}[Choice of $\de$]\label{r.delta}
Let $\delta>0$ be small enough such  that
the closed $\delta$-neighborhood of $\Ga$ (resp.
 $\{\cO(q)\}$)
is contained in $U$ and the function $\phi$ is bigger than some constant $\alpha_\Gamma>0$
(resp. less than some constant $-\alpha_q<0$)  in such a neighborhood.
Reducing $\delta$, if necessary, we can assume that the local manifold
$W^\mathrm{s}_\delta(f^i(q))$ is contained in $U$ and tangent to $\cC^\mathrm{uu}$.
\end{rema}

In the next lemmas we will introduce auxiliary families of discs that will be used to define the sets of the flip-flop
family. The first step is to define a family $\fD_q$ of discs in the $\delta$-neighborhood of the orbit  $\cO(q)$. 
In this way, we have two preliminary families of discs, the discs  $\fD_\Ga$ in the blender and the discs $\fD_q$.
The second step is to define ``transitions" between these two families.

\begin{lemm}[The family $\fD_q$]\label{l.mfDq}
There is a family  $\fD_q\subset D\in\cD^u(M)$ of $C^1$-embedded discs
of dimension $u$ 
containing the discs $W^\mathrm{u}_{\delta/4}(f^i(q))$ in its interior and consisting of discs
$D \in \fD_q$ that  satisfy the following properties:
\begin{enumerate}
\item $D \subset U$;
\item $D$ is tangent to $\cC^\mathrm{uu}$;
\item\label{i.diam} $\diam D < \delta$;
\item $D$ transversely  intersects $W^\mathrm{s}_\delta(f^i(q))$ for some $i$;
\item $\|Df(v)\| \ge \mu \|v\|$ \
for every vector $v$ tangent to $D$;
\item\label{i.nest} $f(D)$ contains a disc in $\fD_q$.
\end{enumerate}
\end{lemm}
The existence of the family $\fD_q$ follows easily from the strict $Df$-invariance of the cone field 
$\cC^\mathrm{uu}$,  the uniform expansion of  the vectors in $\cC^\mathrm{uu}$ by $Df$, and the fact 
that the discs of $\fD_q$ transversely intersect the local stable manifold of $\cO (q)$. 

\medskip

We now  study  the transitions between the families $\fD_q$ and $\fD_\Ga$. 
The first step is the following preliminary result:

\begin{lemm}\label{l.uniform}
Let $\fK \subset \cD^u(M)$ be a family of $C^1$-embedded compact discs of dimension $u$ that are 
contained in $U$ and tangent to 
$\cC^\mathrm{uu}$. 
Assume that there is  $\varepsilon>0$  such that every disc $D\in \fK$
contains a point $x\in W^\mathrm{s}_\delta(f^i(q))$, for some $i$,  whose distance to the boundary of $D$ is larger 
than $\varepsilon$.

Then there exists $n_0 =n_0 (\fK) \ge 0$ such that 
for every disc $D \in\fK$ and every $n\ge n_0$ the set $f^n(D)$ contains a disc 
$D_1 \in \fD_q$ such that $f^{-i}(D_1) \subset U$  for every $i\in\{0,\dots,n\}$. 
Moreover, if $\fK=\fD_q$ then we can take $n_0 = 0$.
\end{lemm}

\begin{proof}
In the case $\fK=\fD_q$, the assertion of the lemma with $n_0=0$
follows arguing recursively from property~(\ref{i.nest}) in the definition of the family $\fD_q$. 

Next consider an arbitrary family $\fK$ satisfying the hypotheses of the lemma.
Recall that the neighborhood $U$ contains all the  discs of the family $\fD_q$ and also contains the discs 
$W^\mathrm{s}_{\delta}(f^i(q))$.

Consider the family of discs of radius $\varepsilon$ centered at some point in $W^\mathrm{s}_\delta(f^i(q))$ and 
contained in some  disc of $\fK$. 
As these discs are tangent to the cone field $\cC^\mathrm{uu}$, they are  
Lipschitz graphs over the local unstable manifold of $f^i(q)$, where the Lipschitz constant 
depends only on $\cC^\mathrm{uu}$. 
Using the Lambda Lemma we  find an uniform $n_0 = n_0(\fK) \ge 0$ such that
the
image $f^{n_0}(D)$ of any disc $D \in\fK$
 contains a disc 
$D_0$ close enough to $W^\mathrm{u}_{\delta/4}(f^{i+n_0}(q))$ such that $D_0\in \fD_q$
and, moreover, $f^{-i}(D_0) \subset U$  for every $i\in\{0,\dots,n_0\}$.

Applying the first case $\fK=\fD_q$ to the disc $D_0$ we conclude that, for every $n\ge n_0$,
the image $f^{n-n_0}(D_0)$ contains a disc 
$D_1 \in \fD_q$ such that $f^{-i}(D_1) \subset U$  for every $i\in\{0,\dots,n-n_0\}$. 
Thus the disc $D_1$ satisfies the required properties.
\end{proof}

We now study the transition from the family $\fD_q$ to the family $\fD_\Gamma$.

\begin{lemm}[Going from $\fD_q$ to $\fD_\Gamma$]\label{l.N_q} 
There is $N_q>0$ such that 
for every disc $D\in\fD_q$ and every $n\ge N_q$,  the image $f^n(D)$ contains a disc 
$D_1 \in \fD_\Gamma$ such that $f^{-i}(D_1) \subset U$  for every $i\in\{0,\dots,n\}$. 
\end{lemm}

\begin{proof}
The first step is the following claim.


\begin{clai}
 There is $N_q$ such that for every $D_0\in \fD_q$ there is $m\in\{1,\dots, N_q\}$ 
such  that the set $f^m(D_0)$ contains a disc 
$D_1 \in \fD_\Gamma$ with $f^{-i}(D_1) \subset U$  for every $i\in\{0,\dots,m\}$.
\end{clai}
\begin{proof}
Fix $n_0$ such that the connecting disc   $\Delta^\mathrm{u}$  is contained
in $f^{n_0}(W^\mathrm{u}_{\delta/4}(q))$ . There is a neighborhood  $\cV$ of $W^\mathrm{u}_{\delta/4}(q)$
contained  in $\fD_q$  consisting of discs whose image
$f^{n_0}(D)$ contains a disc $D_1$ 
$\epsilon$-close to $\Delta^\mathrm{u}$ in the $C^1$-distance for some small $\varepsilon$.
 By definition of a flip-flop configuration,  $\Delta^\mathrm{u}$ belongs to the interior of the 
family $\fD_\Gamma$, thus if $\epsilon>0$ is small enough the same holds for 
the disc $D_1$.
Moreover, since the negative iterates of $\Delta^\mathrm{u}$ are contained in $U$, by shrinking $\epsilon$ if necessary,
we get that $f^{-i}(D_1) \subset U$ for all $i=0,1,\dots, n_0$.

Applying the Lambda Lemma to the discs in $\fD_q$
and by the definition of $\fD_q$,
we obtain $n_1 > 0$ such that  
for every disc $D_0 \in \fD_q$ there is a sequence $D_{0,i}\in\fD_q$, $i=0,\dots, n_1$, 
such that $D_{0,0}=D_0$, $D_{0,i+1}\subset f(D_{0,i})$, 
 and $D_{0,n_1}\in \cV$.

By construction to prove the claim it is enough to take $N_q=n_0+n_1$.
\end{proof}

Take $N_q$ as in the claim.
Consider $D_0\in\fD_q$ and $n\geq N_q$.  
Associated to $D_0$ consider $m\le N_q$ and the disc $D_1\subset f^m(D_0)$ given by the claim.  
By the strict $f$-invariance of the family
$\fD_\Ga$, there is a sequence 
$D_i\subset f(D_{i-1})$, $i=2, \dots, n+1-m$, such that $D_i\in\fD_\Ga$.  By construction, $D_{n-m+1}$ is contained in $\fD_\Ga$
and
the negative iterates $f^{-j}(D_{n-m+1})$ for $j\in \{0,\dots, n-m\}$ are contained in 
$D_{n-m-j+1}\in \fD_\Ga$, hence contained in $U$. Further negative iterates 
$f^{-j}(D_{n-m+1})$, $j\in \{n-m,\dots, n\}$ are contained in the 
negative iterates $f^{-(j-(n-m))}(D_1)$. By the claim and the choice of $m$ and $D_1$, these backwards iterates of $D_1$
are contained in $U$. This ends the  proof of the lemma. 
\end{proof}

In the next lemma we  study the transition from the family $\fD_\Ga$ to the family $\fD_q$.
 
\begin{lemm}[Going from $\fD_\Gamma$ to $\fD_q$]\label{l.N_Ga} 
There is $N_\Ga>0$ such that for every disc $D\in\fD_\Gamma$   
and every $n\ge N_\Ga$, the set $f^n(D)$ contains a disc 
$D_1 \in \fD_q$ such that $f^{-i}(D_1) \subset U$  for every $i\in\{0,\dots,n\}$. 
\end{lemm}

\begin{proof}
Recall that the stable connecting set  $\Delta^\mathrm{s}$ of the flip-flop configuration  given by \eqref{i.FFC2}  
is contained in $W^\mathrm{s}(q)$, there exists $n_1>0$ such that $f^{n_1}(\Delta^\mathrm{s}) \subset W_{\delta/2}^\mathrm{s}(q)$.
Moreover, since  $\Delta^\mathrm{s}$ is compact and 
by definition of a partially hyperbolic neighborhood of the configuration one has
$\bigcup_{j \ge 0} f^j(\Delta^\mathrm{s}) \subset U$,
there is  a neighborhood $B$ of $\De^\mathrm{s}$  such that 
$\bigcup_{j=0}^{n_1} f^j(B) \subset U$.

Recall also that there are $\varepsilon>0$ and an $\varepsilon$-safe stable connecting set of the configuration, that is a compact
subset $K$ contained in the interior of $\Delta^\mathrm{s}$
such  that every disc $D\in \fD_\Ga$ contains a disc $D_\varepsilon$ of radius $\varepsilon$ centered at some point of $K$. 
By shrinking $\varepsilon$ if necessary, we can assume that the discs $D_\varepsilon$ are contained in $B$.

Consider the family 
$$
\fK\eqdef\{f^{n_1}(D_\varepsilon) \colon D \in \fD_{\Gamma}\}.
$$
This family satisfies the hypotheses of Lemma~\ref{l.uniform}. Thus  let $n_0 = n_0(\fK)$ 
as in this  lemma and define $N_\Gamma\eqdef n_1+n_0$. We claim that this number satisfies the conclusions of the lemma.

Let $n \ge N_\Gamma$ and take any disc $D \in \fD_{\Gamma}$ and
consider $E \eqdef f^{n_1}(D_\varepsilon) \in \fK$.
By Lemma~\ref{l.uniform}, the image $f^{n-n_1}(E)$ contains a disc $E_1 \in \fD_q$
such that $f^{-i}(E_1) \subset U$ for all $i \in \{0,1,\dots,n-n_1\}$.
Since  $f^{-j}(E) \subset U$ for all $i \in \{0,1,\dots,n_1\}$,
the disc $D_1 \eqdef E_1$ satisfies all the required properties.  The proof of the lemma is now complete.
\end{proof}

Let us summarize our constructions up to this point. 
We have defined two families of discs $\fD_q$ and $\fD_\Gamma$ 
such  that one can go from each family to the other
in times $N_q$ or $N_\Ga$ (according to the case and in the sense of Lemmas~\ref{l.N_q} and \ref{l.N_Ga}),
and from each family to itself in time $1$
(property~(\ref{i.nest}) of $\fD_q$ 
and strict $f$-invariance property of $\fD_\Ga$).
Moreover, during these transitions the orbits of these discs remain in
a partially hyperbolic neighborhood  $U$ of the  flip-flop configuration.

Recall now that our goal is to construct a flip-flop family associated to
a Birkhoff sum of the function $\phi$.
This function is bounded away from zero in the discs of $\fD_q$,
but this is not necessarily true for the discs of $\fD_\Ga$.
Thus we need to shrink the discs of the family $\fD_\Gamma$
while keeping the transition properties
between the families above.  
This is the reason why we introduce the new family $\fD^m_\Ga$  below. Let us now go to the details of
this construction.

Recall the choice of $\delta$ in Remark~\ref{r.delta} and that $\phi>\alpha_\Ga>0$ in the $\de$-neighborhood
of the blender $\Ga$.
Since the blender $\Gamma$ is the maximal invariant set of $f$
in $\overline V$,
there is an integer $m \ge 0$ such that 
the set $\bigcap_{i=-m}^m f^i (\overline V)$ is contained in the $\delta$-neighborhood
of $\Gamma$, thus in a region where $\phi>\alpha_\Gamma>0$.

\begin{defi}[The family $\fD_{\Gamma}^m$]\label{d.new_family}
Let $\fD_\Gamma^{m}$ be the family of discs $D^+$ 
such that there exist $E_{-m}$, $E_{-m+1}$, \dots, $E_m \in \fD_\Gamma$
satisfying the following conditions:
\begin{itemize}
\item $E_i \subset f(E_{i-1})$ for each $i\in\{-m+1, -m+2, \dots, m\}$;
\item $E_m = f^m(D^+)$.
\end{itemize}
In particular, $f^i(D^+) \subset E_i$ for each $i\in\{-m, -m+1, \dots, m\}$.
\end{defi}

\begin{rema}\label{r.containedDm}
Every disc $D\in \fD_\Ga^m$ is contained in the $\delta$-neighborhood of $\Ga$, thus
$\phi(x)>\alpha_\Ga$ for all $x\in D$.
\end{rema}

Using recursively the strict $f$-invariance of the family $\fD_\Gamma$,
we get  that 
the image $f^m(D)$ of any disc $D \in \fD_\Gamma$ contains a disc of $\fD_\Gamma^{m}$.
In particular, the family $\fD_\Gamma^{m}$ is nonempty.

Let $C:=\sup |\phi|$ and recall that $\alpha_q, \alpha_\Ga>0$ (see Remark~\ref{r.delta}).
Fix an integer 
$$
N \ge \max(N_q,N_\Gamma) + m
$$
 such that the number
$$
\alpha \eqdef \min 
\left\{ (N - N_q - m ) \alpha_q  - (N_q + m) C , ( N- N_\Gamma - m ) \alpha_\Gamma -  (N_\Gamma + m) C \right\}
>0.
$$

In the next two lemmas we obtain  transitions 
between the families $\fD_q$ and $\fD^m_\Ga$: from each family it is possible to go to the other family and to itself.
In what follows, we denote by $D^-$ the discs in $\fD_q$ and by $D^+$ the ones in $\fD^m_\Ga$.

\begin{lemm}[Transitions of discs of $\fD_q$]\label{l.D-} 
Every disc $D^- \in \fD_q$ contains   
subdiscs $D^-_-$, $D^-_+$ such that:

\begin{enumerate}[label=$\mathrm{Tq}\arabic*)$,ref=Tq\arabic*,leftmargin=*,widest=Tq2]
\item\label{i.T1} 
	$f^N(D^-_-)\in \fD_q$ and $f^i(D^-_-) \subset U$ 
	for every $i\in\{0,1,\dots, N-1\}$. 
	Moreover,  
	$$
	\phi_N(x) < - N \alpha_q  \le  -\alpha, \quad \mbox{for every} \quad  x\in D^-_- \, .
	$$
	\item \label{i.T2}
	$f^N(D^-_+)\in\fD_\Gamma^m$ and $f^i(D^-_+) \subset U$ 
	for every $i\in\{0,1,\dots, N-1\}$. 
	Moreover,
	$$
	\phi_N(x) <  - ( N- N_q - m ) \alpha_q  + (N_q + m) C \le - \alpha,
	\quad \mbox{for every} \quad  x\in D^-_+ \, .
	$$
\end{enumerate}
\end{lemm}

\begin{proof}	
Take any disc $D^-$ of the family $\fD_q$.
Applying recursively property (\ref{i.nest}) of the family $\fD_q$
we find a sequence of discs $(D_i)_{i\ge 0}$ in $\fD_q$ 
such that  $D_0 = D^-$ and  $D_{i+1}\subset f(D_i)$ for each $i$.  
Let
$$
D^-_- \eqdef f^{-N}(D_N).
$$
By construction and the definitions of $\fD_q$ and $\alpha_q$, the disc $D^-_-$ satisfies the properties
in \eqref{i.T1}.

To prove  the second part of the lemma, recall that
by Lemma~\ref{l.N_q}, the set $f^{N_q}(D_{N-N_q-m})$ contains a disc 
$F \in \fD_\Gamma$ such that $f^{-j}(F) \subset U$ for every $j \in \{0,\dots,N_q\}$.
Applying recursively the strict $f$-invariance of the family $\fD_\Ga$,
we get a sequence of discs $(F_i)_{i\ge 0}$ in $\fD_\Gamma$ 
such that $F_0 = F$ and $F_{i+1}\subset f(F_i)$ for each $i$.  
Let
$$
D^-_+ \eqdef f^{-N-m}(F_{2m}) \, .
$$
Notice that $f^N(D^-_+) = f^{-m}(F_{2m})$ belongs to the family $\fD_\Gamma^m$.
Indeed the associated sequence of discs $E_{-m}$, \dots, $E_m$
in Definition~\ref{d.new_family} is given by $E_i = F_{m+i}$.
The inclusion properties and the upper bound for the sum $\phi_N$ in \eqref{i.T2} follow straightforwardly by construction.
The proof of the lemma is now complete.
\end{proof}

\begin{lemm}[Transitions of discs of $\fD_\Ga$] \label{l.D+} 
Every disc $D^+ \in \fD_\Gamma^{m}$ contains   
subdiscs $D^+_+$, $D^+_-$ such that:
\begin{enumerate}[label=$\mathrm{TG}\arabic*)$,ref=TG\arabic*,leftmargin=*,widest=Tq2]
\item\label{i.TG1} 
	$f^N(D^+_+)\in \fD_\Gamma^{m}$ and $f^i(D^+_+) \subset U$ 
	for every $i\in\{0,1,\dots, N-1\}$. 
	Moreover,  
	$$
	\phi_N(x) > N \alpha_\Gamma \ge \alpha, \quad \mbox{for every} \quad  x\in D^+_+ \, .
	$$
	\item \label{i.TG2}
	$f^N(D^+_-)\in\fD_q$ and $f^i(D^+_-) \subset U$ 
	for every $i\in\{0,1,\dots, N-1\}$. 
	Moreover,
	$$
	\phi_N(x) > ( N- N_\Gamma - m ) \alpha_\Gamma -  (N_\Gamma + m) C \ge \alpha,
	\quad \mbox{for every} \quad x\in D^+_- \, .
	$$
\end{enumerate}
\end{lemm}

\begin{proof}
Given a disc $D^+$ in the family $\fD_\Gamma^{m}$, 
let $E_{-m}$, $E_{-m+1}$, \dots, $E_m$ be the discs associated to $D^+$   in $\fD_\Gamma$ 
given by Definition~\ref{d.new_family} with $E_m=f^m(D^+)$.
Applying recursively the $f$-invariance property of family $\fD_\Gamma$, 
we find new discs $E_{m+1}$, $E_{m+2}, \dots \in \fD_\Gamma$
such that $E_i \subset f(E_{i-1})$ for each $i > m$.
Notice that by construction the disc  $f^{-m}(E_j)$ is a member of $\fD_\Gamma^m$
for every $j \ge m$.

Let
$$
D^+_+ \eqdef f^{-N-m}(E_{N+m}) \, .
$$
Note that 
$$
D^+_+ \subset  
f^{-N-m}(f^N(E_m))= f^{-m}(E_m)=D^+
$$
and 
that
$$
f^N(D^+_+) = f^{-m} (E_{N+m}) \in \fD_\Gamma^m
$$
as required. To see that the sets $f^i(D^+_+)$ are contained in $U$ for all
 $i\in\{0,1,\dots, N\}$ note that
$f^i(D^+_+)$ is contained in $f^{-m}(E_{m+i}) \in \fD_\Gamma^m$ thus contained in $U$.
By Remark~\ref{r.containedDm} this implies that 
$\phi(f^i(x))>\alpha_\Gamma$ for every $x\in D^+_+$ and therefore 
$\phi_N(x)> N\alpha_\Ga$ for every $x\in D^+_+$. This ends the proof of \eqref{i.TG1}.

As $E_{N-N_\Ga}\in D_\Ga$,
Lemma~\ref{l.N_Ga} implies that
the disc $f^{N_\Gamma}(E_{N-N_\Gamma})$
contains a disc $F \in \fD_q$ such that $f^{-j}(F) \subset U$
for every $j \in \{0,\dots,N_\Gamma\}$.
Let
$$
D^+_- \eqdef f^{-N}(F) \, .
$$
Notice that $D^+_- \subset f^{-N+N_\Gamma}(E_{N-N_\Gamma})\subset D^+$
(since  $E_m=f^m(D^+)$ and $N-N_\Gamma \ge m$).
It is clear that $f^i(D^+_-) \subset U$ for every $i\in\{0,1,\dots, N\}$. This completes the proof of
the inclusion properties.

To get the estimate for the Birkhoff sum note that
for each $i\in\{0,1,\dots, N-N_\Gamma-m\}$,
the set
$f^i(D^+_-)$ 
is contained in $f^{-m}(E_{m+i}) \in \fD_\Gamma^m$ and in particular, by Remark~\ref{r.containedDm}, 
in the part of $U$ where $\phi>\alpha_\Gamma$.
So the lower bound for the Birkhoff sum $\phi_N$ on $D^+_-$ follows. The proof of \eqref{i.TG2} is now 
complete.
\end{proof}

\subsubsection{End of the proof of Proposition~\ref{p.flipflopconf}}
Let $N$  and $\alpha$ be as above.
The flip-flop family 
$\fF = \fF^+ \cup \fF^-$ is defined as follows:
\begin{itemize}
\item 
$\fF^+$ is the family of unions 
$\{D^+_+ \cup D^+_-\}_{D^+ \in \fD_\Gamma^m}$, where $D^+_+$ and $D^+_-$ 
are associated to $D^+ \in \fD_\Gamma^m$ and given by Lemma~\ref{l.D+};
\item 
$\fF^-$ is the family of unions 
$\{D^-_+ \cup D^-_-\}_{D^-\in \fD_q}$, where $D^-_+$ and $D^-_-$ 
are associated to $D^- \in \fD_q$ and given by Lemma~\ref{l.D-}.
\end{itemize}
Lemmas~\ref{l.D+} and \ref{l.D-} 
provide properties (\ref{i.FF1}) and (\ref{i.FF2}) of a flip-flop family
with respect to $f^N$ and the function $\phi_N$.
Property  (\ref{i.FF3}) follows from the fact that
the discs in $\fF$ are tangent to $\cC^\mathrm{uu}$. 

Finally, the construction implies that 
the discs of
$\fF^+$  are contained  in the $\de$-neighborhood of $\Gamma$
and the discs of 
$\fF^-$ are contained  in the $\de$-neighborhood of $\cO(q)$. 
The proof of the proposition is now complete.
\hfill \qed


\section{From robust cycles to
 split flip-flop configurations via spawners}\label{s.spawners}

In this section we see how robust cycles generate flip-flop configurations. This generation is done
throughout a special class of partially hyperbolic sets called 
\emph{spawners} that we will introduce in the next subsection.
The advantage of spawners for us
is that they spawn split flip-flop configurations.
The organization of this section is the following:
$$
\mbox{Robust cycles} \,\, \xRightarrow{{\rm{Prop.~\ref{p.caixapreta}}}} \,\, \mbox{Spawners} \,\,
 \xRightarrow{{\rm{Prop.~\ref{p.spawn}}}} \, \mbox{Split flip-flop configurations}.
$$
The corresponding steps are done in Sections~\ref{ss.cyclestospawners} and \ref{ss.spawnersflipflop}.

\subsection{From robust cycles to spawners}
\label{ss.cyclestospawners}
Given a natural number $i$,  an {\emph{$i$-box}} is  a product of 
$i$ non-degenerate compact intervals.

\begin{defi}[Spawner] 
Let $f \in \Diff^1(M)$ and $u$ and $s$ be positive integers with $u+1+s=d=\dim M$.	
Suppose $C \subset M$ is an embedded $d$-dimensional cube.
For  notational simplicity, let us identify $C$ with $[-1,1]^d$.
Suppose there are disjoint subsets $L_1, L_2, L_3 \subset C$ of the form 
$$
L_i = I_i^{u} \times [-1,1] \times  [-1,1]^s \, , 
$$ 
where each $I_i^{u} \subset (-1,1)^u$ is a $u$-box, 
and positive integers $n_1, n_2, n_3$ such that:
\begin{itemize}
\item 
$f^j(L_i) \cap C = \emptyset$ 
for $0<j<n_i$ and 
$f^{n_i}(L_i)\subset C$;
\item
$f^{n_i}(L_i) = [-1,1]^u \times [-1,1] \times I_i^{s}$, where $I_i^{s}$ is
an $s$-box;
\item 
the restriction of $f^{n_i}$ to $L_i$ is of the form
$$
f^{n_i}(x^\mathrm{u},x,x^\mathrm{s})=(A^\mathrm{u}_i(x^\mathrm{u}),x,A^\mathrm{s}_i(x^\mathrm{s})),
$$
where $A^\mathrm{u}_i$ is an expanding affine map of $\RR^u$ and $A^\mathrm{s}_i$ is a contracting affine map of $\RR^s$. 
\end{itemize}
Let $\Sigma$ be the maximal invariant set of $f$ in the set
\begin{equation} \label{e.neighV}
Q_{123} \eqdef \bigcup_{i=1,2,3}  \Big( \bigcup_{k=0}^{n_i-1} f^k(L_i) \Big) \, .
\end{equation}
The set $\Sigma$ is called a \emph{spawner of $\mathrm{u}$-index $u$},
the set $C$ is its \emph{reference cube}, and 
the sets $L_1$, $L_2$ and $L_3$ are its \emph{legs}. 
The numbers $n_1,n_2,n_3$ are the \emph{first return times of the legs}.
(See the first part of Figure~\ref{fig.spawner} in Section~\ref{sss.perturbationsinduced}).
\end{defi}

Note that each set $f^{n_i}(L_i)$ intersects the legs $L_1, L_2,L_3$ in a Markovian way
and that the set $\Sigma$ is partially hyperbolic with a splitting of the form
$T_\Sigma M = E^\mathrm{uu} \oplus E^\mathrm{c} \oplus E^\mathrm{ss}$
with bundles of respective dimensions $u$, $1$, and $s$, where 
$E^\mathrm{uu}$ is uniformly  expanding and $E^\mathrm{ss}$ is uniformly contracting.
In the cube $C$ these three bundles are  of the form
$\RR^u \times\{0\}^{s+1}$,  $\{0\}^u \times \RR \times\{0\}^s$, 
and $\{0\}^{u+1} \times \RR^s$, respectively. 
Note that the strong stable manifold $W^\mathrm{ss}(\Sigma)$ 
and the strong unstable manifold $W^\mathrm{uu}(\Sigma) $ of $\Sigma$ are well defined.

The spawners that we consider are generated  by means of the following result:

\begin{prop}[From robust cycles to spawners]\label{p.caixapreta}
Consider a diffeomorphism $f$ with a pair of hyperbolic basic sets $\Lambda_f$ and $\Theta_f$
of respective $\mathrm{u}$-indices $i$ and $i-1$ forming a robust cycle.
Then there exists a diffeomorphism $g$ arbitrarily $C^1$-close to $f$ 
 having  a spawner $\Sigma$ of $\mathrm{u}$-index $i-1$
such that
\begin{equation}\label{e.transv}
W^\mathrm{u}(\Lambda_g) \mathrel{\pitchfork} W^\mathrm{ss}(\Sigma) \neq \emptyset
\quad \text{and} \quad
W^\mathrm{s}(\Theta_g)  \mathrel{\pitchfork} W^\mathrm{uu}(\Sigma) \neq \emptyset,
\end{equation}
i.e.\ there exist points of transverse intersection between these manifolds.
\end{prop}

This proposition follows from a sequence of previous results and indeed is a reformulation of results in 
\cite{BD-cycles}. For completeness, 
let us briefly explain the steps involved in this construction.

\begin{proof}[Sketch of the proof of Proposition~\ref{p.caixapreta}]
First, the  existence of a robust cycle implies that there is a $C^1$-neighborhood $\cV$ of $f$ such that 
for every $g\in \cV$ there are
saddles $p_g\in \Lambda_g$ and $q_g\in \Theta_g$ 
depending continuously on $g$ such that
their 
chain recurrence classes  are equal and non trivial.

By \cite{BC}
$C^1$-generically the chain recurrence class of a hyperbolic periodic point is its homoclinic class.
This fact allows us to get an open and dense subset of the neighborhood $\cV$ of $f$ consisting of diffeomorphisms $g$
such that homoclinic classes of $p_g$ and $q_g$ are both 
non trivial. Next, after a new perturbation, if necessary,
 we can replace these saddles  by a pair of saddles homoclinically related to them
whose eigenvalues are all real and have multiplicity $1$ and different moduli (i.e., the linear map $Df^\pi (a)$, $\pi$ the period of $a$, satisfies  such 
a property). For this standard  property see for instance \cite{ABCDW}. Using the terminology in 
\cite{BD-cycles,BD-tang} we say that these new periodic points have real center eigenvalues. As these new saddles are homoclinically related to the initial ones, they are still in the same  chain recurrence class and contained in a pair of hyperbolic sets with a robust cycle.
For simplicity, we continue to denote these  new saddles by $p_g$ and $q_g$. 

In the above setting, \cite[Theorem 2.3]{BD-cycles} claims  that by an arbitrarily small $C^1$-perturbation one can get a
saddle-node or a flip periodic point   with a 
strong homoclinic intersection (called \emph{strong connection}):
the strong stable and strong unstable manifolds of the saddle-node/flip point meet quasi-transversely, meaning that the sum of the tangent
spaces at the intersection point is $\dim M -1$. Furthermore,  \cite[Proposition 5.9]{BD-tang} asserts that  the strong stable (resp.\  unstable) manifold of such  nonhyperbolic periodic point intersects transversely the unstable manifold of $p_g$ (resp.\ the stable manifold of $q_g$). 
With the terminology in  \cite{BD-tang}, this nonhyperbolic periodic  point is
 \emph{strong  intermediate}. As $p_g$ and $q_g$ are robustly in the same chain recurrence class, the 
 strong intermediate point $r$ also belongs to this class. 
 
Finally, \cite[Theorem 2.4]{BD-cycles} shows that these strong homoclinic intersections and intermediate points  yield geometric blenders and robust cycles. To get such a property in 
\cite[pages 501, 502]{BD-cycles}  it is shown that an arbitrarily small $C^1$-perturbation of a strong 
connection associated to a saddle node generates a dynamical configuration that is
exactly what we call here a spawner. A key point is that
 this spawner is by construction strong intermediate with relation to $p_g$ and $q_g$,  thus it is contained in the
 chain recurrence  class of $p_g$. The case of flip periodic points is solved using 
\cite[Remark 4.5]{BD-cycles} that asserts that the strong connection associated 
to a flip periodic point generates a strong connection associated to a saddle node satisfying the strong intermediate properties.
This concludes our sketch of proof.
\end{proof}

\subsection{From spawners to split flip-flop configurations: proof of Theorem~\ref{t.main}}
\label{ss.spawnersflipflop}

Recall the definitions of the sets $L_i$ and the numbers $n_i$, and
define the sets
\begin{equation}
\label{i.Q}
Q_{12} \eqdef  \bigcup_{i=1,2} \bigcup_{k=0}^{n_i-1} f^k(L_i) \quad \mbox{and} \quad
Q_{3}   \eqdef \bigcup_{k=0}^{n_i-1} f^k(L_3)\,.
\end{equation}

\begin{prop}[From spawners to split flip-flop configurations]\label{p.spawn}
Suppose $f$ has a spawner $\Sigma$ of $\mathrm{u}$-index $u$.
Then every neighborhood $\cV$ of $f$ contains a nonempty open set $\cU \subset \cV$ 
such that  every $g \in \cU$ has
\begin{itemize}
\item  
a dynamical blender $\Ga_g$ of $\mathrm{uu}$-index $u$  whose  domain is contained in $Q_{12}$,
\item  
 a unique hyperbolic periodic orbit $\cO(r_g)$ of $\mathrm{u}$-index $u$ contained in $Q_3$,
\end{itemize}
that form a split flip-flop configuration.

Moreover, this split flip-flop configuration has a strict partially hyperbolic neighborhood $U$
contained in the domain of the spawner.  
\end{prop}

We postpone the proof of this proposition to the next subsection and prove now Theorem~\ref{t.main}.

\subsubsection{Proof of Theorem~\ref{t.main}}
Let $\cU \subset \Diff^1(M)$ be an open set of diffeomorphisms  
such that every $f \in \cU$ has a pair of hyperbolic periodic points
$p_f$ and $q_f$ that depend continuously on $f$, have respective $\mathrm{u}$-indices
$i_p > i_q$, and are in the same chain recurrence class $C(q_f,f)$.

In the introduction 
we recalled that there is a $C^1$-dense open subset $\cU_0$ of  $\cU$ such that for every $f\in \cU_0$ and every number
$i_q\leq i\leq i_p$ there is a family of hyperbolic transitive sets $\La_{i,f}$ of $\mathrm{u}$-index $i$ depending continuously on
$f$, contained in $C(q_f,f)$, and such that  that for every $i<i_p$  the sets $\La_{i,f}$ and $\La_{i+1,f}$ form a robust cycle. 
Thus, without loss of generality, we can assume that $i_q=i$, $i_p=i+1$, and that $p_f$ and $q_f$ belong to hyperbolic transitive sets 
$\La_{i+1,f}$ and  $\La_{i,f}$, respectively,
forming a robust cycle. 

By Proposition~\ref{p.caixapreta} there is a $C^1$-dense subset  $\cD\subset \cU$ such that every $f\in \cD$
has  a spawner of $\mathrm{u}$-index $i$ such that the strong unstable  (resp.\ stable) manifold of any point of the spawner
interesects transversally the stable manifold of $\La_{i,f}$ (resp.\ unstable manifold of $\La_{i+1,f}$). 

By Proposition~\ref{p.spawn}, given any diffeomorphism $f\in \cD$ (with a spawner) 
there is an arbitrarily small $C^1$-perturbation $g$ of 
it $f$  with  
 dynamical blender $\Ga_g$ 
of $\mathrm{uu}$-index $i$ and a saddle $r_g$ of  $\mathrm{u}$-index $i$ 
forming  a split flip flop configuration. Moreover, this split flip-flop configuration
has a strict partially hyperbolic neighborhood $U$
contained in the domain of the spawner. 
Recall  that to have a split flip-flop configuration is a robust property, see Remark~\ref{r.splitflipflop}, thus such diffeomorphisms $g$
form a dense open subset $\cU_1$
of $\cU_0$ hence of $\cU$. 

This implies, in particular, that the maximal invariant set in $U$ is contained in the chain recurrence class $C(q_g,g)$. Theorem~\ref{t.practical} now
implies that the maximal $g$-invariant set in $U$ contains a partially hyperbolic set
$K_g\subset C(q_g,f)$ with a partially hyperbolic splitting  of the form
	$$
	T_K M  = E^\mathrm{uu} \oplus E^\mathrm{c} \oplus E^\mathrm{ss},
	$$
	where $E^\mathrm{uu}$ is uniformly expanding and has dimension $i-1 > 0$, 
		$E^\mathrm{c}$ has dimension~$1$, and
		$E^\mathrm{ss}$ is uniformly contracting, such that 
the Lyapunov exponent of any point of $K_g$ along $E^\mathrm{c}$ is zero. Moreover, 
the topological entropy of the restriction of $g$ to $K_g$ is positive.
This concludes the proof Theorem~\ref{t.main}.
\hfill \qed

\subsection{Proof of Proposition~\ref{p.spawn}}

The aim of the rest of the section is the proof of Proposition~\ref{p.spawn}.

\subsubsection{A family of perturbations of a spawner and their induced maps}
\label{sss.perturbationsinduced}
We now assume that $f$ has a spawner with reference cube $C$, legs $L_1,L_2,L_3$ 
and  first return times to  $C$ $n_1,n_2, n_3$,
respectively. By definition, the restriction of $f^{n_i}$ to $L_i$ is of the form
$$
f^{n_i}(x^\mathrm{u},x,x^\mathrm{s})=(A^\mathrm{u}_i(x^\mathrm{u}),x,A^\mathrm{s}_i(x^\mathrm{s}))\,.
$$
For every $i\in\{1,2,3\}$ and $k\in\{1,\dots, n_i-1\}$, we fix  coordinates in $f^k(L_i)$ such  that the expression of the restriction  
$f^k|_{L_i}$ in these coordinates is the identity map, $f^k(x^\mathrm{u},x,x^\mathrm{s})=(x^\mathrm{u},x,x^\mathrm{s})$.

For $\lambda>1$, define one-dimensional maps
$$
g_{\lambda,1}(x) \eqdef \lambda x +(-1+\lambda)/2; \qquad
g_{\lambda,2}(x) \eqdef \lambda x +(1-\lambda)/2; \qquad
g_{\lambda,3}(x) \eqdef \lambda^{-1} x \,.
$$
Note that the maps $g_{\lambda,1}, g_{\lambda,2}, g_{\lambda,3}$
have fixed points $-1/2, 1/2, 0$, respectively.  

Consider a neighborhood $\cV$ of $f$. 
For $\lambda>1$ sufficiently close to $1$ define a
diffeomorphism $f_\lambda$ as follows,
$$
f_\lambda(x^\mathrm{u},x,x^\mathrm{s}) =
\begin{cases} 
&f (x^\mathrm{u},x,x^\mathrm{s}), \,\, \mbox{if $(x^\mathrm{u},x,x^\mathrm{s})\in
\bigcup_{i=1}^3\, \big( \bigcup_{k=1}^{n_i-2} f^k(L_i) \big)$}, 
\\
&(A_i^\mathrm{u}(x^\mathrm{u}),g_{\lambda,i}(x), A_i^\mathrm{s}(x^\mathrm{s})),
\,\,
 \mbox{if }(x^\mathrm{u},x,x^\mathrm{s}) \in f^{n_i-1}(L_i), \ i \in \{1,2,3\}\,.
\end{cases}
$$
We 
take $\lambda>1$ sufficiently close to $1$
such  that $f_\la \in \cV$.
See Figure~\ref{fig.spawner}.

\begin{figure}[hbt]
\begin{minipage}[c]{\linewidth}
\centering
\vspace{0.5cm}
\begin{overpic}[scale=.45, 
  ]{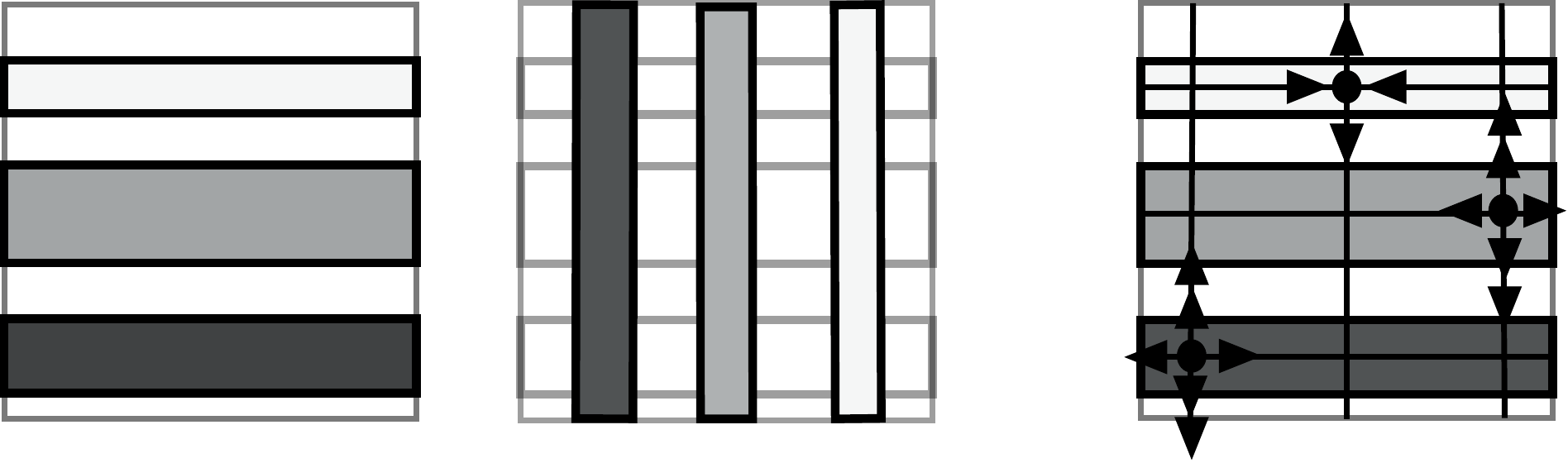}
      
        \put(-5,23){\small  $L_3$}
         \put(-5,15){\small  $L_2$}
 \put(-5,6){\small  $L_1$}
 \put(30,32){\tiny  $f^{n_1}(L_1)$}
 \put(42,31){\tiny  $f^{n_2}(L_2)$}
 \put(54,32){\tiny  $f^{n_3}(L_3)$}
 \put(22,-3){\small  $us$-projection}
\put(76,-3){\small  $uc$-projection}
  \end{overpic}

\medskip
\caption{The maps $f$ and $f_\la$}
\label{fig.spawner}
\end{minipage}
\end{figure}
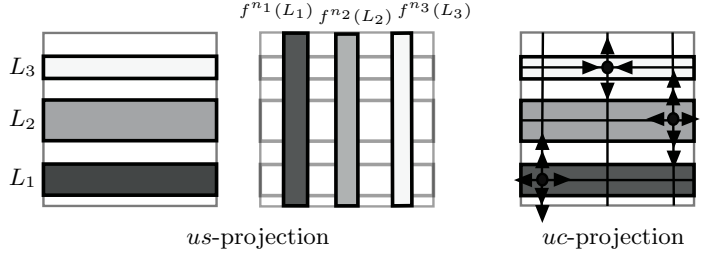

\begin{rema}\label{r.q}
The map $f_\lambda$ has a periodic point $q=(q^\mathrm{u},0,q^\mathrm{s})$ (independent of $\lambda$) in $L_3$, of period $n_3$ 
and $\mathrm{u}$-index $u$.
The local invariant manifolds of $q$ are $W^\mathrm{s}_\mathrm{loc}(q)=\{q^\mathrm{u}\}\times [-1,1]^{s+1}$ and $W^\mathrm{u}_\mathrm{loc}(q)=[-1,1]^u\times \{(0,q^\mathrm{s})\}$. 
\end{rema}

\subsubsection{Dynamical blenders for induced maps}
\label{sss.blendersforinduced}

We fix a small open neighborhood $U$ of $C$ such that, for $\lambda>1$ sufficiently close to $1$,  
$f_\lambda^j(L_i)$ is disjoint from $\overline U$ for every
$0<j<n_i$, and $f_\lambda^{n_i}(L_i)$ is contained in $U$. 
Define the map  $F_\lambda$  by
$$
F_\lambda\colon L_1\cup L_2\to U, \quad F_\lambda(x) \eqdef f_\lambda^{n_i}(x),
\quad \mbox{if $x\in L_i$, $i=1,2$.}
$$ 
In this way we get an induced map  
$F_\la =[f_\la, U, (L_i)_{i=1}^2, (n_i)_{i=1}^2]$
(or for short simply $F_\lambda$) of $f_\la$.
The next step is to get a dynamical blender for this induced map (see Proposition~\ref{p.l.dynamicalinducedblender}). 
For that we need to introduce some ingredients
as domains, cone fields, and families of discs.

Let  $\Ga_\lambda$ be the maximal invariant set of $F_\lambda$ in $L_1\cup L_2$. Note that this set 
is hyperbolic and contained in the interior of $L_1\cup L_2$. 
For every $0<\alpha<1$ consider the cone field
$$
\cC^{\mathrm{uu}}_\alpha \eqdef  \{ (v^\mathrm{u},v^\mathrm{c},v^\mathrm{s}) ;\; \|(v^\mathrm{u},v^\mathrm{s})\| \le \alpha \|v^\mathrm{u}\| \}.
$$
As the maps $A_i^\mathrm{s}$ are affine contractions,
the maps $A_i^\mathrm{u}$ are affine expansions, and $\lambda$ is close to $1$, 
any  cone field $\cC^{\mathrm{uu}}_\alpha$ is strictly $DF_\la$-invariant.  We fix constants $0<\alpha_0<\alpha_1<\frac 1{8\sqrt{u}}$ such
that $DF (\cC^{\mathrm{uu}}_{\alpha_1})$ is 
strictly contained
in $\cC_{\alpha_0}^{\mathrm{uu}}$. 

Recall that $I_i^s=A_i^\mathrm{s}([-1,1]^s)$ and $I_i^{u}=(A_i^\mathrm{u})^{-1}([-1,1]^u)$, for $i=1,2,3$. 
We fix compact discs $J^u, J^u_0\subset (-1,1)^u$  whose interiors contain 
$I_1^u\cup I_2^u\cup I_3^u$ 
and such that $J^u$ is contained in the interior of $J^u_0$. 
Similarly, we fix  compact discs $J^s, J^s_0\subset (-1,1)^s$ whose interiors contain 
$I_1^{s}\cup I_2^{s}\cup I_3^{s}$
and such that $J^s$ is contained in the interior of $J^s_0$.

We consider the following set of graphs of $C^1$-maps with Lipschitz constant less than $\alpha_0$:
\begin{itemize}
\item
$\fD$ is the set of graphs of $C^1$-maps $J^u \to  [-\frac14,\frac 14] \times J^s$;
\item
$\fD_1$ is the set of graphs of $C^1$-maps
$(A_1^\mathrm{u})^{-1}(J^\mathrm{u}_0) \to  [-\frac14,\frac 18] \times J^s$;
\item 
$\fD_2$ is the set of graphs of $C^1$-maps
$
(A_2^\mathrm{u})^{-1}(J^\mathrm{u}_0) \to  [-\frac18,\frac 14] \times J^s. $
\end{itemize}

Finally,  choose open discs $U^u$ and  $U^s$ such that 
$$
 J_0^{u} \subset U^u\subset \overline{U^u}\subset (-1,1)^u
 \quad
 \mbox{and}
 \quad
 J_0^{s} \subset U^s\subset \overline{U^s}\subset (-1,1)^s\,
 $$
and consider the sets
\[
\begin{split}
U_1\eqdef  & (A_1^\mathrm{u})^{-1}(U^{u})\times g_{\lambda,1}^{-1}\left(\Big( -\frac34,\frac34  \Big) \right)\times U^{s}\subset L_1, \\ 
U_2\eqdef  &
(A_2^\mathrm{u})^{-1}(U^{u})\times g_{\lambda,2}^{-1}\left( \Big(-\frac34,\frac34 \Big) \right)\times U^{s} \subset L_2\,.
\end{split}
\]
As $U_i\subset L_i$, $i=1,2$,  the iterates $f_\lambda^j(\overline U_i)$,  $j\in\{1,\dots, n_i-1\}$, are disjoint from $\overline U$ and 
$f_\lambda^{n_i}(\overline U_i)\subset U$. Thus the induced map $F_\la$ of $f_\la$ can be extended to $U_1\cup U_2$.

\begin{prop}[Dynamical blenders for $F_\la$]\label{p.l.dynamicalinducedblender}
Let $\Ga_\la$ be the maximal invariant set of $F_\lambda$ in $L_1\cup L_2$.
Then $(\Ga_\la, U_1\cup U_2, \cC^{\mathrm{uu}}_\alpha, \fD_1\cup \fD_2)$ is a dynamical blender of $F_\la$.
\end{prop}

\begin{proof}
The main step of the proof is the following lemma.
\begin{lemm}\label{l.blenderinduced}
Consider  $\la>1$ close enough to $1$ and
the  induced map $F_\la =[f_\la, U, (L_i)_{i=1}^2, (n_i)_{i=1}^2]$.
Then 
the cone field $\cC^{\mathrm{uu}}_{\alpha_1}$ and
the families of discs
$\fD,\fD_1,\fD_2$ satisfies the hypotheses of Proposition~\ref{p.l.blenderinduced}.
\end{lemm}

\begin{proof}
The invariance of the cone field was obtained above. It remains to check that the families of discs satisfy
items \eqref{i.bi_a}, \eqref{i.bi_b}, and \eqref{i.bi_c} in Proposition~\ref{p.l.blenderinduced}.

For item \eqref{i.bi_a}, just note that there  are neighborhoods of  the families $\fD,\fD_1,\fD_2$ formed by discs 
contained in $U$ and tangent to $\cC^{uu}_{\alpha_1}$.

To prove item \eqref{i.bi_b}, note that a disc $D\in \fD$ is a graph over $J^u\supset I_1^{u}\cup I_2^{u}$. We see that 
either the restriction of this graph to
 $I_1^{u}$ belongs to $\fD_1$ or the restriction to
  $I_2^{u}$ is in $\fD_2$. Suppose, by contradiction, that the first case does not hold.
  This means that the center coordinate
of some point is larger than $\frac 18$.  As the graph is $\alpha_0$-Lipschitz 
with $\alpha_0<\frac1{8\sqrt{u}}$ and the diameter of $J^u$ is strictly 
less than $2\sqrt{u}$ one gets that the central coordinates  of this graph are larger than $-\frac 18$. Thus the central part is 
contained in $[-\frac 18,\frac 14]$ and
 thus 
the restriction to  $I_2^{u}$ belongs to $\fD_2$.

It remains to check item \eqref{i.bi_c}, which is an immediate consequence of the following claim:

\begin{claim} 
Fix $\lambda>1$ close to $1$. Then there is $\varepsilon=\varepsilon(\lambda)>0$ such that for 
every disc $D\in \cV^\delta_\varepsilon (\fD_i)$, $i=1,2$,  the set $F_\lambda(D)$
contains a disc in  $\fD$.
\end{claim}

\begin{proof}
We prove the claim for the family $\fD_1$; the proof for the family $\fD_2$ is identical and hence omitted.

We first prove the claim for  discs in $\fD_1$ (next we extend the proof for discs in a neighborhood).
Recall that 
 $F_\lambda= (A_1^\mathrm{u},g_{\lambda,1},A_1^\mathrm{s})$ and that
a disc $D$ in $\fD_1$ is the graph of  map 
$\varphi=(\varphi^\mathrm{c},\varphi^\mathrm{s})
\colon (A_1^\mathrm{u})^{-1}(J^{u}_0) \to [-\frac14,\frac 18] \times J^s$.
 Recalling  the definition of $g_{\lambda,1}$ we get that
$F_\lambda(D)$ is the graph over 
$J^{u}_0$ of a  map $\varphi_*$ given by
$$
x^\mathrm{u}\mapsto 
\Big( g_{\lambda,1} \big(\varphi^\mathrm{c}( A_1^\mathrm{u})^{-1}(x^\mathrm{u}) \big),
A_1^\mathrm{s}\big(\varphi^\mathrm{s}( A_1^\mathrm{u})^{-1}(x^\mathrm{u})\big)
\Big)\,.
$$
Thus
$$
\varphi_* \colon J^u_0 \to
\left[-\frac14+\frac{\lambda-1}2,\frac14-\,\frac{6-3\lambda}8 \right]\times A_1^\mathrm{s}(J^{s}). 
$$
As $\lambda>1$ is close to $1$ and $J^{s}\subset (-1,1)^{s}$, we get that 
$$
\left[-\frac14+\frac{\lambda-1}2,\frac14-\frac{6-3\lambda}8\right]\times A_1^\mathrm{s}(J^{s})
\subset 
\mbox{int} \left( \left[-\frac14,\frac14\right]\times I_1^{s} \right)
\subset
\mbox{int} \left( \left[-\frac14,\frac14\right]\times J^{s}\right)\,.
$$
Furthermore, the disc $D$ is tangent to the cone field 
$\cC^{\mathrm{uu}}_{\alpha_0}$ which is strictly $DF_\lambda$-invariant. Thus the disc 
$F_\lambda(D)$ is the graph 
 over $J^{u}_0$ of the map $\varphi_*$ whose Lipschitz constant is strictly less than $\alpha_0$. Therefore the graph 
 of the restriction of  $\varphi_*$ to $J^u$ is a disc in $\fD$  and contained in $F_\lambda(D)$.
 This completes the proof for discs in $\fD_1$. 

To extend the result to a small neighborhood of $\fD_1$ note that
there is  $\varepsilon>0$ such that  any disc $\widetilde D\in\cV^\delta_\varepsilon(\cD_1)$ is a graph of a map $\varphi$ whose definition domain contains 
$(A_1^\mathrm{u})^{-1}(J^u)$, with Lipschitz constant $\alpha_1$, and image 
in $g_{\lambda,1}^{-1}([-\frac 14,\frac 14])\times J_0^\mathrm{s}.$
The image by $F_\lambda$ of $\widetilde D$ contains a graph over $J^{u}$ of a
$C^1$-map with Lipschitz constant $\alpha_0$ whose
image is contained in $[-\frac 14,\frac 14]\times J^\mathrm{s}$, thus $F_\la (\widetilde D)$ contains a disc in $\fD$. 
This ends the proof of the claim for the family $\fD_1$.
\end{proof}
The proof of Lemma~\ref{l.blenderinduced} is now complete.
\end{proof}

We are now ready to finish the proof of the proposition.
By construction, the maximal invariant set $\Ga_{F_\lambda}$ of $F_\lambda$ in $L_1\cup L_2$ is  
transitive, hyperbolic, and contained in the interior $U_1\cup U_2$. 
Note that the discs of the family $\fD_i$ are contained in the interior of the sets $U_i$, $i=1,2$,
and satisfies the $F_\la$-invariance properties.
The  cone field $\cC^{\mathrm{uu}}_{\alpha_1}$ is also invariant.
The proof is now complete.
\end{proof}

\subsubsection{Dynamical blenders for $f_\lambda$}

By Proposition~\ref{p.l.blenderinduced}, every neighborhood of the closure of  
$\bigcup_{i=1,2} \bigcup_{j=0}^{n_i-1}f^j_\lambda(U_i)$ contains a safety 
neighborhood $V$ which is the domain of a dynamical blender $\La_\lambda$  of $f_\lambda$ 
whose strictly invariant family of discs $\fD_\lambda$ is such
that the discs $D\in\fD_\lambda$ contained in $U$ are precisely $\cD_1\cup\cD_2$. We can in particular assume that
$\overline V$ is contained in $Q_{1,2}$. 

\begin{coro}[Dynamical blender for $f_\la$]
The diffeomorphism $f_\la$ has a 
dynamical blender
$(\La_\lambda, V, \cC^{\mathrm{uu}}_{\alpha_1},\fD_\lambda)$
such that
the discs of  $\fD_\lambda$ contained in $U$ belong to the family $\cD_1\cup\cD_2$ and $V$ is contained in $Q_{1,2}$.
\end{coro}
%

\subsubsection{Generation of split flip-flop configurations}
Recall the definition of the saddle $q$ of $\mathrm{u}$-index $u$ of $f_\la$, see Remark~\ref{r.q}.

\begin{prop}\label{p.unicamp}
For every $\la>1$ close enough to $1$,
 the saddle 
$q$  and the dynamical blender $(\Ga_\lambda,V, \cC^{\mathrm{uu}}_{\alpha_1},\fD_\lambda)$ satisfy  conditions \eqref{i.FFC0p}, \eqref{i.FFC1}, \eqref{i.FFC2}, \eqref{i.FFC3}, and  \eqref{i.FFC4}.
\end{prop}

In view of Proposition~\ref{p.fff} we get the following corollary of Proposition~\ref{p.unicamp}  that ends the proof of Proposition~\ref{p.spawn}.

\begin{coro}\label{c.p.c.splitflipflop}
 For every sufficiently small $\eta>0$,  
  the saddle $q$ and 
 the dynamical blender $(\Ga_\lambda, V, \cC^{\mathrm{uu}}_{\alpha_1},
 \cV^\delta_\eta(\fD_\lambda))$ are in a split flip-flop configuration. 
\end{coro} 

Let us observe  that the property of the configuration being split follows from  the partial hyperbolicity (with one dimensional center direction)
of the initial diffeomorphism $f$. 

\begin{proof}[Proof of Proposition~\ref{p.unicamp}] 
To define the connecting sets $\De^\mathrm{u}$ and $\De^\mathrm{s}$ of the flip-flop configuration
recall the definitions of the 
local invariant manifolds $W^\mathrm{s}_\mathrm{loc}(q)$ and $W^\mathrm{u}_\mathrm{loc}(q)$ in Remark~\ref{r.q}.
Let $\De^\mathrm{u}\eqdef  (A_1^\mathrm{u})^{-1}(J^{u}_0)\times\{(0,q^\mathrm{s})\}$. This set is a disc contained in 
$W^\mathrm{u}_\mathrm{loc}(q)$ that belongs to $\fD_1$ 
and whose negative iterates are contained in $Q_3$. Hence it is disjoint from $V$ (as $V$ is contained in $Q_{1,2}$). 
This proves properties \eqref{i.FFC0p} and \eqref{i.FFC1}.

Consider $F_\lambda^{-1}(W^\mathrm{s}_\mathrm{loc}(q))$. Note that there are points 
$x^\mathrm{u}_1\in  (A_1^\mathrm{u})^{-1}(J^{u}_0)$ and 
$x^\mathrm{u}_1\in  (A_2^\mathrm{u})^{-1}(J^{u}_0)$  such that 
$F_\lambda^{-1}(W^\mathrm{s}_\mathrm{loc}(q))= \Si_1\cup \Si_2$ where 
$$
\Si_1\eqdef \{x^\mathrm{u}_1\}\times g_{\lambda,1}^{-1}([-1,1])\times [-1,1]^s
\quad  \mbox{and} \quad
\Si_2\eqdef \{x^\mathrm{u}_2\}\times g_{\lambda,2}^{-1}([-1,1])\times [-1,1]^s\,.
$$

Consider the intersections $\widetilde \De_1\eqdef \Si_1\cap U_1$ and $\widetilde \De_2\eqdef  \Si_2\cup U_2$ (recall that $U_1\cap U_2$ is 
the domain of the dynamical blender of $F_\lambda$, see Proposition~\ref{p.l.dynamicalinducedblender}).
By definition of a dynamical blender (existence of a safe stable connecting set)
there are a compact set $K$ contained in the interior of  $\widetilde \De_1\cup \widetilde\De_2$ and
 $\varepsilon >0$ so  that $K$
intersects any disc $D\in\fD_1\cup\fD_2$ at point a distance larger that $\varepsilon$ from the boundary $\partial D$. 
Recall the definitions of the return times $n_1,n_2$ and consider the set
$$
\tilde \De^\mathrm{s}\eqdef \bigcup_{i=1,2} \, \Big( \bigcup_{j=0}^{n_i -1} f_\lambda^j( \tilde \De_i )\Big).
$$
The set $\widetilde \De^\mathrm{s}$ is a manifold with boundary and corners contained in the interior of the safety neighborhood $V$, 
transverse to $\cC^{\mathrm{uu}}$, and 
intersects any disc of the strictly $f_\la$-invariant family $\fD_\lambda$  at distance uniformly bounded away from $0$ from its
boundary. 

Furthermore, $f_\lambda^i(\widetilde \De^\mathrm{s})$ is disjoint form $\overline V$ for every $i>n_1+n_2$ and any point of
 $\widetilde \De^\mathrm{s}$ that exits 
from $V$ does not return to $V$. 
In other words, as a set $\tilde \De^\mathrm{s}$ satisfies  conditions \eqref{i.FFC2}, \eqref{i.FFC3}, and \eqref{i.FFC4}. 
We need to remove the corners of $\widetilde \De^\mathrm{s}$. For that it is enough to
consider a sufficiently large submanifold with boundary $\De^\mathrm{s}\subset\tilde \De^\mathrm{s}$, 
in this way we get a stable connecting set 
$\De^\mathrm{s}$ satisfying
conditions \eqref{i.FFC2}, \eqref{i.FFC3}, and \eqref{i.FFC4}. 
This ends
the proof of the proposition.
\end{proof}


\section{An interval of Lyapunov exponents:
Proof of Theorem~\ref{t.interval}}\label{s.interval}

In this section we prove Theorem~\ref{t.interval}.
Consider an open set $\cU \subset \Diff^1(M)$ 
of diffeomorphisms  
with periodic points $p_f, q_f$
(depending continuously on $f$) 
with $\mathrm{u}$-indices $i$ and $i-1$, respectively, which are
in the same chain recurrence class. 
We will prove  the existence of a $C^1$-open and dense subset $\cV$ of $\cU$ 
consisting of diffeomorphisms $f$
such that for
every $\chi \in (\chi_i(q_f),\chi_i(p_f))$ there is
a partially hyperbolic compact set 
$K_{f,\chi}\subset C(p_f,f)$, with $1$-dimensional center direction and positive entropy such that the
center Lyapunov exponent of any point in $K_{f,\chi}$ is $\chi$.

The proof of this result is different for  $\chi$ close to zero and for large $|\chi|$. The case of $\chi$ close to zero is the more interesting one, and requires split flip-flop families. The case when $\chi$ is away from zero follows from simpler and essentially hyperbolic arguments.

\subsection{An interval of small Lyapunov exponents}
The constructions in Section~\ref{s.spawners}
give an open and dense subset $\cU_0$ of $\cU$ such that every $f\in \cU_0$ has 
a saddle $r$ of $\mathrm{u}$-index $i-1$ and a dynamical blender
$\Ga$ of $\mathrm{u}$-index $i$ which are in a split flip-flop configuration.
The following conditions hold:
\begin{itemize}
 \item 
The saddle $r$ and the blender $\Ga$
are contained in a partially hyperbolic set  with 
$(i-1)$-dimensional strong unstable bundle and one-dimensional center.
This set is contained in the chain recurrence class of
$p_f$.
 \item
 There is $\alpha_f>0$ such that the  logarithm of the  center Jacobian is  
less than $-\alpha_f$ on the orbit of $r$ and larger than $\alpha_f$ on the blender.
Moreover, since the flip-flop configuration is robust (Proposition~\ref{p.ffrobust}), the constant $\alpha_f>0$ can be chosen locally constant with $f$.
\end{itemize}

Recall that in the proof of Theorem~\ref{t.practical} (see Section~\ref{s.flipflopyield}) 
we consider a continuous extension $\phi$ of the logarithm of the central Jacobian and
apply Proposition~\ref{p.flipflopconf} to get a flip-flop family with respect to a power $f^N$ and the corresponding 
$N$-th Birkhoff sum $\phi_N$ of $\phi$.
Thereafter using Theorem~\ref{t.FlipFlop_weak} we obtain a compact set $K_f$ with
positive entropy whose $i$-th exponent is zero.

Now for any given $\chi\in [-\alpha_f,\alpha_f]$, we apply Proposition~\ref{p.flipflopconf} to the function $\phi+\chi$ instead, thus obtaining a flip-flop family with respect to a Birkhoff sum of this function. So Theorem~\ref{t.FlipFlop_weak} provides a compact set $K_{f,\chi}$ with positive entropy  whose $i$-th exponent is  $\chi$.
This completes the proof of the part the theorem about Lyapunov exponents close to zero.

\subsection{Intervals of large Lyapunov exponents}
We now fix $f$ as above and the constant $\alpha_f=\alpha$. The proof below follows exactly as the one of 
\cite[Theorem~1]{ABCDW}, thus we just explain the main steps.

By hypotheses, the saddles
$p_f$ and $q_f$ are involved in a robust cycle. 
Recalling the sketch of the proof of Proposition~\ref{p.caixapreta} in Section~\ref{ss.cyclestospawners},
after an arbitrarily small $C^1$-perturbation 
we can assume that the homoclinic classes
of these saddles are both non-trivial
and, after replacing $p_f$ and $q_f$ by some point of the class, 
that the eigenvalues corresponding to $p_f$ and $q_f$
are all positive and different.

The occurrence of a robust cycle  implies that after
a new arbitrarily small $C^1$-perturbation we can obtain a heterodimensional cycle 
associated to $p_g$ and $q_g$. The unfolding of this cycle generates a saddle $\bar p_g$ 
homoclinically related to $p_g$ (thus of $\mathrm{u}$-index $i$) and such that
$\chi_i(\bar p_g)$ is close to $0$,  in particular,
$\chi_i(\bar p_g)\in (0,\frac{\alpha}2)$. This construction  implies that $\chi_i(\bar p_g)>\chi_{i+1}(\bar p_g)$.
Note that this configuration is robust.

Since the initial cycle associated to $p_f$ and $q_f$ is robust
and $g$ is close to $f$, the saddles
$p_g$ and $q_g$ also belong to hyperbolic sets involved
in a robust cycle. Arguing as above, a perturbation generates a new heterodimensional cycle 
that yields a saddle $\bar q_h$ homoclinically related to $q_h$ (thus of $\mathrm{u}$-index $i-1$) 
with $\chi_i(\bar q_h)\in (-\frac{\alpha}2,0)$.
Moreover, we have that $\chi_i(\bar q_g)>
\chi_{i+1}(\bar q_g)$. 

As $p_h$ and $\bar p_h$ are homoclinically related,
there is a hyperbolic basic set $\Lambda_h$ containing 
$p_h$ and $\bar p_h$.  
Since $\chi_i(p_h)> 
\chi_{i+1}(p_h)$ and $\chi_i(\bar p_h)>
\chi_{i+1}(\bar p_h)$.
Taking care that the intersection of the parts of the invariant manifolds
of $p_h$ and $\bar p_h$ involved in the construction of the set $\Lambda_h$ to be in general position, 
we can assume that the unstable bundle $E^\mathrm{u}$ of $\Lambda_h$ has a dominated splitting
$E^\mathrm{u}=E^\mathrm{uu} \oplus E^{\mathrm{cu}}$, where $E^{\mathrm{cu}}$ is one-dimensional.
This implies  that the logarithm of the center Jacobian is well defined and continuous along the 
strong unstable manifold of $\Lambda_h$. 

A standard  argument involving Markov partitions  implies that,
for every $\chi\in (\chi_i(\bar p_h),\chi_i(p_h))$
there is an invariant compact set $ K_{h,\chi}$ contained in $\Lambda_h$ having  positive entropy
and such  that the $i$-th Lyapunov exponent of every point in $ K_{h,\chi}$  is $\chi$. 
Here the use of  Markov partitions substitutes the flip-flop-like arguments.

Fixed now
$\chi\in (\chi_i (q_h),\chi_i(\bar q_h))$, 
the same construction as above with $q_h$ and $\bar q_h$ provides invariant compact subset $K_{h,\chi}$
of the chain recurrence class of $q_h$ with
positive entropy  that
consists of points whose $i$-th Lyapunov exponent
is  $\chi$. 

Since by construction 
$$
\big(\chi_i(q_h),\chi_i(\bar q_h)\big)\cup \big[-\alpha,\alpha\big] \cup \big(\chi(\bar p_h),
\chi_i(p_h)\big)= \big(\chi_i(q_h),\chi_i(p_h)\big),
$$
we get the announced family of compact sets $K_{h,\chi}$, $\chi\in \big( \chi_i(q_h), \chi_i(p_h) \big)$. 
This ends the proof of Theorem~\ref{t.interval}.


\end{document}